\documentclass[12pt,a4paper,reqno]{amsart}
\usepackage{amssymb}
\usepackage{latexsym}
\usepackage{amsmath}
\usepackage{mathrsfs}
\usepackage{cite}

\usepackage{calc}

\newcommand{\scal}[2]{\langle #1,#2\rangle}
\newcommand{\rr}[1]{\mathbf R^{#1}}
\newcommand{\cc}[1]{\mathbf C^{#1}}
\newcommand{\nm}[2]{\Vert #1\Vert _{#2}}
\newcommand{\nmm}[1]{\Vert #1\Vert }

\newcommand{\op}{\operatorname{Op}}
\newcommand{\sets}[2]{\{ \, #1\, ;\, #2\, \} }
\newcommand{\ep}{\varepsilon}
\newcommand{\fy}{\varphi}
\newcommand{\cdo}{\, \cdot \, }

\newcommand{\wpr}{{\text{\footnotesize $\#$}}}

\newcommand{\essup}{\operatorname{ess\, sup}}

\newcommand{\eabs}[1]{\langle #1\rangle}     %%%%%   for <x>

\newcommand{\tp}{\operatorname{Tp}}
\newcommand{\tw}{\operatorname{tw}}
\newcommand{\vrum}{\vspace{0.2cm}}
\newcommand{\im}{i}
\newcommand{\mascB}{\mathscr B}
\newcommand{\mascP}{\mathscr P}
\newcommand{\maclS}{\mathcal S}

\numberwithin{equation}{section}          %Detta goer att man faar
\newtheorem{thm}{Theorem}
\numberwithin{thm}{section}
\newtheorem{prop}[thm]{Proposition}
\newtheorem{cor}[thm]{Corollary}
\newtheorem{lemma}[thm]{Lemma}

\newcommand{\rubrik}{}

\theoremstyle{definition}

\newtheorem{defn}[thm]{Definition}
\newtheorem{example}[thm]{Example}

\theoremstyle{remark}
\newtheorem{rem}[thm]{Remark}

\newcommand{\rd}{\mathbf{R} ^{d}}

\newcommand{\intrd}{\int _{\rd }}

\title{The Bargmann transform on a broad family of Banach spaces,
with applications to Toeplitz and pseudo-differential operators}

\author{Joachim Toft}

\address{Department of Computer science, Physics and Mathematics,
Linn{\ae}us University, V{\"a}xj{\"o}, Sweden}

\email{joachim.toft@lnu.se}

\begin{document}

\begin{abstract}
We investigate mapping properties for the Bargmann transform on modulation spaces whose 
weights and their reciprocals are allowed to grow faster than exponentials. We prove
that this transform is isometric and bijective from modulation spaces
to convenient Lebesgue spaces of analytic functions. We use this to prove that such modulation spaces fulfill most of the continuity properties which are well-known when the weights are moderated. Finally we use the results to establish continuity properties of Toeplitz and pseudo-differential operators in the context of these modulation spaces.
\end{abstract}

\maketitle

\par

%%%%%%%%%%%%%%%%%%%%%%%%%%%%
\section{Introduction}
%%%%%%%%%%%%%%%%%%%%%%%%%%%%

\par

In this paper we introduce and establish basic continuity properties for a broad
family of (quasi-)Banach spaces of functions and distributions of Gelfand-Shilov
types, in the framework of harmonic analysis. We establish close links between
these spaces and
(weighted) Lebesgue spaces $A^{p,q}_{(\omega )}$ of analytic functions
related to Bargmann-Fock spaces. The family of spaces consists of modulation
spaces, where each modulation space is obtained by imposing a weighted mixed
norm estimate on the short-time Fourier transforms
of the involved distributions. Important cases of such spaces are
given by $M^{p,q}_{(\omega )}$, where the weighted mixed norm estimate
is constituted by the $L^{p,q}_{(\omega)}$ norm.

\par

Among the involved parameters $p$, $q$ and the weight $\omega$,
it follows that $\omega$ is most important concerning
imposing regularity, or relaxing growth,
oscillations and singularity conditions on the involved distributions.
In comparison to already established theories of such spaces
(cf. \cite{FGW,Gc1,SiT2} and the references therein) the conditions for
the involved weight functions are significantly relaxed in the present paper.
This leads to that our family of modulation spaces are significantly larger compared
to the "usual" families of such spaces. For example, for each fixed $s >1/2$, the
modulation space $M^{p,q}_{(\omega )}$ can be made "arbitrary close" to the
Gelfand-Shilov space $\maclS _s$ or to $\maclS _s'$, by choosing the weight
$\omega$ in appropriate ways.

\par

%Among the parameters, the choice of weight $\omega$ has in
%general the largest impact when imposing regularity, or relaxing growth,
%oscillations and singularity conditions on the involved distributions.
%In comparison to already established theories of such spaces (cf. \cite{FGW,Gc1,SiT2} and the
%references therein) the conditions for the involved weight functions are significantly
%relaxed. This leads to that our family of modulation spaces are significantly larger compared
%to the "usual" families of such spaces. For example, for each fixed $s >1/2$, the
%modulation space $M^{p,q}_{(\omega )}$ can be made "arbitrary close" to the
%Gelfand-Shilov space $\maclS _s$ or to $\maclS _s'$, by choosing the weight
%$\omega$ in appropriate ways.
%
%\par

%In this paper we introduce and establish basic continuity properties for a family
%of (quasi-)Banach spaces, defined in 
%
%general types of 
%modulation spaces $M^{p,q}_{(\omega )}$, and for canonical (weighted) Lebesgue spaces
%$A^{p,q}_{(\omega )}$ of analytic functions related to Bargmann-Fock spaces.
%In comparison to already established theories of such spaces (cf. \cite{FGW,Gc1,SiT2} and the
%references therein) the conditions for the involved weight funcitons are significantly
%relaxed. This leads to that our families of spaces are significantly larger compared
%to the "usual" families of such spaces. For example, for each fixed $s >1/2$, the
%modulation space $M^{p,q}_{(\omega )}$ can be made arbitrary close to the
%Gelfand-Shilov space $\maclS _s$ or to $\maclS _s'$, by choosing the weight
%$\omega$ in appropriate ways.
%
%\par

\par

An essential part of our investigations concerns the establishment of the
links between the modulation spaces and the $A^{p,q}_{(\omega)}$
spaces, by proving that the Bargmann transform is isometric and bijective
between these spaces. One of the benefits is that any
property valid for the $A^{p,q}_{(\omega )}$ spaces, carry over to the
modulation spaces, and vise versa. For example, we prove that any 
modulation space is a Banach or quasi-Banach space,
and that convenient density, duality and interpolation properties hold
for such spaces, because similar properties are valid for corresponding
spaces of analytic functions.

\par

Finally we use our results to extend the theory of pseudo-differential operators to involve more extreme symbols and target distributions comparing to earlier investigations.

\medspace

We recall that the (classical)
modulation space $M^{p,q}_{(\omega )}$, $p,q \in [1,\infty]$, as introduced and carefully
investigated by
Feichtinger and Gr{\"o}chenig in \cite{F1,FG1,FG2,Feichtinger5,Gc1}, consists of all
tempered distributions whose short-time Fourier transforms (STFT) have finite mixed
$L^{p,q}_{(\omega)}$ norm. Here the weight $\omega$ quantifies the degree
of asymptotic decay and singularity of the distribution in
$M^{p,q}_{(\omega )}$. In general it is assumed that $\omega$ should be \emph{moderate}, which imposes
several properties on $\omega$ and thereby on the modulation space $M^{p,q}_{(\omega )}$.
(See Sections \ref{sec1} and \ref {sec2} for strict definitions.)
For example, the moderate property implies that $\omega$ is not allowed
to grow or decay faster than exponentials, that $M^{p,q}_{(\omega )}$ are
invariant (but not norm invariant) under pullbacks of translations, and that
several properties valid for weighted Lebesgue spaces (e.{\,}g. density,
duality and interpolation properties) carry over to classical modulation spaces. 

\par

A major idea behind the design of these spaces was to find
useful Banach spaces, which are defined in a way similar to Besov
spaces, in the sense of replacing the dyadic decomposition on the
Fourier transform side, characteristic to Besov spaces, with a
\textit{uniform} decomposition. From the construction of these
spaces, it turns out that modulation spaces and Besov spaces in some
sense are rather similar, and sharp embeddings between these
spaces can be found in \cite{Toft2, To04B}, which
are improvements of certain embeddings in \cite {Grobner}. (See
also \cite {Sugimoto1,WH} for verification of the sharpness.)

\par

During the last 15 years, several results have been proved which confirm
the usefulness of the modulation
spaces in time-frequency analysis, where they occur naturally. For
example, in \cite{Feichtinger5,Gc2,GL}, it
is shown that all modulation spaces admit
reconstructible sequence space representations using Gabor frames.

\par

Parallel to this development, modulation spaces have been incorporated
into the calculus of pseudo-differential operators, which also involve
Toeplitz operators. (See
e.{\,}g. \cite{Gc2, GT1, GT2, He1, HTW, Sugimoto1, Toft2, To04B, To5, To8A, To8B} and the
references therein concerning symbol classes embedded in $\mathscr S'$,
and \cite{CPRT10, Gc2, PT1, PT2, PT3, Te1, Te2} for results
involving ultra-distributions. Here and in what follows we use the usual notations
for the usual function and distribution spaces, see e.{\,}g. \cite{Ho1}.)

\medspace

The Bargmann transform can easily be reformulated in terms of the
short-time Fourier transform, with a particular Gauss function as
window function. By reformulating the Bargmann transform in such way,
and using the fundamental role of the short-time Fourier
transform in the definition of modulation spaces, it easily follows
that the Bargmann transform is continuous and injective from
$M^{p,q}_{(\omega )}$ to $A^{p,q}_{(\omega)}$. Furthermore, by
choosing the window function as a particular Gaussian function in the
$M^{p,q}_{(\omega )}$ norm, it follows that $\mathfrak V\, :\,
M^{p,q}_{(\omega)}\to A^{p,q}_{(\omega)}$ is isometric.

\par

These facts and several other mapping properties for the Bargmann
transform on (classical) modulation spaces
were established in \cite{FG1, FGW, Gc1, GW, SiT2}. 
In fact, here it is proved that the Bargmann transform from
$M^{p,q}_{(\omega )}$ to
$A^{p,q}_{(\omega )}$ is not only injective, but in fact
\emph{bijective}. 

\medspace

For the modulation space $M^{p,q}_{(\omega )}$, the weight function $\omega$ is
important for imposing or relaxing conditions on the distributions $f$
in $M^{p,q}_{(\omega )}$.
More precisely, the weight $\omega =\omega (x,\xi )$ depends on both the space (or time) 
variables $x$ as well as the momentum (or frequency) variables $\xi$. Roughly speaking, the weight 
function posses (cf.  \cite{F1,CJT2,Gc2,GT1,GT2}):
\begin{itemize}
\item $\omega$ tending rapidly to infinity as $x$ tends to infinity,
imposes that $f$ tends rapidly to zero at infinity;

\vrum

\item $\omega$ tending rapidly to zero as $x$ tends to infinity,
relaxes the growth conditions on $f$ at infinity;

\vrum

\item $\omega$ tending rapidly to infinity as $\xi $ tends to infinity,
imposes high regularity for $f$;

\vrum

\item $\omega$ tending rapidly to zero as $\xi $ tends to infinity,
relaxes the conditions on singularities of $f$;

\vrum

\item $\omega$ tending rapidly to infinity as both $x$ and $\xi $
tends to infinity, imposes stronger restrictions on oscillations for $f$ at infinity;

\vrum

\item $\omega$ tending rapidly to zero as both $x$ and $\xi $
tends to infinity, relax the restrictions on oscillations for $f$ at infinity.
\end{itemize}

\par

The condition that $\omega$ should be moderate implies that
\begin{equation}\label{conseqmoder}
\omega +1/\omega \le v
\end{equation}
for some $v=Ce^{c|\cdo |}$, where $c,C>0$ are constants. In this case, $\omega$ is called a weight of
\emph{exponential type}. We remark that corresponding modulation spaces
$M^{p,q}_{(\omega)}$ are subsets of appropriate spaces of Gelfand-Shilov distributions, and
for certain choices of $\omega$ we may have that  $M^{p,q}_{(\omega )}$ is contained
in $\mathscr S$, or that $\mathscr S'$ is contained in  $M^{p,q}_{(\omega )}$.
A more restrictive case appears when \eqref{conseqmoder}
is true for some $v=Ce^{c|\cdo |^s}$, with $0\le s<1$. In this case, $\omega$ is
called a weight of \emph{subexponential type}. If instead $v$ in \eqref{conseqmoder}
can be chosen as polynomial, then $\omega$
is said to be of \emph{polynomial type}. In this case, $M^{p,q}_{(\omega )}$ contains
$\mathscr S$, and is contained in $\mathscr S'$.

\par

Several properties for the modulation spaces might be violated when passing
from the subexponential type weights into exponential type weights. For example,
if $\omega$ is of exponential type, then $M^{p,q}_{(\omega )}$ might be contained
in the set of real analytic functions, which in particular
implies that there are no non-trivial compactly supported elements in $M^{p,q}_
{(\omega )}$.
%Consequently, there are no compactly supported Gabor atoms, implying
%that certain parts of the time-frequency machinery break  down because compactly
%supported Gabor atoms are in some situations fundamental in the applications.
Consequently, there are no compactly supported Gabor atoms, implying
the time-frequency machinery breaks  in those parts were compactly
supported Gabor atoms are needed.

\par

In the present paper we go beyond these situations and relax the assumptions on $v$
even more. For example, we permit $v$ in \eqref{conseqmoder} to be superexponential, i.{\,}e. 
$v=Ce^{c|\cdo |^\gamma}$, where $1<\gamma <2$. In this situation, almost no arguments 
in classical modulation space theory can be used, because the main results in that theory are based 
on the fact that $\omega$ should be moderate. This condition is violated when $v$ in
\eqref{conseqmoder} has to be superexponential.

\par

In Sections \ref{sec1} and \ref{sec2} we give the explicit conditions on the weight functions, 
and in Sections \ref{sec3} and \ref{sec4} we prove:
\begin{enumerate}
\item any extended weight class contains all weights in classical modulation space theory, 
including weights which are moderated by exponential type weights. Furthermore, any 
superexponential weight with $\gamma$ above less than $2$ are included, as well as 
weights of the form $\omega =\eabs \cdo ^{\eabs \cdo}$ and $\omega = \Gamma (\eabs \cdo 
+1)$. Here $\eabs x=(1+|x|^2)^{1/2}$ and $\Gamma$ is the Gamma function.

\vrum

\item $A^{p,q}_{(\omega )}$ and $M^{p,q}_{(\omega )}$ are Banach spaces and fulfill convenient density, duality and interpolation properties.

\vrum

\item the Bargmann transform is isometric and bijective from $M^{p,q}_{(\omega )}$ to $A^{p,q}_
{(\omega )}$.
\end{enumerate}

\par

In the last section we establish new forms of pseudo-differential calculi in the framework of these 
modulation spaces. This means that the spans of the spaces for operator symbols, target functions 
and image functions, are significantly larger comparing to earlier theories. Therefore, these spaces 
may be smaller as well as larger comparing the usual situations. The approach here is similar to \cite
{To5, To8A, To9}, where similar results were obtained in background of classical modulation space 
theory. The results here are, to some extent, also related to the results in
\cite{Cap,PT1,PT2, PT3, Rod, Te1, Te2}, when $v$ in \eqref{conseqmoder} is bounded by a  
subexponential function.

\par

We remark that in contrast to classical theory of pseudo-differential operators,
(cf. e.{\,}g. \cite{Ho1}), there are no explicit regularity assumptions on the symbols.
On the other hand, if $1<\gamma _1<\gamma <\gamma _2 <2$ with $c>0$,
and the weight $\omega$ is given by
\begin{equation}\label{subsupexp}
\omega (x,\xi )=e^{c(|x|^{\gamma}+|\xi |^\gamma )},
\end{equation}
then the corresponding modulation spaces are contained in the Gelfand-Shilov space
$\mathcal S_{1/\gamma _1}$, and contain $\mathcal S_{1/\gamma _2}$. In particular, 
this means that the involved functions and their derivatives are extendable to
entire analytic functions and fulfill estimates of the form
$$
|f(x)|\le Ce^{-c|x|^{\gamma _1}},\quad \text{and}\quad |\widehat f(\xi )|\le Ce^{-c|\xi |^{\gamma _1}},
$$
for some positive constants $c$ and $C$. It is therefore obvious that in this situation,
the elements in these  modulation spaces posses strong regularity properties.

\par

On the other hand, if $c<0$ in \eqref{subsupexp}, then the corresponding
modulation spaces contain  the dual $\mathcal S_{1/\gamma _1}'$ of
$\mathcal S_{1/\gamma _1}$, which in turn is significantly larger than
e.{\,}g. $\mathscr S'$, the space of tempered distributions.

\par

Finally, in Section \ref{sec5} we apply the continuity results for modulation
spaces to establish  continuity properties for Toeplitz operators with symbols
in weighted mixed norm space of Lebesgue types. (Cf. \cite
{Bau, Berezin71, Bog1, Cob01, CG03, To8B} and the references
therein for similar and related investigations.)

\par

%%%%%%%%%%%%%%%%%%%%%%%
\section*{Acknowledgement}
%%%%%%%%%%%%%%%%%%%%%%%

\par

First I would like to express my gratitude to
M. Signahl for interesting and valuable discussions on questions
related to certain parts of the present paper during 2009 and 2010. I also thank
K. Gr{\"o}chenig for valuable discussions.
In fact, important ideas to the paper appeared when I communicated
with him during the spring 2010. I am also grateful to A. Galbis, C. Fernandez,
P. Wahlberg and S. Pilipovi{\'c} for supports and careful reading of the original paper,
leading to several improvements of the style and content. 

\par

%%%%%%%%%%%%%%%%%%%%%%%
\section{Preliminaries}\label{sec1}
%%%%%%%%%%%%%%%%%%%%%%%

\par

In this section we give some definitions and recall some basic
facts. The proofs are in general omitted. In the first part we consider
appropriate cnditions on the involvd weight functions. Thereafter
we review some facts for Gelfand-Shilov spaces. Then we discuss
basic properties of the short-time Fourier transform, which is thereafter
used in the definition of modulation spaces, and obtaining basic
properties for such spaces. The last part of the section is devoted
to the Bargmann transform and appropriate Banach spaces of
entire functions, which are appropriate in the background of the
Bargmann transform.

\par

\subsection{Weight functions}\label{subsec1.1}

\par

We start by discussing general properties on the involved weight
functions. A \emph{weight} on $\rr d$ is a positive function $\omega$
on $\rr d$ such that $\omega \in 
L^\infty _{loc}(\rr d)$, and for each compact set $K\subseteq \rr d$, there
is a constant $c>0$ such that
$$
\omega (x)\ge c\qquad \text{when}\qquad x\in K.
$$
A usual condition on $\omega$ is that it should be \emph{$v$-moderate}
for some positive function $v \in L^\infty _{loc}(\rr d)$. This means that
\begin{equation}\label{moderate}
\omega (x+y) \leq C\omega (x)v(y),\qquad x,y\in \rr d,
\end{equation}
for some constant $C$ which is independent of $x,y\in \rr d$. We note
that \eqref{moderate} implies that $\omega$ fulfills the estimates
\begin{equation*}%\label{moderateconseq}
C^{-1}v(-x)^{-1}\le \omega (x)\le Cv(x).
\end{equation*}

\par

We say that $v$ is
\emph{submultiplicative} when \eqref{moderate} holds with $\omega =v$.
In the sequel, $v$ and $v_j$ for $j\ge 0$, always stand for
submultiplicative weights if nothing else is stated.

\par

The weight $\omega$ is called a weight of \emph{exponential type}, if $v$
in \eqref{moderate} can be chosen as $v(x)=Ce^{c|x|}$ for some $c,C>0$.
If, more restrictive, $v$ can be chosen as a polynomial, then $\omega$
is called a weight of \emph{polynomial type}. We let $\mascP (\rr d)$ and
$\mascP _E(\rr d)$ be the sets
of all weights on $\rr d$ of polynomial type and exponential type, respectively. Obviously,
$\mascP (\rr d)\subseteq \mascP _E(\rr d)$.

\par

A broader class of moderate weights comparing to $\mascP (\rr d)$ is obtained by replacing
the polynomial assumption on $v$ by the so called GRS condition 
(Gelfand-Raikov-Shilov condition). That is, $v\in L^\infty _{loc}(\rr d)$ is positive and satisfies
$$
\lim _{n\to \infty}\frac {\log v(nx)}{n}=0.
$$
An important class of submultiplicative weights which fulfills the GRS condition is
the so called subexponential weights, i.{\,}e. weights of the form
\begin{equation}\label{subexpdef}
\omega (x)=Ce^{c|x|^{s}},
\end{equation}
when $\omega =v$, and $c$, $C$ and $s$ are positive constants such that $s<1$.
On the other hand, if $v$ is a weight of 
exponential type, then the GRS condition is violated. Furthermore,
if $\omega$ is $v$-moderate for some $v$, then it is moderated by an 
exponential type weight.
Consequently, the $\mascP _E(\rr d)$ contains \emph{all} weights on $\rr d$
which are moderated by some functions, including those weights
moderated by $v$ which fulfills the GRS-conditions. We refer to \cite{Gc2.5} and the
references therein for these facts.

\medspace

In this paper we permit weights where the moderate condition \eqref{moderate} on $\omega $ has been relaxed by appropriate local and global conditions. In most of the situations, the local condition is
\begin{equation}\label{modrelax}
C^{-1}\omega (x)\le \omega (x+y)\le C\omega (x) \quad \text{when}\quad Rc \le |x|\le c/|y| ,\quad R\ge 2,
\end{equation}
for some positive constants $c$ and $C$. However, in most of the situations, the condition
\eqref{modrelax} is relaxed into
\begin{equation}\tag*{(\ref{modrelax})$'$}
C^{-1}\omega (x)^2\le \omega (x+y)\omega (x-y)\le C\omega (x)^2 \quad \text{when}\quad Rc \le |x|\le c/|y| ,\quad R\ge 2,
\end{equation}
for some positive constants $c$ and $C$.

\par

Important examples of weights satisfying \eqref{modrelax} are those which satisfy \eqref
{subexpdef}, when $C$ and $s$ being positive such that $s\le 2$, and $c\in \mathbf R$. 
Especially we note that if $\omega$ is given by \eqref{subexpdef} with $1<s\le 2$, then $
\omega$ is \emph{not} moderated by any weight $v$, but satisfies \eqref{modrelax} for some 
choices of $c>0$ and $C>0$. On the other hand, if $\omega >0$ and satisfies \eqref
{modrelax}, then Proposition \ref{PGsmoothness} in Section \ref{sec2} shows that
\begin{equation}\label{Gaussest}
C^{-1}e^{-c|x|^2}\le \omega (x)\le Ce^{c|x|^2},
\end{equation}
holds for some positive constants $c$ and $C$.

\par

\begin{defn}\label{defweights}
Let $\omega$ be a weight on $\rr d$. 
\begin{enumerate}
\item $\omega$ is called a \emph{weight of Gaussian type} (\emph{weakly Gaussian type}) on $\rr d$, if \eqref{modrelax} holds (if \eqref{modrelax}$'$ holds) for some positive $c$ and $C$, and \eqref{Gaussest} holds for some positive $c$ and $C$. The set of Gaussian type and weakly Gaussian type weights on $\rr d$ are denoted by $\mascP _G (\rr d)$ and $\mascP  _{Q}(\rr d)$, respectively;

\vrum

\item $\omega$ is called a \emph{weight of subgaussian type} (\emph{weakly
subgaussian type}) on $\rr d$, if \eqref{modrelax} holds (if \eqref{modrelax}$'$
holds) for some positive $c$ and $C$, and for every $c>0$, there is a
constant $C>0$ such that \eqref{Gaussest} holds. The set of subgaussian
type and weakly subgaussian type weights on $\rr d$ are denoted by
$\mascP  _{G}^0(\rr d)$ and $\mascP  _{Q}^0(\rr d)$, respectively.
\end{enumerate}
\end{defn}

\par

We note that each one of the families of weight functions in Definition \ref{defweights}
are groups under the ordinary multiplications.

\par

\begin{rem}\label{extentionPQ0}
The family $\mascP _Q^0$ is larger than $\mascP _G^0$, but its definition is
somewhat more complicated. An important reason for introducing this family is that we
may prove that the general modulation spaces, introduced later on, can be made close to
Gelfand-Shilov spaces in the sense of Proposition \ref{capcupmodsp} in Section \ref{sec3}.
So far we are unable
to prove any similar result when the family  $\mascP _Q^0$ is replaced by $\mascP _G^0$.

\par

On the other hand, for any weight in $\mascP _G^0$ one may find an equivalent smooth
weight (cf. Proposition \ref{PGsmoothness} in Section \ref{sec2}). So far we are unable to
extend this property to all weights in $\mascP _Q^0$.

\par

We note that if $\omega \in \mascP _Q^0(\rr d)$, then $\omega$ satisfies
the following conditions:
\begin{enumerate}
\item there are invertible $d\times d$-matrices $T_1,\dots ,T_N$ whose norms are at
most one, i.{\,}e. $\nmm {T_j}\le 1$, $j=1,\dots ,N$, and such that
\begin{multline}\tag*{(\ref{modrelax})$''$}
C^{-1}\omega (x)^N\le \prod _{j=1}^N\omega (x+T_jy)\le C\omega (x)^N
\\[1ex]
\text{when}\quad Rc \le |x|\le c/|y| ,\quad R\ge 2,
\end{multline}
for some positive constants $c$ and $C$;

\vrum

\item for every $c>0$, there is a constant $C>0$ such that \eqref{Gaussest} holds.
\end{enumerate}

\par

Hence if we modify the definition of $\mascP _Q^0(\rr d)$ in such way
that it should contain all weights $\omega$ satisfying (1) and (2), then
we obtain a larger family of weights, comparing to Definition
\ref{defweights}. By straight-forward computations it follows
that all results in the paper are true after the definition of $\mascP _Q^0$
in Definition \ref{defweights} has been modified in this way.

\par

A special situation appears
for Proposition \ref{capcupmodsp} in Section \ref{sec3}, where the symmetry condition in
$y$ in \eqref{modrelax}$'$ is essential for its proof. However, it follows that
Proposition \ref{capcupmodsp} is true, after \eqref{modrelax}$'$ in the definition of
$\mascP _Q^0$ has been replaced by
$$
C^{-1}\omega (x)^{2N}\le \prod _{j=1}^N\omega (x+T_jy)\omega (x-T_jy)\le C\omega (x)^{2N},
$$
when $Rc \le |x|\le c/|y|$ and $R\ge 2$, where $T_j$ are invertible
matrices with norm at most one.
\end{rem}

\par

In Section \ref{sec2} we introduce other convenient subfamilies of $\mascP _Q(\rr d)$.

\par

\begin{example}\label{weightexampl1}
Let $c,s\in \mathbf R$, $C>0$ and $t>1/2$. Then
\begin{equation}\label{defsigmas}
\sigma _s(x)\equiv \eabs {x}^s = (1+|x|^2)^{s/2},\quad \omega _1(x)=e^{c|x|^{1/t}}\quad
\text{and}\quad \omega _2(x)=e^{c|x|^{2}},
\end{equation}
are weights of polynomial type, subgaussian type and Gaussian type, respectively.
\end{example}

\par

\begin{defn}\label{defweightfam}
Let $\Omega \subseteq \mascP _Q(\rr d)$.
%be such that for each $\omega \in \Omega$, there are constants $c>0$
%and $C>0$ such that \eqref{Gaussest} holds.
Then $\Omega$ is called an \emph{admissible family of weights}, if
%$\Omega \subseteq  \mascP _Q(\rr d)$, and
there is a rotation invariant
function $0<\omega _0(x)\in L^\infty _{loc}(\rr d)\cap L^1(\rr d)$ which
decreases with $|x|$ and such that
$$
\omega \cdot \omega _0\in \Omega \quad \text{and}\quad
\omega / \omega _0\in \Omega \quad \text{when}\ \omega \in \Omega .
$$
\end{defn}

\par

We note that for some choice of $\omega _0$ in Definition \ref{defweightfam} we have
\begin{equation}\label{omega0est}
\omega _0(x)\le C\eabs x^{-d}
\end{equation}
for some constant $C>0$.

\par

\begin{example}\label{exadmweights}
Every family in Definition \ref{defweights} are admissible. Moreover, if $\omega _0\in \mascP _{Q}(\rr d)$ and $\Omega$ is a family of admissible weights, then
\begin{enumerate}
\item $\sets {\sigma _N}{N\in \mathbf Z}$ is admissible;

\vrum

\item $\omega _0\cdot \Omega \equiv \sets {\omega _0\omega}{\omega \in \Omega}$ is admissible.
\end{enumerate}
\end{example}

\par

%For each $\omega \in \mascP (\rr d)$ and $p\in [1,\infty ]$, we let
%$L^p_{(\omega )}(\rr d)$ be the Banach space which consists of all
%$f\in L^1_{loc}(\rr d)$ such that $\nm {f}{L^p_{(\omega )}}\equiv \nm
%{f\, \omega }{L^p}$ is finite.
%
%\par

\subsection{Gelfand-Shilov spaces}

\par

Next we recall the definition of Gelfand-Shilov spaces.

\par

Let $0<h,s\in \mathbf R$ be fixed. Then we let $\mathcal S_{s,h}(\rr d)$ be the set of all $f\in C^\infty (\rr d)$ such that
\begin{equation*}%\label{gfseminorm}
\nm f{\mathcal S_{s,h}}\equiv \sup \frac {|x^\beta \partial ^\alpha
f(x)|}{h^{|\alpha | + |\beta |}(\alpha !\, \beta !)^s}
\end{equation*}
is finite. Here the supremum should be taken over all $\alpha ,\beta \in
\mathbf N^d$ and $x\in \rr d$. Obviously $\mathcal S_{s,h}\subseteq
\mathscr S$ is a Banach space which increases with $h$ and $s$.
Furthermore, if $s>1/2$ or $s=1/2$ and $h\ge 1$, then $\mathcal
S_{s,h}$ contains all finite linear combinations of Hermite functions.
Since such linear combinations are dense in $\mathscr S$, it follows
that the dual $\mathcal S_{s,h}'(\rr d)$ of $\mathcal S_{s,h}(\rr d)$ is
a Banach space which contains $\mathscr S'(\rr d)$.

\par

The \emph{Gelfand-Shilov spaces} $\mathcal S_{s}(\rr d)$ and
$\Sigma _{s}(\rr d)$ are the inductive and projective limit respectively
of $\mathcal S_{s,h}(\rr d)$. This implies that
\begin{equation}\label{GSspacecond1}
\mathcal S_{s}(\rr d) = \bigcup _{h>0}\mathcal S_{s,h}(\rr d)
\quad \text{and}\quad \Sigma _{s}(\rr d) =\bigcap _{h>0}\mathcal S_{s,h}(\rr d),
\end{equation}
and that the topology for $\mathcal S_{s}(\rr d)$ is the strongest possible one such that each inclusion map from $\mathcal S_{s,h}(\rr d)$ to $\mathcal S_{s}(\rr d)$ is continuous. The space $\Sigma _s(\rr d)$ is a Fr{\'e}chet space with semi norms $\nm \cdo{\mathcal S_{s,h}}$, $h>0$. 

\par

We remark that the space $\mathcal S_s(\rr d)$ is zero when $s<1/2$, and that $\Sigma _s(\rr d)$ is zero when $s\le 1/2$. Furthermore, for each $\ep >0$ and $s\ge 1/2$ we have
$$
\Sigma _s (\rr d)\subseteq \mathcal S_s(\rr d)\subseteq \Sigma _{s+\ep}(\rr d).
$$
On the other hand, in \cite{pil} there is an alternative elegant definition of $\Sigma _{s_1}(\rr d)$ and $\mathcal S _{s_2}(\rr d)$ such that these spaces agrees with the definitions above when $s_1>1/2$ and $s_2\ge 1/2$, but $\Sigma _{1/2}(\rr d)$ is non-trivial and contained in $\mathcal S_{1/2}(\rr d)$.

\par

From now on we assume that $s>1/2$ when considering $\Sigma _s(\rr d)$.

\par

\medspace

The \emph{Gelfand-Shilov distribution spaces} $\mathcal S_{s}'(\rr d)$ and $\Sigma _{s}'(\rr d)$ are the projective and inductive limit respectively of $\mathcal S_{s,h}'(\rr d)$.  This means that
\begin{equation}\tag*{(\ref{GSspacecond1})$'$}
\mathcal S_{s}'(\rr d) = \bigcap _{h>0}\mathcal S_{s,h}'(\rr d)\quad
\text{and}\quad \Sigma _{s}'(\rr d) =\bigcup _{h>0}\mathcal S_{s,h}'(\rr d).
\end{equation}
We remark that already in \cite{GS} it is proved that $\mathcal S_{s}'(\rr d)$ is the dual of $\mathcal S_{s}(\rr d)$, and if $s>1/2$, then $\Sigma _{s}'(\rr d)$ is the dual of $\Sigma _{s}(\rr d)$ (also in topological sense).

\par

The Gelfand-Shilov spaces are invariant under several basic transformations. For example they are invariant under translations, dilations, tensor products and under any Fourier transformation.

\par

From now on we let $\mathscr F$ be the Fourier transform which takes the form
$$
(\mathscr Ff)(\xi )= \widehat f(\xi ) \equiv (2\pi )^{-d/2}\int _{\rr
{d}} f(x)e^{-i\scal  x\xi }\, dx
$$
when $f\in L^1(\rr d)$. Here $\scal \cdo \cdo$ denotes the usual scalar product on $\rr d$. The map $\mathscr F$ extends 
uniquely to homeomorphisms on $\mathscr S'(\rr d)$, $\mathcal S_s'(\rr d)$ and $\Sigma _s'(\rr d)$, and restricts to 
homeomorphisms on $\mathscr S(\rr d)$, $\mathcal S_s(\rr d)$ and $\Sigma _s(\rr d)$,  and to a unitary operator on $L^2(\rr d)$. 

\par

The following lemma shows that elements in Gelfand-Shilov spaces can be characterized by estimates of the form
\begin{equation}\label{GSexpcond}
|f(x)|\le Ce^{-\ep |x|^{1/s}}\quad \text{and}\quad |\widehat f (\xi )|\le Ce^{-\ep |\xi |^{1/s}} .
\end{equation}
The proof is omitted, since the result can be found in e.{\,}g. \cite{GS, ChuChuKim}.

\par

\begin{lemma}\label{GSFourierest}
Let $f\in \mathcal S'_{1/2}(\rr d)$. Then the following is true:
\begin{enumerate}
\item if $s\ge 1/2$, then $f\in \mathcal S_s(\rr d)$, if and only if there are constants $\ep >0$ and $C>0$
such that \eqref{GSexpcond} holds;

\vrum

\item  if $s> 1/2$, then $f\in \Sigma _s(\rr d)$, if and only if for each $\ep >0$, there is a constant $C$
such that \eqref{GSexpcond} holds.
\end{enumerate} 
\end{lemma}

\par

Gelfand-Shilov spaces posses several other  convenient properties. For example, they can easily be characterized by Hermite functions. We recall that the Hermite function $h_\alpha$ with respect to the multi-index $\alpha \in \mathbf N^d$ is defined by
$$
h_\alpha (x) = \pi ^{-d/4}(-1)^{|\alpha |}(2^{|\alpha |}\alpha
!)^{-1/2}e^{|x|^2/2}(\partial ^\alpha e^{-|x|^2}).
$$
The set $(h_\alpha )_{\alpha \in \mathbf N^d}$ is an orthonormal basis for $L^2(\rr d)$. In particular,
\begin{equation}\label{hermexp}
f=\sum _\alpha c_\alpha h_\alpha ,\quad c_\alpha =(f,h_\alpha )_{L^2(\rr d)},
\end{equation}
and
$$
\nm f{L^2}=\nm {\{ c_\alpha \} _\alpha }{l^2}<\infty ,
$$
when $f\in L^2(\rr d)$. Here and in what follows, $(\cdo ,\cdo )_{L^2(\rr d)}$ denotes any continuous extension of the $L^2$ form on $\mathcal S_{1/2}(\rr d)$.

\par

It is well-known that $f$ here belongs to $\mathscr S(\rr d)$, if and only if
\begin{equation}\label{Shermexp}
\nm {\{ c_\alpha \eabs \alpha ^t\} _\alpha }{l^2}<\infty 
\end{equation}
for every $t\ge 0$. Furthermore, for every $f\in \mathscr S'(\rr d)$, the expansion \eqref{hermexp} still holds, where the sum converges in $\mathscr S'$, and \eqref{Shermexp} holds for some choice of $t\in \mathbf R$, depending on $f$. The same conclusion holds after the $l^2$ norm has been replaced by any $l^p$ norm with $1\le p\le \infty$.

\par

The following proposition, which can be found in e.{\,}g. \cite{GrPiRo}, shows that similar conclusion for Gelfand-Shilov spaces hold, after the estimate \eqref{Shermexp} is replaced by
\begin{equation}\label{GShermexp}
\nm { \{ c_\alpha e^{t |\alpha |^{1/2s}} \} _\alpha }{l^p}<\infty .
\end{equation}
(Cf. formula (2.12) in \cite{GrPiRo}.)

\par

\begin{prop}\label{stftGS2}
Let $p\in [1,\infty ]$, $f\in \mathcal S_{1/2}'\rr d)$, $s\ge 1/2$ and let $c_\alpha$ be as in \eqref{hermexp}. Then the following is true:
\begin{enumerate}
\item $f\in \mathcal S_{s}'(\rr d)$, if and only if \eqref{GShermexp} holds for every $t<0$. Furthermore, \eqref{hermexp} holds where the sum converges in $\mathcal S_{s}'$;

\vrum

\item  $f\in \Sigma _{s}'(\rr d)$, if and only if \eqref{GShermexp} holds for some $t<0$. Furthermore, \eqref{hermexp} holds where the sum converges in $\Sigma _{s}'$;

\vrum

\item $f\in \mathcal S_{s}(\rr d)$, if and only if \eqref{GShermexp} holds for some $t>0$. Furthermore, \eqref{hermexp} holds where the sum converges in $\mathcal S_{s}$;

\vrum

\item  $f\in \Sigma _{s}(\rr d)$, if and only if \eqref{GShermexp} holds for every $t>0$. Furthermore, \eqref{hermexp} holds where the sum converges in $\Sigma _{s}$.
\end{enumerate}
\end{prop}

\par

\subsection{The short-time Fourier transform}

\par

Let $\phi \in \mathscr S(\rr d)\setminus 0$ be fixed. For every $f\in
\mathscr S'(\rr d)$, the \emph{short-time Fourier transform} $V_\phi
f$ is the distribution on $\rr {2d}$ defined by the formula
\begin{equation}\label{defstft}
(V_\phi f)(x,\xi ) =\mathscr F(f\, \overline{\phi (\cdo -x)})(\xi ).
\end{equation}
We note that the right-hand side defines an element in $\mathscr
S'(\rr {2d})\cap C^\infty (\rr {2d})$, and 
that if $f\in L^q_{(\omega )}$ for some $\omega \in
\mascP (\rr d)$, then $V_\phi f$ takes the form
\begin{equation}\tag*{(\ref{defstft})$'$}
V_\phi f(x,\xi ) =(2\pi )^{-d/2}\int _{\rr d}f(y)\overline {\phi
(y-x)}e^{-i\scal y\xi}\, dy.
\end{equation}

\par

In order to extend the definition of the short-time Fourier transform we reformulate \eqref{defstft} in terms of partial Fourier transforms and tensor products. More presisely, we let $\mathscr F_2F$ be the partial Fourier transform of $F(x,y)\in \mathscr S'(\rr {2d})$ with respect to the $y$-variable, and we let $U$ be the map which takes $F(x,y)$ into $F(y,y-x)$. Then it follows that
\begin{equation}\label{tensorsftf}
V_\phi f =(\mathscr F_2\circ U)(f\otimes \overline \phi )
\end{equation}
when $f\in \mathscr S'(\rr d)$ and $\phi \in \mathscr S(\rr d)$.

\par

We remark that tensor products of elements in Gelfand-Shilov
spaces are defined
in similar ways as for tensor products for distributions (cf. Chapter V in
\cite{Ho1}). Let $f,g\in \mathcal S_{1/2}'(\rr d)$ and let $s\ge 1/2$.
Then it follows that $f\otimes g \in \mathcal S_{s}'(\rr {2d})$, if and
only if $f,g\in \mathcal S_{s}'(\rr d)$. Similar fact holds for any
other choice of Gelfand-Shilov spaces of functions or distributions.

\par

The following result is essentially a restatement of ?? in \cite{CPRT10} and concerns the map
\begin{equation}\label{stftmap}
(f,\phi )\mapsto V_\phi f,
\end{equation}
and follows immediately from \eqref{tensorsftf},
and the facts that tensor products, $\mathscr F_2$ and $U$ are
continuous on Gelfand-Shilov spaces. (See also \cite{GZ, CPRT10}
for general properties of the short-time Fourier transform in
background of Gelfand-Shilov spaces.)

\par

\begin{prop}\label{stftGelfand1}
Let $s\ge 1/2$ and let $f,\phi \in \mathcal S'_{1/2}(\rr d)\setminus 0$. Then the map \eqref{stftmap} from $\mathscr S(\rr d)\times \mathscr S(\rr d)$ to $\mathscr S'(\rr {2d})$ is uniquely extendable to a continuous map from $\mathcal S_{1/2}'(\rr d)\times \mathcal S'_{1/2}(\rr d)$ to $\mathcal S'_{1/2}(\rr {2d})$. Furthermore, the following is true:
\begin{enumerate}
\item the map \eqref{stftmap} restricts to a continuous map from $\mathcal S_{s}(\rr d)\times \mathcal S_{s}(\rr d)$ to $\mathcal S_{s}(\rr {2d})$. Moreover, $V_\phi f\in \mathcal S_{s}(\rr {2d})$, if and only if $f,\phi \in \mathcal S_{s}(\rr d)$;

\vrum

\item the map \eqref{stftmap} restricts to a continuous map from $\mathcal S_{s}'(\rr d)\times \mathcal S_{s}'(\rr d)$ to $\mathcal S_{s}'(\rr {2d})$. Moreover, $V_\phi f\in \mathcal S_{s}'(\rr {2d})$, if and only if $f,\phi \in \mathcal S_{s}'(\rr d)$.
\end{enumerate}

\par

Similar facts hold after the spaces $\mathcal S_s$ and $\mathcal S_s'$ have been replaced by $\Sigma _s$ and $\Sigma _s'$ respectively.
\end{prop}

\par

We also have the following proposition.

\par

\begin{prop}\label{stftGelfand2}
Let $s\ge 1/2$, $\phi \in \mathcal S_{s}(\rr d)\setminus 0$ be even, and let $f\in \mathcal S_{1/2}'(\rr d)$.
Then the following is true:
\begin{enumerate}
\item $f\in  \mathcal S_{s}(\rr d)$, if and only if for some $\ep > 0$ and some constant $C_\ep$ it holds
\begin{equation}\label{stftexpest2}
|V_\phi f(x,\xi )|\le C_\ep e^{-\ep (|x|^{1/s}+|\xi |^{1/s})}\text ;
\end{equation}

\vrum

\item if $f\in \mathcal S_{s}'(\rr d)$, then there are constants $\ep > 0$ and $C_\ep>0$ such that
\begin{equation}\label{stftexpest}
|V_\phi f(x,\xi )|\le C_\ep e^{\ep (|x|^{1/s}+|\xi |^{1/s})}\text ;
\end{equation}

\vrum

\item if for every $\ep > 0$, there is a constant $C_\ep$ such that \eqref{stftexpest} holds, then $f\in \mathcal S_{s}'(\rr d)$.
\end{enumerate}
\end{prop}

\par

Proposition \ref{stftGelfand2} can be found in \cite{CPRT10} and to some
extend also in \cite{GZ}. Since the arguments in the proof are important
later on, we present here an explicit  proof, based on reformulation of the
statements in terms of Wigner distributions.

\par

First let $f,g\in L^2(\rr d)$.
Then the \emph{Wigner distribution} of $f$ and $g$ is defined by the 
formula
$$
W_{f,g}(x,\xi )=(2\pi )^{-d/2}\int f(x-y/2)\overline{g(x+y/2)}e^{i\scal y\xi}\, dy .
$$
We note that the Wigner distribution is closely connected to
the short-time Fourier transform, since
$$
V_\phi f(x,\xi ) =2^{-d}e^{i\scal x\xi /2}W_{f, \check \phi}(-x/2,\xi /2),
$$
which follows by straight-forward computations. Here $\check f(x)=f(-x)$.
From this relation it follows that 
most of the properties which involve short-time Fourier transform also hold after replacing 
the short-time Fourier transforms by Wigner distributions. For example, Propositions \ref
{stftGelfand1} and \ref{stftGelfand2} remain the same after such replacements.

\par

\begin{proof}
(1) If $f\in \mathcal S_s(\rr d)$, then it follows from Lemma \ref{GSFourierest} and Proposition 
\ref{stftGelfand1} that \eqref{stftexpest2} holds for some constants $\ep >0$ and $C_\ep >0$.

\par

Now assume instead that \eqref{stftexpest2} holds for some constants $\ep >0$ and $C_\ep 
>0$. Then \eqref{stftexpest2} still holds after $V_\phi f$ has been replaced by $W_{f, \check 
\phi}=W_{f,\phi}$, provided the constants $\ep$ and $C_\ep$ have been replaced by other 
suitable ones, if necessary. Since
$$
|\mathscr F(W_{f,\phi })(\xi ,x)| = |V_\phi f(-x,\xi )|,
$$
by Parseval's formula, it follows that \eqref{stftexpest2} holds for both $W_{f,\phi}$ and $
\mathscr F(W_{f,\phi})$. Hence, $f\in \mathcal S_s(\rr d)$ by Lemma \ref{GSFourierest}. 
This proves (1).

\par

The assertion (2) follows by straight-forward computations, using the fact that
$$
V_\phi f(x,\xi ) =\scal f{\overline {\phi (\cdo -x)}e^{-i\scal \cdo \xi }}
$$
in combination  with Lemma \ref{GSFourierest}.

\par

On the other hand, if for every $\ep >0$ there is a constant $C_\ep >0$ such that
\eqref{stftexpest} holds and $\fy $ is a finite sum of Hermite functions, 
then $V_\phi \fy \in \mathcal S_{s}(\rr {2d})$, and
\begin{equation}\label{dualstftrel}
(f,\fy ) _{L^2(\rr d)}= c (V_\phi f,V_\phi \fy )_{L^2(\rr {2d})}
\end{equation}
is well-defined. Here $c=\nm {\phi}{L^2}^{-2}>0$. Now, by \eqref{stftexpest}, (1) and the fact 
that finite sums of Hermite functions are dense in $\mathcal S_{s}(\rr {d})$, it follows that the 
right-hand side of \eqref{dualstftrel} defines a continuous linear form on $\mathcal S_{s}(\rr 
{d})$ with respect to $\fy$. Hence, $f\in \mathcal S_{s}'(\rr {d})$, which gives (3), and the 
proof is complete.
\end{proof}

\par

\begin{rem}\label{Upsilspace}
There is obviously a gap between the necessary and sufficiency conditions in (2) and (3) of 
Proposition \ref{stftGelfand2}. In general it seems to be difficult to find convenient
equivalent conditions 
for the short-time Fourier transform of $f\in \mathcal S'_{1/2}$ in order for $f$ should belong to
$\mathcal S'_{s}$, for some $s>1/2$.

\par

On the other hand, for each $s\ge 1/2$ and $\phi \in \mathcal S_s(\rr d)\setminus 0$, let $
\Upsilon  _{s,\phi} (\rr d)$ be the space which consists of all $f\in \maclS '_s(\rr d)$ such 
that for every $\ep >0$ there is a constant $C_\ep >0$ such that \eqref{stftexpest} holds. Then 
$\Upsilon _{s,\phi} (\rr d)$ is still a "large space" in the sense that it contains every $\mathcal 
S'_t(\rr d)$ for $t>s$.

\par

For future references we set $\Upsilon = \Upsilon _{s,\phi} $ when
$s=1/2$ and $\phi (x)=\pi ^{-d/4}e^{-|x|^2/2}$.
\end{rem}

\par

\subsection{Modulation spaces}\label{subsec1.4}

\par

We shall now discuss modulation spaces and recall some basic properties. In what follows 
we let $\mascB $ be a \emph{mixed norm space} on $\rr d$. This means that for some 
$p_1,\dots ,p_n\in [1,\infty ]$ and vector spaces
\begin{equation}\label{Vdirsum}
V_1,\dots ,V_n\subseteq \rr d\quad \text{such that}\quad
V_1\oplus \cdots \oplus V_n =\rr d,
\end{equation}
then $\mascB =\mascB _n$, where $\mascB _j$, $j=1,\dots ,n$ is inductively defined 
by
\begin{equation}\label{mixnormspacenorm}
\mascB _j =
\begin{cases}
L^{p_1}(V_1), & \ j=1
\\[1ex]
L^{p_j}(V_j; \mascB _{j-1}), &\ j=2,\dots , n.
\end{cases}
\end{equation}
The minimal and maximal exponents $\min (p_1,\dots ,p_n)$ and $\max (p_1,\dots ,p_n)$ 
are denoted by $\nu _1(\mascB )$ and $\nu _2(\mascB )$ respectively, and the norm of 
$\mascB $ is given by $\nm f{\mascB } \equiv \nm {F_{n-1}}{L^{p_n}(V_n)}$, where $F_0=f$ and
$$
F_j(x_{j+1},\dots ,x_n) =
\nm {F_{j-1}(\cdo ,x_{j+1},\dots ,x_n)}{L^{p_j}(V_j)},\quad j=1,\dots , n-1.
$$
In several situations the notation $L^p(V)$ is used instead of $\mascB$, where
\begin{equation}\label{Vparrays}
V=(V_1,\dots ,V_n)\quad \text{and} \quad p=(p_1,\dots ,p_n).
\end{equation}
We set  $\mascB ' = L^{p'}(V)$, where $p'=(p_1',\dots ,p_2')$ and
$p_j'\in [1,\infty ]$ is the conjugate
exponent of $p_j$, $j=1,\dots ,n$. That is, $p_j$ and $p_j'$ should
satisfy $1/p_j+1/p_j'=1$. If $\nu _2(\mascB )<\infty$, then the dual of $\mascB$ with
respect to $(\cdo ,\cdo )_{L^2}$ is given by $\mascB '$.

\par

In some situations we relax the conditions on $p_1,\dots ,p_n$ in such
way that they should belong to $(0,\infty ]$ instead of $[1,\infty ]$. Still we set
$$
\nm f{L^{p_j}(V_j)}\equiv
\begin{cases}
\left (\int _{V_j}|f(x_j)|^{p_j}\, dx_j\right )^{1/p_j}, & \text{when} \  0< p_j<\infty 
\\[2ex]
\underset {x_j\in V_j}  \essup \big ( |f(x_j)| \big ), & \text{when} \  \phantom{1<}p_j=\infty ,
\end{cases}
$$
where $f$ is measurable on $V_j$. (Cf. \cite{BLo}.) We note that
$\nm \cdo{L^{p_j}(V_j)}$ is a quasi-norm, but not a norm, when
$p_j<1$. Furthermore, in this case, $L^{p_j}(V_j)$ is a quasi-Banach
space, with topology defined by this quasi-norm.

\par

Now, if $p_1,\dots ,p_n\in (0,\infty ]$, and $V=(V_1,\dots V_n)$ is the same as above, then
$L^p(V)$ is called \emph{mixed quasi-norm space} on $\rr d$, and is defined as $\mascB _n$
in \eqref{mixnormspacenorm}.

\par

\begin{example}
Let $p,q\in [1,\infty ]$, and $L^{p,q}(\rr {2d})$ and its twisted space $L^{p,q}_{\tw}(\rr {2d})$
be the Banach spaces, which consist of all  $F\in L^1_{loc}(\rr {2d})$ such that
$$
\nm F{L^{p,q}} \equiv \Big (\intrd \Big (\intrd |F(x,\xi
 )|^p\, dx\Big )^{q/p}\, d\xi \Big )^{1/q}<\infty \, .
$$
and
$$
\nm F{L^{p,q}_{\tw}} \equiv \Big (\intrd \Big (\intrd |F(x,\xi
 )|^q\, d\xi \Big )^{p/q}\, dx \Big )^{1/p}<\infty \, ,
$$
respectively (with obvious modifications when $p=\infty$ or $q=\infty$).
Then it follows that both $L^{p,q}(\rr {2d})$ and $L^{p,q}_{\tw}(\rr {2d})$
are mixed norm spaces.

\par

If instead $p,q \in (0,\infty ]$, then $L^{p,q}$ and $L^{p,q}_{\tw}$ are defined
in analogous ways, where the condition $F\in L^1_{loc}(\rr {2d})$ has to
be replaced by $F\in L^r_{loc}(\rr {2d})$ with $r=\min (p,q )$. In this case,
one obtains mixed quasi-norm spaces.
\end{example}

\par

The definition of modulation spaces is given in the following.

\par

\begin{defn}\label{bfspaces2}
Let $\mascB $ be a mixed quasi-norm space on
$\rr {2d}$, $\omega \in \mascP _{Q}^0(\rr {2d})$, and let $\phi =\pi ^{-d/4}e^{-|x|^2/2}$. Then the \emph{modulation space}
$M(\omega ,\mascB )$ consists of all $f\in
\mathcal S_{1/2}'(\rr d)$ such that
\begin{equation}\label{modnorm2}
\nm f{M(\omega ,\mascB )}
\equiv \nm {V_\phi f\, \omega }{\mascB }<\infty .
\end{equation}
If $\omega =1$, then the notation $M(\mascB )$ is used
instead of $M(\omega ,\mascB )$.
\end{defn}

\par

Since the cases $\mascB =L^{p,q}(\rr {2d})$ and
$\mascB =L^{p,q}_{\tw}(\rr {2d})$ are especially important to us we set $M^{p,q}_
{(\omega )}(\rr d) = M(\omega ,L^{p,q}(\rr {2d}))$ and $W^{p,q}_{(\omega )}(\rr d)
= M(\omega ,L^{p,q}_{\tw}(\rr {2d}))$. We recall that 
if $\omega \in \mathscr P(\rr {2d})$, then the former space is a
"classical modulation space", and the latter space is related to 
certain types of classical Wiener amalgam spaces. For convenience we
set $M^p_{(\omega)}=M^{p,p}_{(\omega
)}=W^{p,p}_{(\omega )}$. Furthermore, we set $M^{p,q}_s=M^{p,q}_{(\sigma _s)}$ and
$M^{p}_s=M^{p}_{(\sigma _s)}$, where $\sigma _s$ is given by \eqref{defsigmas},
and if  $\omega =1$, then we use the notations $M(\mascB )$, $M^{p,q}$, $W^{p,q}$
and $M^p$, instead of $M(\omega ,\mascB )$, $M^{p,q}_{(\omega )}$,
$W^{p,q}_{(\omega )}$ and $M^{p}_{(\omega )}$, respectively. Here
we note that
$$
\sigma _s(x)=\eabs {x}^s = (1+|x|^2+|\xi |^2)^{s/2}\quad \text{and}\quad
\sigma _s(x,\xi )=\eabs {x,\xi }^s = (1+|x|^2+|\xi |^2)^{s/2}.
$$

\par

For exponential type weights we have the following proposition.
We omit the proof,  since the result can be found in \cite{F1,FG1,FG2,Gc2,To5}.
Here and in what follows we write $p_1\le p_2$, when
\begin{equation}\label{parrays}
p_1 =(p_{1,1},\dots ,p_{1,n})\in (0,\infty ]^n\quad \text{and}\quad p_2 =(p_{2,1},
\dots ,p_{2,n})\in (0,\infty ]^n
\end{equation}
satisfy $p_{1,j}\le p_{2,j}$ for every $j=1,\dots ,n$.

\par

\begin{prop}\label{p1.4}
Let $p,q,p_j,q_j\in [1,\infty ]$, $\omega
,\omega _j,v\in \mascP  _{E}(\rr {2d})$ for $j=1,2$ be such that $v$ is
submultiplicative and even, $\omega$ is $v$-moderate, and let $\mascB $
be a mixed normed space  on $\rr {2d}$. Then the following is true:
\begin{enumerate}
\item[{\rm{(1)}}] if $\phi \in M^1_{(v)}(\rr d)\setminus 0$, then
$f\in M{(\omega ,\mascB )}$ if and only if \eqref{modnorm2} holds,
i.{\,}e. $M{(\omega ,\mascB )}$ is independent of the choice of
$\phi$. Moreover, $M{(\omega ,\mascB )}$ is a Banach space under the
norm in \eqref{modnorm2}, and different choices of $\phi$ give rise to
equivalent norms;

\vrum

\item[{\rm{(2)}}] if  \eqref{Vdirsum}, \eqref{Vparrays} and \eqref{parrays} hold
with $p_1\le p_2$, and $\omega _2\le C \omega _1$ for some constant $C$, then
%%
%\begin{alignat*}{2}
$$
\Sigma _{1}(\rr d) \subseteq M(\omega _1,L^{p_1}(V))  
   \subseteq M(\omega _2,L^{p_2}(V))\subseteq
\Sigma ' _{1}(\rr d)\text ;
$$
%\\[1ex]
%M^1_{(v_0v)}(\rr d)\subseteq M(&\omega ,\mascB )\subseteq
%M^{\infty}_{(1/(v_0v))}(\rr d)\text ;
%\end{alignat*}
%%

\vrum

\item[{\rm{(3)}}] the sesqui-linear form $( \cdo ,\cdo )_{L^2}$ on
$\Sigma _1(\rr d)$ extends to a continuous map from $M^{p,q}_{(\omega
)}(\rr d)\times M^{p'\! ,q'}_{(1/\omega )}(\rr d)$ to $\mathbf
C$. This extension is unique, except when $p=q'\in \{ 1,\infty \}$. On the
other hand, if $\nmm a = \sup |{(a,b)_{L^2}}|$, where the supremum is
taken over all $b\in M^{p',q'}_{(1/\omega )}(\rr d)$ such that
$\nm b{M^{p',q'}_{(1/\omega )}}\le 1$, then $\nmm {\cdot}$ and $\nm
\cdot {M^{p,q}_{(\omega )}}$ are equivalent norms;

\vrum

\item[{\rm{(4)}}] if $p,q<\infty$, then $\Sigma _1(\rr d)$ is dense
in $M^{p,q}_{(\omega )}(\rr d)$, and the dual space of $M^{p,q}_{(\omega
)}(\rr d)$ can be identified with $M^{p'\! ,q'}_{(1/\omega )}(\rr
d)$, through the form $(\cdo  ,\cdo )_{L^2}$. Moreover,
$\Sigma _1(\rr d)$ is weakly dense in $M^{\infty }_{(\omega )}(\rr
d)$.
\end{enumerate}
\end{prop}

\par

\subsection{The Bargmann transform}

\par

We shall now consider the Bargmann transform which is defined by the
formula
\begin{equation*}%\label{bargtransf}
(\mathfrak Vf)(z) =\pi ^{-d/4}\int _{\rr d}\exp \Big ( -\frac 12(\scal
z z+|y|^2)+2^{1/2}\scal zy \Big )f(y)\, dy,
\end{equation*}
when $f\in L^2(\rr d)$. We note that if $f\in
L^2(\rr d)$, then the Bargmann transform
$\mathfrak Vf$ of $f$ is the entire function on $\cc d$, given by
$$
(\mathfrak Vf)(z) =\int \mathfrak A_d(z,y)f(y)\, dy,
$$
or
\begin{equation}\label{bargdistrform}
(\mathfrak Vf)(z) =\scal f{\mathfrak A_d(z,\cdo )},
\end{equation}
where the Bargmann kernel $\mathfrak A_d$ is given by
$$
\mathfrak A_d(z,y)=\pi ^{-d/4} \exp \Big ( -\frac 12(\scal
zz+|y|^2)+2^{1/2}\scal zy\Big ).
$$
Here
$$
\scal zw = \sum _{j=1}^dz_jw_j,\quad \text{when} \quad
z=(z_1,\dots ,z_d) \in \cc d\quad  \text{and} \quad w=(w_1,\dots ,w_d)\in \cc d,
$$
and otherwise $\scal \cdo \cdo $ denotes the duality between test function spaces and their
corresponding duals.
We note that the right-hand side in \eqref{bargdistrform} makes sense
when $f\in \maclS _{1/2}'(\rr d)$ and defines an element in $A(\cc d)$,
since $y\mapsto \mathfrak A_d(z,y)$ can be interpreted as an element
in $\maclS _{1/2} (\rr d)$ with values in $A(\cc d)$. Here and in what follows,
$A(\cc d)$ denotes the set of all entire functions on $\cc d$.

\par

It was proved by Bargmann that $f\mapsto \mathfrak Vf$ is a bijective and isometric map 
from $L^2(\rr d)$ to the Hilbert space $A^2(\cc d)$, the set of entire functions $F$ on $\cc 
d$ which fulfills
\begin{equation}\label{A2norm}
\nm F{A^2}\equiv \Big ( \int _{\cc d}|F(z)|^2d\mu (z)  \Big )^{1/2}<\infty .
\end{equation}
Here $d\mu (z)=\pi ^{-d} e^{-|z|^2}\, d\lambda (z)$, where $d\lambda (z)$ is the Lebesgue measure on $\cc d$, and the scalar product on $A^2(\cc d)$ is given by
\begin{equation}\label{A2scalar}
(F,G)_{A^2}\equiv  \int _{\cc d} F(z)\overline {G(z)}\, d\mu (z),\quad F,G\in A^2(\cc d).
\end{equation}

\par

Furthermore, Bargmann proved that there is a convenient reproducing kernel on
$A^2(\cc d)$, given  by the formula
\begin{equation}\label{reproducing}
F(z)= \int _{\cc d} e^{(z,w)}F(w)\, d\mu (w),\quad F\in
A^2(\cc d),
\end{equation}
where $(z,w)$ is the scalar product of $z\in \cc d$ and $w\in \cc d$ (cf. \cite{B1,B2}).
Note that this reproducing kernel is unique in view of \cite{Kr}.

\par

From now on we assume that $\phi$ in \eqref{defstft},
\eqref{defstft}$'$  and \eqref{modnorm2} is given by
\begin{equation}\label{phidef}
\phi (x)=\pi ^{-d/4}e^{-|x|^2/2},
\end{equation}
if nothing else is stated. Then it follows that the Bargmann transform
can be expressed in terms of the short-time Fourier transform
$f\mapsto V_\phi f$. More precisely, let $S$ be the dilation operator given by
\begin{equation}\label{Sdef}
(SF)(x,\xi ) = F(2^{-1/2}x,-2^{-1/2}\xi ),
\end{equation}
when $F\in L^1_{loc}(\rr {2d})$. Then it
follows by straight-forward computations that
\begin{multline}\label{bargstft1}
  (\mathfrak{V} f)(z)  =  (\mathfrak{V}f)(x+\im \xi ) 
 =  (2\pi )^{d/2}e^{(|x|^2+|\xi|^2)/2}e^{-i\scal x\xi}V_\phi f(2^{1/2}x,-2^{1/2}\xi )
\\[1ex]
=(2\pi )^{d/2}e^{(|x|^2+|\xi|^2)/2}e^{-i\scal x\xi}(S^{-1}(V_\phi f))(x,\xi ),
\end{multline}
or equivalently,
\begin{multline}\label{bargstft2}
V_\phi f(x,\xi )  =  
(2\pi )^{-d/2} e^{-(|x|^2+|\xi |^2)/4}e^{-\im \scal x \xi /2}(\mathfrak{V}f)
(2^{-1/2}x,-2^{-1/2}\xi).
\\[1ex]
=(2\pi )^{-d/2}e^{-i\scal x\xi /2}S(e^{-|\cdo |^2/2}(\mathfrak{V}f))(x,\xi ).
\end{multline}
For future references we observe that \eqref{bargstft1} and \eqref{bargstft2}
can be formulated into
\begin{equation*}%\label{BargStftlink}
\mathfrak V = U_{\mathfrak V}\circ V_\phi ,\quad \text{and}\quad
U_{\mathfrak V}^{-1} \circ \mathfrak V =  V_\phi ,
\end{equation*}
where $U_{\mathfrak V}$ is the linear, continuous and bijective operator on
$\mathscr D'(\rr {2d})\simeq \mathscr D'(\cc d)$, given by
\begin{equation}\label{UVdef}
(U_{\mathfrak V}F)(x,\xi ) = (2\pi )^{d/2} e^{(|x|^2+|\xi |^2)/2}e^{-i\scal x\xi}
F(2^{1/2}x,-2^{1/2}\xi ) .
\end{equation}

\par

\begin{defn}\label{thespaces}
Let $\omega _1\in \mascP _{Q}^0(\rr {2d})$, $\omega _2\in \mascP _{Q}(\rr {2d})$,
$\mascB $ be a mixed quasi-norm space on $\rr {2d}=\cc d$, and let $r>0$ be such that $r\le \nu _1(\mascB )$.
\begin{enumerate}
\item The space $B(\omega _2,\mascB )$ is the modified
weighted
$\mascB $-space which consists of all $F\in L^r_{loc}(\rr {2d})=
L^r_{loc}(\cc {d})$ such that
$$
\nm F{B(\omega _2,\mascB )}\equiv \nm {(U_{\mathfrak V}^{-1}F)\omega _2}{\mascB }<\infty .
$$
Here $U_{\mathfrak V}$ is given by \eqref{UVdef};

\vrum

\item The space $A(\omega _2,\mascB )$ consists of all $F\in A(\cc
d)\cap B(\omega _2,\mascB )$ with topology inherited
from $B(\omega _2,\mascB )$;

\vrum

\item The space $A_0(\omega _1,\mascB )$ is given by
$$
A_0(\omega _1,\mascB ) \equiv \sets {(\mathfrak Vf)}{f\in M(\omega _1
,\mascB )},
$$
and is equipped with the quasi-norm $\nm F{A_0(\omega _1,\mascB )}\equiv \nm
f{M(\omega _1,\mascB )}$, when $F=\mathfrak Vf$.
\end{enumerate}
\end{defn}

\par

We note that the spaces in Definition \ref{thespaces} are normed
spaces when $\nu _1(\mascB )\ge 1$.

\par

For conveneincy we set $\nm F{B(\omega ,\mascB )}=\infty$, when
$F\notin B(\omega ,\mascB )$ is measurable, and
$\nm F{A(\omega ,\mascB )}=\infty$, when $F\in A(\cc d)\setminus B(\omega ,\mascB )$.
We also set
$$
B^{p,q}_{(\omega )} = B^{p,q}_{(\omega )}(\cc d) = B(\omega ,\mascB ),\quad
A^{p,q}_{(\omega )} = A^{p,q}_{(\omega )}(\cc d) = A(\omega ,\mascB )
$$
when $\mascB = L^{p,q}(\cc d)$, $A^p_{(\omega )}=A^{p,p}_{(\omega )}$,
and if $\omega =1$, then
we use the notations $B^{p,q}$, $A^{p,q}$, $B^p$ and $A^p$ instead of
$B^{p,q}_{(\omega )}$, $A^{p,q}_{(\omega )}$, $B^p_{(\omega )}$
and $A^p_{(\omega )}$, respectively.

\par

For future references we note that the $B^{p}_{(\omega )}$ quasi-norm is given by
\begin{multline}\label{BLpnorm}
\nm F{B^{p}_{(\omega )}} = 2^{d/p}(2\pi )^{-d/2}\left ( \int _{\cc d} |e^{-|z|^2/2}F(z)\omega (2^
{1/2}\overline z)|^p\, d\lambda (z)  \right )^{1/p}
\\[1ex]
= 2^{d/p}(2\pi )^{-d/2}\left ( \iint _{\rr {2d}} |e^{-(|x|^2+|\xi |^2)/2}F(x+i\xi )\omega (2^
{1/2}x,-2^{1/2}\xi )|^p\, dxd\xi \right )^{1/p}
\end{multline}
(with obvious modifications when $p=\infty$). Especially it follows that the norm and scalar
product in $B^2_{(\omega )}(\cc d)$ take the forms
\begin{equation*}%\label{BL2norm}
\begin{alignedat}{2}
\nm F{B^{2}_{(\omega )}} &= \left ( \int _{\cc {d}} |F(z)\omega (2^
{1/2}\overline z)|^2\, d\mu (z) \right )^{1/2},& \quad F &\in B^2_{(\omega )}(\cc d)
\\[1ex]
(F,G)_{B^2_{(\omega )}} &=  \int _{\cc {d}} F(z)\overline {G(z)}\omega (2^
{1/2}\overline z)^2\, d\mu (z) ,& \quad F,G &\in B^2_{(\omega )}(\cc d)
\end{alignedat}
\end{equation*}
(cf. \eqref{A2norm} and \eqref{A2scalar}).

\par

The following result shows that the norm in $A_0(\omega,\mathcal{B})$ is well-defined.

\par

\begin{prop}\label{mainstep1prop}
Let $\omega \in \mascP  _{Q}^0(\rr {2d})$, $\mascB $ be a mixed
norm space on $\rr {2d}$ and let $\phi$ be as in
\eqref{phidef}. Then $A_0(\omega ,\mascB )\subseteq A(\omega ,\mascB )$,
and the map $\mathfrak V$ is an
isometric injection from $M(\omega ,\mascB )$ to $A(\omega ,\mascB )$.
\end{prop}

\par

\begin{proof}
The result is an immediate consequence of \eqref{bargstft1},
\eqref{bargstft2} and Definition \ref{thespaces}.
\end{proof}

\par

In the case $\omega =1$ and $\mascB =L^2$, it follows from
\cite {B1} that Proposition \ref{mainstep1prop}
holds, and the inclusion is replaced by equality. That is, we have
$A^2_0=A^2$ which is called the Bargmann-Fock space, or just the Fock
space. In Section \ref{sec4} we improve the latter property and show
that for any choice of $\omega \in \mascP  _{Q}^0$ and every mixed quasi-norm space
$\mascB $, we have $A_0(\omega ,\mascB )=A(\omega ,\mascB )$.

\par

%%%%%%%%%%%%%%%%%%%%%%%
\section{Weight functions}\label{sec2}
%%%%%%%%%%%%%%%%%%%%%%%

\par

In this section we establish results on weight functions which are needed. In the first part we
investigate weights belonging to $\mathscr P_E$. Here we are especially focused on finding 
properties which are needed to show that $\mathscr P_E$ is contained in convenient
subfamilies of $ \mascP _Q$, which are introduced in the second part of the section.

\par

\subsection{Moderate weights}

\par

For a moderate weight $\omega$ there are convenient ways to find smooth weigths $\omega _0$ which are equivalent in the sense that for some constant $C>0$ we have
\begin{equation}\label{weightequiv}
C^{-1}\omega _0\le \omega \le C\omega _0.
\end{equation}
In fact, we have the following result, which extends Lemma 1.2 (4) in \cite{To04B}. Here the weight
$\omega \in \mascP _E(\rr d)$ is called \emph{elliptic} if $\omega \in C^\infty (\rr d)$, and for each
multi-index $\alpha$, we have
\begin{equation}\label{omegaderest}
(\partial ^\alpha \omega _0)/\omega _0 \in L^\infty (\rr d).
\end{equation}

\par

\begin{lemma}\label{smoothweights}
Let $\omega \in \mascP _E(\rr d)$. Then it exists an elliptic weight
$\omega _0\in \mascP _E(\rr d)$ such that \eqref{weightequiv} holds.
\end{lemma}

\par

Lemma \ref{smoothweights} follows by similar arguments as in the proof of
Lemma 1.2 (4) in \cite{To04B}. In order to be self-contained we here present a proof.

\par

\begin{proof}
Let $\psi \in C^\infty (\rr d)$ be a positive function which is bounded by Gauss functions together with all its derivatives, and let $\omega _0=\psi *\omega$. Also let $v\in \mascP _E(\rr d)$  be chosen such that $\omega$ is $v$-moderate. Then $\omega _0$ is smooth, and
$$
\omega _0 (x)=\int \omega (x-y)\psi (y)\, dy \le C_1\omega (x),
$$
with
$$
C_1=\int v(-y)\psi (y)\, dy <\infty .
$$

\par

Furthermore, if
$$
C_2=\int v(y)^{-1}\psi (y)\, dy <\infty ,
$$
then
$$
C_2\omega (x) =\int \omega (x-y+y)\frac {\psi (y)}{v(y)}\, dy \le C\int \omega (x-y )\psi (y)\, dy =C\omega _0(x),
$$
for some constant $C$. This proves that $\omega _0\in \mascP _E \cap C^\infty$, and that \eqref{weightequiv} is fulfilled.

\par

By differentiating $\omega _0$, the first part of the proof gives that for some Gaussian $\psi _1$, and some constants $C_1$ and $C_2$ we have
\begin{equation*}
|\partial ^\alpha \omega _0| = |(\partial ^\alpha \psi )*\omega | \le |\partial ^\alpha \psi |*\omega
\le \psi _1*\omega \le C_1\omega \le C_2\omega _0.
\end{equation*}
Hence, \eqref{omegaderest} is fulfilled for  $\omega =\omega _0$. The proof is complete.
\end{proof}

\par

We also need some properties concerning the minimal weight which moderates a specific weight $\omega \in \mascP _E(\rr d)$. For any such weight, the fact that $\omega$ is $v$-moderate for some function $v$, implies that
\begin{equation}\label{v0def}
v_0(x)\equiv \sup _{y_0\in \rr d}\frac {\omega (x+y_0)}{\omega (y_0)}
\end{equation}
is well-defined. Furthermore, by straight-forward computations it follows that
\begin{equation}\label{vomegamod}
\omega (x+y)\le \omega (x)v(y)\quad \text{and}\quad v(x+y)\le v(x)v(y),\qquad x,y\in \rr d,
\end{equation}
holds for $v=v_0$. The following result shows that $v_0$ is
minimal among elements which moderates $\omega$, and
satisfies \eqref{vomegamod}. Furthermore, here we establish
differential properties of $v_0$ in terms of the functionals
\begin{equation*}%\label{J1J2def}
\begin{aligned}
(J_1\omega )(x,y ) &\equiv \inf _{y_0\in \rr d}\Big ( \frac {(\partial _y \omega ) (x+y_0)}{\omega (y_0)}
\Big ) = \inf _{y_0\in \rr d}\Big ( \frac {\scal {(\nabla \omega ) (x+y_0)}y }{\omega (y_0)}\Big ),
\\[1ex]
(J_2\omega )(x,y ) &\equiv \sup _{y_0\in \rr d}\Big ( \frac {(\partial _y \omega ) (x+y_0)}{\omega (y_0)}
\Big ) = \sup _{y_0\in \rr d}\Big ( \frac {\scal {(\nabla \omega ) (x+y_0)}y }{\omega (y_0)}\Big ),
\end{aligned}
\end{equation*}
when in additional $\omega$ is elliptic. We note that $J_1\omega$ and $J_2\omega$ satisfy
\begin{equation*}%\label{J1J2est}
|(J_k\omega )(x,y)|\le \sup _{y_0\in \rr d}
\Big ( \frac {|(\nabla \omega ) (x+y_0)|}{\omega (y_0)}\Big )|y| \le Cv_0(x)|y|,\quad k=1,2,
\end{equation*}
for such $\omega$ and some constant $C$, depending on $\omega$ only.

\par

\begin{lemma}\label{minmodweight}
Let $\omega \in \mascP _E(\rr d)$, $0<v\in L^\infty _{loc}(\rr d)$ be such that $\omega (x+y)
\le \omega (x)v(y)$ when $x,y\in\rr d$, and let $v_0$ be as in \eqref{v0def}. Then
$v_0\le v$, and \eqref{vomegamod} holds.

\par

Moreover, if in addition $\omega$ is elliptic, then the following is true:
\begin{enumerate}
\item $(J_1\omega )(x,y)$ and $(J_2\omega )(x,y)$ are continuous functions which are positively homogeneous in the $y$-variable of order one;

\vrum

\item if $x,y\in \rr d$, then
$$
\inf _{0\le t\le 1}(J_1\omega )(x+ty,y)\le v_0(x+y)-v_0(x)\le \sup _{0\le t\le 1}(J_2\omega )(x+ty,y) .
$$
In particular, $v_0$ is locally Lipschitz continuous;

\vrum

\item if $x,y\in \rr d$, then
\begin{multline*}
(J_1\omega )(x,y)\le \liminf _{h\to 0}\,  \frac {v_0(x+hy)-v_0(x)}{h}
\\[1ex]
\le \limsup _{h\to 0}\, 
\frac {v_0(x+hy)-v_0(x)}{h}\le (J_2\omega )(x,y).
\end{multline*}
\end{enumerate}
\end{lemma}

\par

\begin{proof}
The first part and (1) are simple consequences of the definitions. The details are left for the reader. It is also obvious that (3) follows from (1) and (2) if we replace $y$ in (2) by $hy$, and let $h$ turns to zero.

\par

It remains to prove (2), and then we may assume that $y\neq 0$ and $x$ are fixed. By \eqref{v0def} it follows that for every $\ep >0$, there is an element $y_0\in \rr d$ such that
$$
v_0(x+y)\le \frac {\omega (x+y+y_0)}{\omega (y_0)}+\ep .
$$
By Taylor's formula we get for some $\theta \in [0,1]$,
\begin{multline*}
v_0(x+y)\le \frac {\omega (x+y_0)}{\omega (y_0)} + \frac {\scal {(\nabla \omega )(x+\theta y+y_0)}y}{\omega (y_0)} + \ep 
\\[1ex]
\le v_0(x)+(J_2\omega )(x+\theta y,y)+\ep \le v_0(x)+\sup _{0\le t\le 1}(J_2\omega )(x+ty,y)+\ep.
\end{multline*}
Since $\ep$ was arbitrary chosen, the last inequality in (2) follows.

\par

In the same way, let $\ep >0$ be arbitrary and choose $y_1\in \rr d$ such that
$$
v_0(x)\le \frac {\omega (x+y_1)}{\omega (y_1)}+\ep .
$$
By Taylor's formula we get for some $\theta \in [0,1]$,
\begin{multline*}
v_0(x+y)\ge \frac {\omega (x+y+y_1)}{\omega (y_1)} =
\frac {\omega (x+y_1)}{\omega (y_1)} + \frac {\scal {(\nabla \omega )(x+\theta y+y_1)}y}{\omega (y_1)} 
\\[1ex]
\ge v_0(x)+(J_1\omega )(x+\theta y,y)-\ep \ge v_0(x)+\inf _{0\le t\le 1}(J_1\omega )(x+ty,y) -\ep ,
\end{multline*}
and the first inequality in (2) follows. The proof is complete.
\end{proof}

\par

\subsection{Subfamilies of Gaussian type weights}

\par

Next we discuss further appropriate conditions for subfamilies to $\mathscr P_Q^0(\rr d)$
and show that these
subfamilies contain both $\mathscr P_E(\rr d)$ as well as all weights of the form
$\omega (x)=Ce^{c|x|^\gamma}$, when $C>0$, $c\in \mathbf R$ and $0\le \gamma <2$.

\par

In the following definition we list most of the needed properties. The definitions involve global 
conditions of the form
\begin{equation}\label{addomegcond}
\frac {\omega (2^{1/2}x)e^{-\ep (1-\lambda ^2)|x|^2/2}}{\omega (2^{1/2}\lambda x)}
\le C_\ep ,\qquad 1-\theta _\ep < \lambda <1 ,\ x\in \rr d ,
\end{equation}
and
\begin{multline}\label{addomegcond2}
\lim  _{\lambda \to 1-}\Big (\sup _{x_2\in V^\bot}\frac {\omega _0(2^{1/2}x)e^{-\ep (1-\lambda ^2)|x|^2/2}}
{\omega _0(2^{1/2}\lambda x)}\Big ) \le 1,
\\[1ex]
x_1\in V,\ x_2\in V^\bot,\ x=x_1+x_2\in \rr d,
\end{multline}
for some vector space $V$. Here the dilation factor $2^{1/2}$, is needed because of the relation between the Bargmann transform and the short-time Fourier transform in \eqref{bargstft1}, and the quasi-norms in Definition \ref{thespaces}.

\par

\begin{defn}\label{defdilsuit}
Let $V\subseteq \rr d$ be a vector space.
\begin{enumerate}
\item  The weight $\omega \in \mathscr P_Q^0(\rr d)$ is called
\emph{dilated suitable} with respect to
$\ep \in (0,1]$, if there are constants $C_\ep >0$ and
$\theta _\ep \in (0,1)$ such that \eqref{addomegcond} holds.
If both $\omega$ and $1/\omega$ are dilated 
suitable with respect to every $\ep \in (0,1]$,
%and in addition $\omega \in \mascP _G^0(\rr d)$,
then $\omega$ is called \emph{strongly dilated suitable};

\vrum

\item  The weight $\omega \in \mathscr P_Q^0(\rr d)$ is called \emph{narrowly
dilated suitable} with respect to
$\ep \in (0,1]$ and $V$, if it is dilated suitable with respect to $\ep$,
and \eqref{addomegcond2} holds for every $x_1\in V$ and some
equivalent
%smooth
continuous
weight $\omega _0$ to $\omega$. If $\omega$
and $1/\omega$ are narrowly dilated 
suitable with respect to $V$ and every $\ep \in (0,1]$,
%and in addition $\omega \in \mascP _G^0(\rr d)$,
then $\omega$ is
called \emph{strongly narrowly dilated suitable}.
\end{enumerate}
\end{defn}

\par

The set of strongly dilated and strongly narrowly dilated suitable weights with respect to $V$
are denoted by $\mascP _{D}(\rr d)$ and $\mascP ^0_{D,V}(\rr d)$ respectively.
Furthermore we set $\mascP ^0_{D}(\rr d)=\mascP ^0_{D,V}(\rr d)$, when $V=\{ 0\}$.
We note that $\mascP ^0_{D,V}(\rr d)$ is increasing with respect to $V$, and that 
$\mascP ^0_{D,V}(\rr d)\subseteq \mascP _{D}(\rr d)$.

\par

If $\omega $ is $v$-moderate for some $v$, then $\omega$ is moderated
by some function which grow 
exponentially. It follows easily that $\omega$ satisfies \eqref{addomegcond}
in this case. Hence, any weight in $\mathscr P_E$ is 
dilated suitable.

\par

%In the following proposition we improve this property and show that any weight in $\mascP $
%or $\mascP _E$ are strongly dilated  suitable.
In the following proposition we stress the latter property and prove
that we in fact have that $\mascP _E\subseteq \mascP _D^0$.

\par

\begin{prop}\label{weightfamincl}
Let $V\subseteq \rr d$ be a vector space. Then the following
inclusions are true:
$$
\mascP (\rr d)  \subseteq \mascP  _E(\rr d)  \subseteq \mascP _{D}^0(\rr d)\bigcap
\mascP _G^0(\rr d),\quad 
\mascP _{D}^0(\rr d) \subseteq \mascP  _{D,V}^0(\rr d) \subseteq \mascP _{D}(\rr d),
$$
$$
\mascP _G^0(\rr d) \subseteq \mascP  _Q^0(\rr d) \bigcap \mascP  _G(\rr d) \subseteq
\mascP  _Q^0(\rr d) \bigcup \mascP  _G(\rr d) \subseteq \mascP  _Q(\rr d) 
$$

\par

Furthermore, if $C>0$, $0\le \gamma <2$, $t\in \mathbf R$, and
$\omega (x)=\displaystyle{Ce^{t|x|^\gamma}}$, $x\in \rr d$, then $\omega \in
\mascP  _{D}^0(\rr d)\cap \mascP _G^0(\rr d)$.
\end{prop}

\par

\begin{proof}
All the inclusions, except $\mascP _E\subseteq \mascP _{D}^0$
are immediate consequences of the definitions.

\par

Therefore, let $\omega \in \mascP _E(\rr d)$, and let $v_0$ be the same as in \eqref{v0def}. By Lemma \ref{smoothweights} we may assume that $\omega$ is elliptic.

\par

We shall prove that \eqref{addomegcond} and \eqref{addomegcond2}
holds for $\omega$ and $1/\omega$, for every choice of $\ep \in (0,1]$, and since
$\mascP _E$ is a group under multiplications, it suffices to prove
these relations for $\omega$.
%It is then no restriction to assume that $\ep <1$, and since 
Since the left-hand side of \eqref{addomegcond} is equal to $1$ as $x=0$,
the result follows if we prove
\begin{equation}\label{omeglambdest}
\frac {\omega (x)e^{-\ep (1-\lambda ^2)|x|^2/2}}{\omega (\lambda x)}
\le 1+C_0\sqrt {1-\lambda},\qquad 0<\lambda <1,
\end{equation}
for some constant $C_0$ which depends on $\ep \in (0,1)$.

\par

Let $\ep \in (0,1)$ be fixed. We have $v_0(x)\le Ce^{C|x|}$ for some
$C>1$, and we choose $R>0$ such that
$$
Ce^{RC-\ep R^2/2}<1.
$$
Since $0<\lambda <1$ we have $\lambda ^2<\lambda$. Hence
Lemma \ref{minmodweight} gives
\begin{equation}\label{fracest}
\frac {\omega (x)e^{-\ep (1-\lambda ^2)|x|^2/2}}{\omega (\lambda x)} \le v_0((1-\lambda )x)
e^{-\ep (1-\lambda )|x|^2/2}.
\end{equation}
We shall estimate the right-hand side, and start by considering the case when
$|x|\le R/\sqrt {1- \lambda}$. Since $v_0(0)=1$ and $(1-\lambda )x$ stays bounded,
the right-hand side of \eqref{fracest} can be estimated by
\begin{multline*}
v_0((1-\lambda )x)e^{-\ep (1-\lambda )|x|^2/2}
\le |v_0((1-\lambda )x) -v_0(0)|e^{-\ep (1-\lambda )|x|^2/2} + e^{-\ep (1-\lambda )|x|^2/2}.
\\[1ex]
\le C_1(1-\lambda )|x|e^{-\ep(1-\lambda )|x|^2/2}+1,
\end{multline*}
where the last inequality follows from the fact that $v_0$ is locally Lipschitz
continuous, in view of Lemma \ref{minmodweight}. Since
$$
(1-\lambda )|x|e^{-\ep (1-\lambda )|x|^2/2}\le R\sqrt {1-\lambda }
$$
as $|x|\le R/\sqrt {1-\lambda}$, we obtain
\begin{equation*}%\label{inclcase1}
v_0((1-\lambda )x)e^{-\ep (1-\lambda )|x|^2/2}\le 1+C_1R\sqrt {1-\lambda },
\qquad |x|\le R/\sqrt {1-\lambda},
\end{equation*}
and \eqref{omeglambdest} follows in this case.

\par

Next we consider the case $R/\sqrt {1-\lambda}\le |x|\le R/(\ep (1-\lambda ))$.
Then we have
$$
v_0((1-\lambda )x)\le Ce^{C(1-\lambda )|x|}\le Ce^{RC/\ep }
$$
and
$$
 e^{-(1-\lambda )|x|^2/2}\le e^{-R^2/2},
$$
which gives
\begin{equation}\label{inclcase2}
v_0((1-\lambda )x)e^{-\ep (1-\lambda )|x|^2/2}\le Ce^{RC/\ep -R^2/2}
=C^{1-1/\ep}\Big ( Ce^{RC-\ep R^2/2}\Big )^{1/\ep}<1,
\end{equation}
where the last inequality follows from our choice of $R$, together
with the fact that $C>1$. This proves \eqref{omeglambdest} in this case.

\par

Finally, we consider the case $R/(\ep (1-\lambda ))\le |x|$. From the assumptions on $R$,
it follows that $R>2C/\ep$, which implies that
\begin{equation}\label{CxRest}
C-|x|/2<C-R/(2\ep)<C-\ep R/2 <0.
\end{equation}
Then
\begin{multline*}
v_0((1-\lambda )x)e^{-(1-\lambda ^2)|x|^2/2}\le
Ce^{C(1-\lambda )|x|}e^{-(1-\lambda )|x|^2/2}
\\[1ex]
\le Ce^{(1-\lambda )|x|(C-|x|/2)}\le Ce^{R(C-\ep R/2)/\ep}<1,
\end{multline*}
where the last inequalities follows from \eqref{inclcase2} and \eqref{CxRest}.
This gives \eqref{omeglambdest} also for $R/(\ep (1-\lambda ))\le |x|$, and the proof is complete.
\end{proof}

\par

In most of our investigations, the pairs of weights and mixed norm
spaces fulfill the conditions in the following definition.

\par

\begin{defn}\label{feasiblpair}
Let $\mascB =L^p(V)$ be a mixed norm space on $\rr d$ such that \eqref{Vdirsum}
and \eqref{Vparrays} hold for some $p\in 
[1,\infty ]^n$, and let $\omega \in \mascP _Q^0(\rr d)$. Then the pair $(\mascB ,\omega )$
is called \emph{feasible} (\emph{strongly feasible}) on $\rr d$
if one of the following conditions hold:
\begin{enumerate}
\item $\nu _1(\mascB )>1$ and $\omega$ is dilated suitable with respect to $\ep =1$
($\omega$ is strongly dilated suitable);

\vrum

\item $\nu _2(\mascB )<\infty$, and $\omega$ is dilated suitable with respect to
$\ep =1$ ($\omega$ is strongly dilated suitable);

\vrum

\item $p_1=\infty$, $1<p_2,\dots ,p_{n-1}<\infty$, $p_n=1$, and $\omega$ is narrowly dilated 
suitable with respect to $\ep =1$ and $V=\{ 0\}$ ($\omega$ is strongly narrowly dilated suitable
with respect to $V=\{ 0\}$).
\end{enumerate}
\end{defn}

\par

In some situations it is convenient to separate the case (3) in Definition
\ref{feasiblpair} from the other ones. Therefore we say
that the pair $(\mascB ,\omega )$ is \emph{narrowly feasible}
(\emph{strongly narrowly feasible}) if it is feasible and satisfies (3) in
Definition \ref{feasiblpair}.

\par

We note that if $\mascB $ fulfills (1) or (2) in Definition \ref{feasiblpair}
and $\omega \in \mascP _E(\rr d)$, then the pair $(\mascB ,
\omega )$ is feasible. If instead $\mascB $ fulfills (3), then Proposition
\ref{weightfamincl} shows that the pair $(\mascB ,
\omega )$ is narrowly feasible.

\par

The following result related to Lemma \ref{smoothweights} shows that for
any weight in $\mathscr P_G(\rr d)$, it is always possible to find a smooth
equivalent weight.

\par

\begin{prop}\label{PGsmoothness}
Let $\omega$  be a weight on $\rr d$ such that \eqref{modrelax} holds.
Then $\omega \in \mathscr P_G(\rr d)$. Furthermore, 
there is a weight $\omega _0\in \mathscr P_G(\rr d) \cap C^\infty (\rr d)$
such that the following is  true:
\begin{enumerate}
\item for every multi-index $\alpha$, there is a constant $C_\alpha$ such that
$$
|\partial ^\alpha \omega _0(x)|\le C_\alpha \omega _0(x)\eabs x^{|\alpha |} \text ;
$$

\vrum

\item \eqref{weightequiv} holds for some constant $C$;

\vrum

\item if in addition $\omega$ is rotation invariant, then $\omega _0$ is rotation invariant.
\end{enumerate}
\end{prop}

\par

In the proof and later on we let $B_r(a)$ denote an open ball in $\rr d$ or $\cc d$
with radius $r\ge 0$ and center at $a$ in $\rr d$ or $\cc d$, respectively.

\par

\begin{proof}
Let $B=B_{c}(0)$, where $c$ be the same as in \eqref{modrelax}, and let $0\le \psi \in
C_0^\infty (B)$ be rotation invariant such that $\int \psi 
\, dx=1$. Also let $\psi _0\in C_0^\infty (3B)$ be rotation
invariant such that $0\le \psi _0\le 1$ and $\psi _0(x)=1$ when $x\in 2B$.
Then it follows by straight-forward computations that (1)--(3) are fulfilled when
$$
\omega _0(x) \equiv \psi _0(x) + (1-\psi _0(x))|x|^d\int \psi (|x|(x-y))\omega (y)\, dy .
$$

\par

It remains to prove that $\omega _0,\omega \in \mascP _G(\rr d)$, and
then it suffices to  prove that $\omega _0$ satisfies \eqref{Gaussest}.
Let $1\le t\in \mathbf R$, $x _0\in \rr d$ be fixed such that $|x _0| 
= 1$, and let
$$
C_1 =\inf _{|x|=1}\omega _0(x)\quad \text{and}\quad C_2=\sup _{|x|=1}\omega _0(x).
$$
Then (1) implies that
$$
\psi (t)=\psi _{x_0}(t) \equiv \omega _0(tx _0),\qquad t\ge 1,
$$
satisfies
$$
\psi '(t)-Ct\psi (t) \le 0\quad \text{and}\quad \psi '(t)+Ct\psi (t) \ge 0,
\qquad 0<C_1\le \psi (1)\le C_2,
$$
for some constant $C>0$ which is independent of $x_0$. This implies that
$$
C_1e^{-C(t^2-1)/2}\le \psi (t)\le C_2e^{C(t^2-1)/2},\qquad t\ge 1.
$$
Since $C_1$ and $C_2$ are independent of the choice of $x_0$ on the unit sphere,
\eqref{Gaussest} follows from the latter inequalities. Hence $\omega _0\in  \mascP _G$, and the proof is complete.
\end{proof}

\par

We finish this section by proving the following result on existence
rotation invariant weights in our families of weights.

\par

\begin{prop}\label{rotinvprop}
Let
%$V\subseteq \rr d$ be a vector space,
%and let
$\mathcal P$ be equal to
\begin{alignat*}{3}
&\mascP (\rr d),&\quad &\mascP _E(\rr d),&\quad &\mascP _{D}^0(\rr d),
\quad \mascP _{G}^0(\rr d),
\\[2ex]
&\mascP _{D}(\rr d),&\quad &\mascP _{Q}^0(\rr d),&\quad &\mascP _{G}(\rr d)
\quad \text{or}\quad \mascP _{Q}(\rr d).
\end{alignat*}
If $\omega \in \mathcal P$, then there are rotation invariant weights $\omega _1,\omega _2\in \mathcal P$ such that
$$
C^{-1}\omega _1\le \omega \le C\omega _2 .
$$
\end{prop}

\par

For the case $\mathcal P = \mascP _{Q}^0$ we need the following lemma.

\par

\begin{lemma}\label{PQGweights}
Let $f\in L^\infty _{loc}(\rr d)$ be such that for each $\ep >0$, there is a
constant $C_\ep >0$ such that
$$
|f(x)|\le C_\ep e^{\ep |x|^2}.
$$
Then there is a rotation invariant $\omega \in \mascP _{Q}^0(\rr d)$
such that $|f|\le \omega$.
\end{lemma}

\par

\begin{proof}
Let
\begin{equation}\label{Gdef}
g\equiv |f|+e\quad \text{and}\quad h_0(x)\equiv \frac {\log g(x)}{|x|^2}.
\end{equation}
Then
\begin{equation}\label{Glim}
\lim _{|x|\to \infty}h(x)=0,
\end{equation}
for $h=h_0$ due to the assumptions.

\par

Now set
$$
h_1(x) =\int \fy (x-y)\left [ \sup _{|y_0|\ge |y|-1}h_0(y_0)\right ]\, dy,\quad |x|\ge 2,
$$
where $0\le \fy \in C_0^\infty (\rr {d})$ is rotation invariant, supported in the unit ball, and 
satisfies $\nm \fy {L^1}=1$. Then it follows that $h_1$ is rotation invariant, smooth and larger than 
$h_0$ in $\Omega _2$, where $\Omega _r \equiv \rr d \setminus B_r(0)$.
Furthermore, \eqref{Glim} holds for 
$h=h_1$, and since $h_0$ is bounded in $\Omega _1$, it follows that $h_1^{(\alpha )}$ is bounded 
in $\Omega _2$ for every multi-index $\alpha$.

\par

Hence, if $|y|\le 1/|x|$ with $x\in \Omega _3$, Taylor's formula gives
\begin{equation*}%\label{Gtayest}
|h_1(x+y)+h_1(x-y)-2h_1(x)|\le C|y|^2\le C/|x|^2.
\end{equation*}
By again using the fact that $h_1$ is bounded in $\Omega _2$, it follows that
$$
\omega (x) \equiv 
\begin{cases}
\sup _{|y|\le 3}|g(y)| & \text{when}\  |x|\le 3,
\\[1ex]
e^{h_1(x)|x|^2} & \text{when}\  |x|\ge 3,
\end{cases}
$$
is rotation invariant, larger than $g$ and fulfills \eqref{modrelax}$'$. This together with \eqref{Glim} shows that $\omega \in \mascP _{Q}^0(\rr {d})$, and the result follows.
\end{proof}

\par

\begin{proof}[Proof of Proposition \ref{rotinvprop}]
If $\mathcal P=\mascP _Q^0$, then the result is an immediate
consequence of Lemma \ref{PQGweights} and the fact
that $\mascP _{Q}^0$ is a group under multiplications.
By straight-forward computations it also follows that
$\omega$ in Lemma \ref{PQGweights} is dilated suitable
or strongly dilated suitable when this is true for $f$. This proves the
statement for $\mathcal P=\mathscr P_D(\rr d)$ and for
$\mathcal P=\mathscr P_D^0(\rr d)$.

\par

For $\mathcal P=\mascP $ we may choose $\omega _j=C_j\eabs \cdo ^{N_j}$ for appropriate constants $C_j$
and $N_j$, and if $\mathcal P=\mascP  _{E}$ we may choose $\omega _j=C_je^{c_j |\cdo |^{s_j}}$ for
appropriate constants $0<s_j\le 1$, $c_j$ and $C_j$. If instead $\mathcal P=\mascP  _{G}$ or
$\mathcal P=\mascP  _{Q}$, then similar is true after the condition on $s_j$ is replaced by $s_j=2$.

\par

It remains to consider the case $\mathcal P= \mascP _G^0$.
Therefore, assume that $\omega \in \mascP _G^0$, and set
$$
\omega _2(x) = \sup _{|z|=|x|}\omega (z)\quad \text{and}\quad \omega _1(x) = \inf _{|z|=|x|}\omega (z).
$$
Now choose $c$ and $C$ such that \eqref{modrelax} holds, and let $x,y$ be such that $|x|\ge 2c$
and $|y|\le c/|x|$. Then for each $\ep >0$, there exists $z\in \rr d$ such that $|z|=|x|$ and
$$
\omega _2(x)\le \omega (z)+\ep .
$$
By \eqref{modrelax} we get
\begin{multline*}
\omega _2(x) \le \omega (z)+\ep \le C\inf _{|y_0|\le c/|x|}\omega (z+y_0)+\ep 
\\[1ex]
\le C\inf _{|y_0|\le c/|x|}\omega _2(x+y_0)+\ep \le C\omega _2(x+y)+\ep .
\end{multline*}
This proves that $\omega _2(x)\le C\omega _2(x+y)$. In the same way it follows that
$\omega _2(x+y)\le C\omega _2(x)$. Since $\omega _2$ fulfills similar types of estimates
as $\omega$, it follows that $\omega _2$ is subgaussian. In the same way it follows
that $\omega _1$ is subgaussian, and the result follows for $\mathcal P=\mascP _{G}^0$.
%
%\par
%
%The case $\mathcal P=\mascP _{D,V}^0$ follows by similar arguments
%and is left for the reader. The proof is complete.
\end{proof}

\par

\subsection{Examples} Next we give some examples on weights in $\mascP _D^0(\rr d)$.  
First we note that any weight of the form $\sigma _s$ and $\omega _1$ in Example \ref
{weightexampl1} belongs to $\mathscr P(\rr d)$ in view of Proposition
\ref{weightfamincl}. In order to give other examples it is convenient to
consider corresponding logarithmic 
conditions on those weights. We note that if $\omega$ is a weight on
$\rr d$ and $\fy (x) =\log \omega (x)$, then \eqref{Gaussest} is equivalent to
\begin{equation}\label{equivGauss}
|\fy (x)|\le C+\ep |x|^2,
\end{equation}
for some positive constants $C$ and $\ep$. The conditions \eqref{modrelax} and
\eqref{modrelax}$'$ are the same 
as
\begin{equation}\label{equivmodrel}
|\fy (x+y)-\fy (x)|\le C,\qquad |x|\ge 2c,\quad |y|\le c/|x| ,
\end{equation}
and
\begin{equation}\tag*{(\ref{equivmodrel})$'$}
|\fy (x+y)+\fy (x-y)-2\fy (x)|\le C,\qquad |x|\ge 2c,\quad |y|\le c/|x| ,
\end{equation}
respectively, for some positive constants $c$ and $C$.
Finally, $\omega$ and $1/\omega$ are dilated suitable 
with respect to $\ep \in (0,1]$, if and only if there are constants $C_\ep >0$ and $
\theta _\ep \in (0,1)$ such that
\begin{equation}\label{equivaddomcond}
|\fy (\lambda x)-\fy (x)|\le C_\ep +\ep (1-\lambda ^2)|x|^2,\quad 1-\theta _\ep<\lambda <1,
\end{equation}
and $\omega$ and $1/\omega$ are narrowly dilated suitable 
with respect to $\ep \in (0,1]$ and $V=\{ 0\}$ when
\begin{equation}\label{equivaddomcond2}
\limsup _{\lambda \to 1-}\Big (|\fy (\lambda x)-\fy (x)| -\ep (1-\lambda ^2)|x|^2\Big )=0.
\end{equation}

\par

In particular, the following lemma is an immediate consequence of the definitions.

\par

\begin{lemma}\label{logweightslemma}
Let $\omega$ be a weight on $\rr d$, and let $\fy (x)=\log \omega (x)$. Then the following is 
true:
\begin{enumerate}
\item $\omega \in \mathscr P_G^0(\rr d)$, if and only if \eqref{equivmodrel}
holds for some positive constants $c$ and $C$, and for every $\ep >0$,
there is a positive constant $C$ such that \eqref{equivGauss} holds:

\vrum

\item $\omega \in \mathscr P_Q^0(\rr d)$, if and only if \eqref{equivmodrel}$'$
holds for some positive constants $c$ and $C$, and for every $\ep >0$,
there is a positive constant $C$ such that \eqref{equivGauss} holds:

\vrum

\item $\omega \in \mathscr P_D^0(\rr d)$, if and only if \eqref{equivmodrel}$'$
holds for some positive constants $c$ and $C$, and for every $\ep >0$,
there are positive constants $C$, $C_\ep$ and $\theta _\ep \in (0,1)$
such that \eqref{equivGauss}, \eqref{equivaddomcond} and
\eqref{equivaddomcond2} hold.
\end{enumerate}
\end{lemma}

\par

\begin{example}
Let $\omega (x)=\eabs x^{t\eabs x}$, $x\in \rr d$, for some choice of $t\in \mathbf R$. Then it 
follows by straight-forward computations that $\fy (x) = \log \fy (x) =\eabs x\log \eabs x$ 
satisfies \eqref{equivGauss}--\eqref{equivaddomcond2}. Hence $\omega \in \mathscr P_D^0
(\rr d)\cap \mathscr P_G^0
(\rr d)$.
\end{example}

\par

\begin{example}
Let $\omega (x)=\Gamma (\eabs x+1+r)$ and $\fy (x)=\log \omega (x)$, where $\Gamma$ is the gamma function, and $r>-2$ is real. Then we have
$$
\omega (x)=(2\pi (\eabs x +r))^{1/2}\left ( \frac {\eabs x +r}{e}\right )^{\eabs x+r}(1+ o (\eabs x^{-1})),
$$
by Stirling's formula. This is gives
\begin{equation}\label{fyloggamma}
\fy (x) = \frac 12\log (2\pi )+ \frac {1+2r}2 \log \eabs x +\eabs x(\log \eabs x-1) + \psi (x),
\end{equation}
where $\psi (x)$ is continuous and satisfies
\begin{equation}\label{xrlimit0}
\lim _{|x|\to \infty}\eabs x \psi (x)=0.
\end{equation}

\par

By straight-forward computations it follows that the first three terms in
\eqref{fyloggamma} satisfy the conditions
\eqref{equivGauss}--\eqref{equivaddomcond2}. Furthermore, the condition
\eqref{xrlimit0} together with the proof of Proposition \ref{weightfamincl}
show that also $\psi (x)$ satisfies \eqref{equivGauss}--\eqref{equivaddomcond2}. 
Consequently, $\omega \in \mathscr P_D^0(\rr d)\cap \mathscr P_G^0
(\rr d)$.
\end{example}

\par

The following result shows that there are weights in $(\mascP _D^0\cap \mathscr P_G^0 )\setminus \mascP$
which fulfills \eqref{conseqmoder} for some polynomial $v$.

\par

\begin{prop}
Let $\omega (x,\xi )=(1+|x|^r|\xi |^r)^s$ for some $r>0$ and $s\in \mathbf R\setminus 0$.
Then $\omega$ belongs to $(\mascP _D^0(\rr {2d})\cap \mathscr P_G^0
(\rr {2d}) )\setminus
\mascP (\rr {2d})$ and fulfills \eqref{conseqmoder} for some polynomial $v$.
\end{prop}

\par

\begin{proof}
Since $\mascP _G^0$, $\mascP _D^0$ and $\mascP$ are groups under multiplications,
and invariant under mappings $\omega \mapsto \omega ^t$ for $t\neq 0$,
and that $\omega$ is equivalent to
$$
\omega _0(x,\xi ) = (1+|x|^{1/2}|\xi |^{1/2})^{2rs},
$$
we may assume that $r=1/2$ and $s=1$.

\par

It is obvious that $\omega$ fulfills \eqref{conseqmoder} for some
polynomial $v$, and by straight-forward computations it also
follows that \eqref{modrelax} is fulfilled. Furthermore, by choosing
$x,y,\xi , \eta \in \rr d$ in such way that $|\eta |=1/|y|$ and $|x|=|\xi |=1$,
it follows that
$$
\sup _{y,\eta}\frac {\omega (x+y,\xi +\eta )}{\omega (y,\eta )} \ge
\sup _{|\eta |=1/|y|} \frac {1+|x+y|^{1/2}|\xi +\eta|^{1/2}}2 =\infty .
$$
This implies that $\omega \notin \mascP$.

\par

It remains to prove that $\omega \in \mascP _D^0$, which follows if we prove that for
every $\ep >$, then \eqref{addomegcond2} holds for $V=\{ 0\}$, after
$\rr d$ and $x$ have been replaced by $\rr {2d}$ and $(x,\xi )$, respectively.

\par

Here it is obvious that \eqref{addomegcond2} is true with
$\omega =\omega _0$, since
$$
\frac {\omega (\lambda x,\lambda \xi)e^{- \ep (1-\lambda ^2)(|x|^2+|\xi ^2)|/2}}{\omega (x,\xi )}\le 1
$$
with equality when $x=\xi =0$. Let
$$
h(t_1,t_2) = \frac {(1+\sqrt {t_1t_2})e^{-\ep (1-\lambda )(t_1^2+t_2^2)/2}}{1+\lambda
\sqrt {t_1t_2}},\qquad t_1,t_2\ge 0.
$$
Since
\begin{equation*}
h(|x|,|\xi |) = \frac {\omega (x,\xi )e^{- \ep (1-\lambda )(|x|^2+|\xi |^2)/2}}
{\omega (\lambda x,\lambda \xi)}
\ge
\frac {\omega (x,\xi )e^{- \ep (1-\lambda ^2)(|x|^2+|\xi |^2)/2}}
{\omega (\lambda x,\lambda \xi)},
\end{equation*}
and $h(0,0)=1$, the result follows if we prove
$$
\lim _{\lambda \to 1-}\big (\sup _{t_1,t_2\ge 0}h(t_1,t_2)\big ) =1.
$$

\par

Now, by straight-forward computations it follows that
$$
\sup _{t_1,t_2\ge 0}h(t_1,t_2) = h(t_0,t_0),\quad
\text{where}\quad t_0=-\frac {1+\lambda}{2\lambda} +
\sqrt {\displaystyle{\left (\frac {1+\lambda }{2\lambda}\right )^2+
\frac {1-\ep}{\ep \lambda}}},
$$
and it is straight-forward to control that $h(t_0,t_0)\to 1$
as $\lambda \to 1-$. The proof is complete.
\end{proof}

\par

%%%%%%%%%%%%%%%%%%%%%%%
\section{Harmonic estimates and mapping properties for the Bargmann
transform on modulation spaces}\label{sec3}
%%%%%%%%%%%%%%%%%%%%%%%

\par

In the first part of the section we establish certain invariance properties
of spaces of harmonic or analytic functions.
Thereafter we apply these properties to prove that $A_0(\omega ,\mascB) = A(\omega ,\mascB)$,
for appropriate weights $\omega$. In the end of the section we use these results to prove general properties
of $A(\omega ,\mascB)$ and $M(\omega ,\mascB)$, for example that they
are Banach spaces when $\mascB$ is a mixed norm space.

\par

\subsection{Analytic and harmonic estimates}

\par

Let $\Omega$ be an admissible
family of weights on $\rr d$, and let $\mascB $
be a mixed quasi-norm space on $\rr d$. (Cf. Definition \ref{defweightfam}.)
Then we prove that subsets of
\begin{align*}
E_1(\Omega ,\mascB ) &\equiv \sets {f\in L^r_{loc}(\rr d)}{f\omega \in \mascB
\ \text{for every}\ \omega \in \Omega}
\intertext{and}
E_2(\Omega ,\mascB ) &\equiv \sets {f\in L^r_{loc}(\rr d)}{f\omega \in \mascB
\ \text{for some}\ \omega \in \Omega},\quad r=\nu _1(\mascB ),
\end{align*}
of analytic or harmonic functions are independent of the choice of $\mascB$.
Recall here that $\nu _1(\mascB )$ is the
smallest involved Lebesgue exponent in $\mascB$, and belongs to
$(0,\infty ]$. Also recall that $\mascB$ is a mixed norm space, if and
only if $\nu _1(\mascB )\ge 1$.
Some restrictions are needed when considering subsets of harmonic functions.

\par

It is easy to prove one direction. In fact, by the definitions and H{\"o}lder's inequality we get
\begin{multline}\label{Lebincl}
E_j(\Omega ,L^\infty (\rr d))\subseteq E_j(\Omega ,\mascB )
\subseteq E_j(\Omega ,L^r (\rr d)),
\\[1ex]
\text{where}\ j=1,2\ \text{and} \ r=\min (1,\nu _1(\mascB ))
\end{multline}
(with continuous inclusions).

\par

In order to establish opposite inclusions to \eqref{Lebincl}, for corresponding subsets of analytic
or harmonic functions, we will use techniques based on harmonic estimates for such functions.

\par

We start with the following result. Here and in what follows we let
$\mathcal H(\rr d)$ be the set of harmonic functions on $\rr d$.

\par

\begin{prop}\label{harmident}
Let $\Omega \subseteq \mascP _G(\rr d)$ be an admissible family of weights on $\rr d$,
and let $\mascB _1$ and $\mascB _2$ be mixed norm space on $\rr d$. Then
\begin{equation}\label{harmidentform}
E_j(\Omega ,\mascB _1) \bigcap \mathcal H(\rr d)= E_j(\Omega ,\mascB _2)
\bigcap \mathcal H(\rr d),\quad j=1,2.
\end{equation}
\end{prop}

\par

\begin{proof}
It suffices to prove $E_j(\Omega ,L^1) \cap \mathcal H\subseteq  E_j(\Omega ,L^\infty )\cap
\mathcal H$, in view of \eqref{Lebincl}.

\par

Assume that $f\in E_2(\Omega ,L^1)
\cap \mathcal H$. By \eqref{omega0est} and the assumptions we have
\begin{equation}\label{fL1bound}
\nm {f}{L^1_{(\omega )}}<\infty \quad \text{and}\quad \omega _1\le C\eabs \cdo ^{-d}\omega ,
\end{equation}
for some $\omega ,\omega _1\in \Omega$ and some constant $C>0$. Let $c$ and $C$
be the same as in \eqref{modrelax}. Then the result follows if we prove
\begin{equation}\label{L2normest}
\nm {f\, \chi }{L^\infty _{(\omega _1)}}<\infty ,
\end{equation}
when $\chi$ is the characteristic function of $\sets {x\in \rr d}{|x|\ge 2c}$.

\par

Since $f\in \mathcal H(\rr d)$, the mean-value property for harmonic functions gives
$$
f(x)=c_d^{-1}(|x|/c)^{d}\int _{|y|\le c/|x|}f(x+y)\, dy,
$$
where $c_d$ is the volume of the $d$-dimensional unit ball. If $|x|\ge 2c$,
then $\eabs x\le C_1|x|$ for some constant $C_1>0$, and \eqref{modrelax}
and \eqref{fL1bound} give
\begin{multline*}
|f(x)\omega _1(x)| = \left | c_d^{-1}(|x|/c)^{d}\omega _1(x)\int _{|y|\le c/|x|}f(x+y)\, dy  \right |
\\[1ex]
\le C_1\int _{|y|\le c/|x|}|f(x+y)\omega (x)|\, dy
\\[1ex]
\le  C_2\int _{|y|\le c/|x|}|f(x+y)\omega (x+y)|\, dy \le C_2\nm {f}{L^1_{(\omega )}}<\infty ,
\end{multline*}
for some constants $C_1$ and $C_2$. This gives \eqref{harmidentform}
for $j=2$. By similar arguments, \eqref{harmidentform} also follows for
$j=1$. The details are left for the reader. The proof is complete.
\end{proof}

\par

Next we discuss similar questions for spaces of analytic functions, i.{\,}e.
we present sufficient conditions in order for the identity
\begin{equation}\label{analidentform}
E_j(\Omega ,\mascB _1) \cap A(\cc d)= E_j(\Omega ,\mascB _2)\cap A(\cc d),\quad j=1,2.
\end{equation}
should hold.

\par

\begin{thm}\label{analident}
Let $\Omega$ be an admissible family of weights on $\cc d\simeq \rr {2d}$,
and let $\mascB _1$ and $\mascB _2$ be mixed quasi-norm spaces on
$\cc d$. Then \eqref{analidentform} holds.

\par

Furthermore, for every fixed $\omega \in \Omega$, there are
$\omega _1,\omega _2\in \Omega$ and constant $C>0$ such that
\begin{equation}\label{normembds}
C^{-1}\nm F{A(\omega _1,\mascB _2)} \le \nm F{A(\omega ,\mascB _1)}\le
C\nm F{A(\omega _2,\mascB _2)},\qquad F\in A(\cc d).
\end{equation}
\end{thm}

\par

For the proof we need the following lemma, concerning mean-value
properties for analytic functions.

\par

\begin{lemma}\label{analmeans}
Let $\nu$ be a positive Borel measure on $\cc d$ which is rotation invariant under
each coordinate $z_1,\dots ,z_d\in \mathbf C$, $T _1,\dots ,T _n$ be (complex)
$d\times d$-matrices, and let $F_1,\dots ,F_n\in A(\cc d)$. Also let $r>0$, and let $\Omega
\subseteq \cc d$ be compact and convex, which is rotation invariant under
each coordinate $z_1,\dots ,z_d\in \mathbf C$. Then
\begin{align}
\prod _{j=1}^nF_j(z) &= \frac 1{\nu (\Omega )} \int _{\Omega}
\prod _{j=1}^nF_j(z+T_jw)\, d\nu (w)\label{analmv}
\intertext{and}
\prod _{j=1}^n |F_j(z) |^r &\le  \frac 1{\nu (\Omega )}  \int _{\Omega} \prod _{j=1}^n
| F_j(z+T_jw) | ^r\, d\nu (w).\label{modanalmv}
\end{align}
\end{lemma}

\par

\begin{proof}
Let $G(z,w)= \prod _{j=1}^nF_j(z+T_jw)$. Then $w\mapsto G(z,w)$ is analytic. By the mean-value property
for harmonic functions we get
$$
G(z,0) =\frac 1{\nu (\Omega )} \int _{\Omega} G(z,w)\, d\nu (w),
$$
which is the same as \eqref{analmv}.

\par

From the same analyticity property it follows that the map $w\mapsto |G(z,w)|^r$
is subharmonic, in view of \cite[Corollary 2.1.15]{Kr}. Hence, by \cite[Theorem 2.1.4]{Kr} we get
$$
|G(z,0)|^r \le \frac 1{\nu (\Omega )} \int _{\Omega} |G(z,w)|^r\, d\nu (w),
$$
which is the same as \eqref{modanalmv}, and the proof is complete.
\end{proof}

\par

\begin{proof}[Proof of Theorem \ref{analident}]
We only prove \eqref{analidentform} for $j=2$, leaving the small modifications of the case $j=1$ for the reader. By \eqref{Lebincl} it suffices to prove
$$
E_2(\Omega ,L^r)\cap A(\cc d) \subseteq E_2(\Omega ,L^\infty)\cap A(\cc d),\qquad r=\nu _1(\mascB ).
$$

\par

Therefore, assume that $F\in E(\Omega ,L^r)\cap A(\cc d)$.
We shall mainly follow the ideas in Proposition \ref{harmident}.
By \eqref{omega0est} and the assumptions we have
$$
\nm {F}{L^r_{(\omega )}}<\infty \quad \text{and}\quad \omega _1\le C\eabs \cdo ^{-2d}\omega ,
$$
for some $\omega ,\omega _1\in \Omega$ and some constant $C>0$. Let $c$ and $C$ be
the same as in \eqref{modrelax}$'$. Then the result follows if we prove that \eqref{L2normest} holds
when $\chi$ is the characteristic function of $\sets {z\in \cc d}{|z|\ge 2c}$.

\par

By Lemma \ref{analmeans} and Cauchy-Schwartz inequality we get
\begin{multline*}
|F(z)\omega _1(z)|^r \le  c_{2d}^{-1}(|z|/c)^{d}\omega _1(z)^r
\int _{|w|\le c/|z|}|F(z+w)F(z-w)|^{r/2}\, d\lambda (w)
\\[1ex]
\le C_1\int _{|w|\le c/|z|}|F(z+w)F(z-w)\omega (z)^2|^{r/2}\, d\lambda (w)
\\[1ex]
\le  C_2\int _{|w|\le c/|z|}|F(z+w)\omega (z+w)|^{r/2} |F(z-w)
\omega (z-w)|^{r/2}\, d\lambda (w)
\\[1ex]
\le C_2\int _{\cc d}|F(z+w)\omega (z+w)|^{r/2} |F(z-w)
\omega (z-w)|^{r/2}\, d\lambda (w) 
\\[1ex]
\le C_2\left ( \int _{\cc d}|F(z+w)\omega (z+w)|^r\, \lambda (w)
\right )^{1/2}\cdot \left (
\int _{\cc d}|F(z-w)\omega (z-w)|^r\, \lambda (w)\right )^{1/2}
\\[1ex]
= C_2\nm {F}{L^r_{(\omega )}}^r<\infty ,
\end{multline*}
for some constants $C_1$ and $C_2$. Here recall that $d\lambda (z)$
is the Lebesgue measure on $\cc d$. This proves \eqref{analidentform}
for $j=2$, and \eqref{normembds}. By similar arguments,
\eqref{analidentform} follows for $j=1$. The details are left for the reader, and
the proof is complete.
\end{proof}

\par

\subsection{Mapping properties for the Bargmann transform
on modulation spaces}

\par

Next we prove that $A_0(\omega ,\mascB )$ is equal to
$A(\omega ,\mascB )$ for every choice of $\omega$ in $\mascP  _{Q}^0$
and mixed quasi-norm space $\mascB $.

\par

\begin{thm}\label{mainthm}
Let $\mascB $ be a mixed quasi-norm space on $\rr {2d}\simeq \cc d$
and let $\omega \in \mascP  _{Q}^0(\cc {d})$. Then $A_0(\omega ,\mascB )
=A(\omega ,\mascB )$, and the map $f\mapsto \mathfrak Vf$ from
$M(\omega ,\mascB )$ to $A(\omega ,\mascB )$ is isometric and
bijective.
\end{thm}

\par

We need some preparations for the proof, and start to show that
$M(\omega ,\mascB)$ is a Banach space when $\nu _1(\mascB )\ge 1$,
which is a consequence of the following result. Here, for each
$\psi \in \maclS _{1/2}(\rr d)\setminus 0$, $0<\omega \in
L^\infty _{loc}(\rr {2d})$ and mixed norm space $\mascB$ on
$\rr {2d}$, we let $M_\psi (\omega ,\mascB )$ be the set of all
$f\in (\maclS _{1/2})'(\rr d)$ such that
$$
\nm f{M_\psi (\omega ,\mascB )}\equiv \nm {V_\psi f \cdot \omega}{\mascB}<\infty .
$$

\par

\begin{prop}\label{modbanach}
Let $\psi \in \maclS _{1/2}(\rr d)\setminus 0$ and $\omega \in L^\infty _{loc} (\rr {2d})$
be such that for every $\ep >0$ there is a constant $C_\ep >0$ such that
\begin{equation}\label{omegaussbound}
\omega \ge C_\ep ^{-1}e^{-\ep |\cdo |^2}.%\le \omega \le C_\ep e^{\ep |\cdo |^2}.
\end{equation}
Also let  $\mascB$ be a mixed norm space on $\rr {2d}$. Then
$M_\psi (\omega ,\mascB )$ is a Banach space.
\end{prop}

\par

\begin{proof}
We may assume that $\nm \psi {L^2}=1$, and start by proving that
$M_\psi (\omega ,\mascB)$ is continuously embedded in
$(\maclS _{1/2 ,\ep})'(\rr d)$, for every choice of $\ep >0$.
We have $\psi \in \maclS _{1/2,\ep _0}$, for some $\ep _0>0$.

\par

Therefore let $\ep >0$ be arbitrary. By a straight-forward combination of
Lemma \ref{GSFourierest} and Proposition \ref{stftGelfand1} and their proofs it follows that
for some $\delta >0$ we have
\begin{equation}\label{Vpsifyest}
\nm {V_\psi \fy}{\maclS _{1/2,\delta}}\le C_\ep \nm \fy{\maclS _{1/2,\ep}}
\nm \psi {\maclS _{1/2,\ep _0}},
\end{equation}
for some constant $C_\ep>0$, which only depends on $\ep$. (Cf. \cite{CPRT10}.)

\par

Now we recall that
$$
(f,\fy )_{L^2(\rr d)} \equiv (V_\psi f,V_\psi \fy )_{L^2(\rr {2d})},
\quad \fy \in \maclS _{1/2,\ep}(\rr d),
$$
for any $f\in M_\psi (\omega ,\mascB )$. Hence, \eqref{omegaussbound},
\eqref{Vpsifyest} and H{\"o}lder's inequality give
\begin{multline*}
|(f,\fy )_{L^2}| = |(V_\psi f,V_\psi \fy )_{L^2}|
\le
\nm f{M_\psi (\omega ,\mascB )}\nm {V_\psi \fy /\omega}{\mascB '}
\\[1ex]
\le
C_{1} \nm f{M_\psi (\omega ,\mascB )}
\nm {V_\psi \fy \cdot e^{\delta |\cdo |^2}\cdot \eabs \cdo ^{-2d-1}}{\mascB '}
\\[1ex]
\le
C_{2} \nm f{M_\psi (\omega ,\mascB )}
\nm {V_\psi \fy \cdot e^{\delta |\cdo |^2}}{L^\infty}\nm {\eabs \cdo ^{-2d-1}}{\mascB '}
\le
C_{3} \nm f{M_\psi (\omega ,\mascB )}
\nm {V_\psi \fy }{\maclS _{1/2,\delta}},
\\[1ex]
\le
C_4 \nm f{M_\psi (\omega ,\mascB )} \nm \fy {\maclS _{1/2 ,\ep}},
\end{multline*}
for some constants $C_1,\dots ,C_4>0$. Consequently, $M_\psi (\omega ,\mascB )$
is continuously embedded in $(\maclS _{1/2 ,\ep})'(\rr d)$ and in
$(\maclS _{1/2})'(\rr d)$ for every $\ep >0$.

\par

Now let $\{ f_j \} _{j=1}^\infty$ be a Cauchy sequence in $M_\psi (\omega ,\mascB )$.
Since $(\maclS _{1/2 ,\ep})'(\rr d)$ decreases when $\ep$ decreases, \eqref{GSspacecond1}$'$ in combination with the previous embedding properties show that there is a element $f\in (\maclS _{1/2})'(\rr d)$ (which is independent of $\ep$) such that $f_j\to f$ in $(\maclS _{1/2 ,\ep})'$ as $j\to \infty$.

\par

Since $V_\psi f_j \to V_\psi f$ pointwise, it now follows that $f\in M_\psi (\omega ,\mascB )$ by Fatou's lemma. This shows that $M_\psi (\omega ,\mascB )$ is a Banach space, and the proof is complete.
\end{proof}

\par

Next we consider the case when $\mascB =L^2$ and $\omega $ is rotation invariant
in each coordinate. In this case we have the following.

\par

\begin{lemma}\label{basislemma}
Let $\omega \in \mascP  _{Q}^0(\rr {2d})$ be rotation invariant. If
$$
c_{\alpha}=\nm {z^\alpha}{A^2_{(\omega)}}^{-1},\quad \text{then}
\quad \{ c_\alpha z^\alpha  \} _{\alpha \in \mathbf N^d}
$$
is an orthonormal basis for $A^2_{(\omega)}(\cc d)$.
\end{lemma}

\par

In the proof of Lemma \ref{basislemma} and in several other situations later on,
we encounter the integral
\begin{equation}\label{I0def}
\int _{\Delta _d}e^{i\scal n\theta}\, d\theta
=
\begin{cases}
(2\pi )^d, & n =0,
\\[1ex]
0, & n\neq 0,
\end{cases}
\quad
\text{when} \quad
\Delta _d=[0,2\pi ]^d\quad \text{and}\quad n\in \mathbf Z^d .
\end{equation}

\par

\begin{proof}
First we prove that the scalar product $(z^\alpha ,z^\beta)_{A^2_{(\omega )}}$
is zero when $\alpha \neq \beta$. By polar coordinates we have
$$
z=(r_1e^{i\theta _1},\dots ,r_de^{i\theta _d}),
$$
where
$$
r=(r_1,\dots ,r_d)\in [0,\infty )^d,\quad \text{and}\quad \theta =
(\theta _1,\dots ,\theta _d)\in \Delta _d .
$$
Furthermore it follows from the assumptions that $\omega (z)
=\omega _0(2^{-1/2}r)$, for some positive function $\omega _0$
on $[0,\infty )^d$. Hence \eqref{BLpnorm} and \eqref{I0def} give
$$
(z^\alpha ,z^\beta )_{A^2_{(\omega )}} = \pi ^{-d}I_{(\omega _0)}
(\alpha +\beta )\cdot \int _{\Delta _d}e^{i\scal {\alpha -\beta}\theta}
\, d\theta = 2^dI_{(\omega _0)}(2\alpha )\delta _{\alpha ,\beta}, 
$$
where
\begin{equation*}%\label{I1def}
I_{(\omega _0)} (\alpha ) = \int _{[0,\infty )^d} \omega _0(r)r^{\alpha }
e^{-|r|^2}r_1\cdots r_d\, dr > 0,\qquad \alpha \in \mathbf N^d,
\end{equation*}
and $\delta _{\alpha ,\beta}$ is the Kronecker's delta function. In particular,
$(z^\alpha ,z^\beta )_{A^2_{(\omega )}}=0$, if and only if $\alpha \neq \beta$, and
we have proved that $\{ c_\alpha z^\alpha  \}$ is an orthonormal system for
$A^2_{(\omega)}$.

\par

It remains to prove that the set of linear combinations of
$c_\alpha z^\alpha$ spans $A^2_{(\omega)}$. By Hahn-Banach's
theorem, it suffices to prove that if $F\in A^2_{(\omega )}$, and
\begin{equation}\label{fzalpha}
(F,z^\alpha )_{A^2_{(\omega )}} =0\quad \text{for all}\ \alpha \in \mathbf N^d,
\end{equation}
implies that $F=0$.

\par

Therefore assume that \eqref{fzalpha} holds. Since $F$ is entire, it follows that
its Taylor series expansion
\begin{equation}\label{fTaylor}
F(z) =\sum _\beta a_\beta z^\beta ,\qquad a_\beta =\frac {F^{(\beta )}(0)}{\beta !}
\end{equation}
is locally uniformly convergent, and that
\begin{equation}\label{perfunct2}
\sum _\beta |a_\beta r^\beta| <\infty
\end{equation}
holds. Hence
%
%. Hence, if $D_R$ is the poly-disc
%$$
%\sets {z\in \cc d}{|z_j|\le R,\ j=1,\dots ,d},
%$$
\eqref{I0def} gives
\begin{multline*}
0=(F,z^{\alpha})_{A^2_{(\omega )}} =  \int _{\cc d}\left (
\sum _\beta a_\beta z^\beta \right ) \overline z^\alpha \omega (2^{1/2}\overline z )
\, d\mu (z)
\\[1ex]
=\pi ^{-d}\int _{[0,\infty )^d}\left ( \int _{\Delta _d} \left (\sum _\beta a_\beta
r^{\alpha +\beta}e^{i\scal {\beta -\alpha}\theta } \right ) \, d\theta \right )
\omega _0(r)e^{-|r|^2}r_1\cdots r_d\, dr
\\[1ex]
=\pi ^{-d}\int _{[0,\infty )^d} \left (\sum _\beta a_\beta r^{\alpha +\beta}
\left ( \int _{\Delta _d}e^{i\scal {\beta -\alpha}\theta } \, d\theta  \right ) \right )
\omega _0(r)e^{-|r|^2}r_1\cdots r_d\, dr
\\[1ex]
= 2^dI_{(\omega _0)}(2\alpha ) a_\alpha .
\end{multline*}
Since $I_{(\omega _0)}(2\alpha )>0$, we get $a_\alpha =0$ for every $\alpha$.
Consequently, $F$ is identically zero, and the proof is complete.
\end{proof}

\par

We may now prove Theorem \ref{mainthm} in the important special case that $\mascB =L^2$ and $\omega$ is rotation invariant in every coordinate. 

\par

\begin{prop}\label{mainthmspecial}
If $\omega \in \mascP  _{Q}^0(\cc {d})$ is rotation invariant in each coordinate, then $A_0
(\omega ,L^2)=A(\omega ,L^2)$.
\end{prop}

\par

\begin{proof}
We use the same notations as in the proof of Lemma \ref{basislemma}.
The image under the Bargmann transform $\mathfrak V$ of the hermite
function $h_\alpha$ is $z^\alpha /(\alpha !)^{1/2}$. Since $M^2_{(\omega )}$
is a Banach space in view of Proposition \ref{modbanach}, and  $\mathfrak V$
is isometric and injective from $M^2_{(\omega )}$ to $A^2_{(\omega )}$, it
follows from Lemma \ref{basislemma} that
$$
\{ (\alpha !)^{1/2}c_\alpha h_\alpha \} _{\alpha \in \mathbf N^d}
$$
is an orthonormal basis of $M^2_{(\omega )}$, and that $\mathfrak V$ is
bijective from  $M^2_{(\omega )}$ to $A^2_{(\omega )}$.
\end{proof}

\par

\begin{rem}
Lemma \ref{basislemma} and Proposition \ref{mainthmspecial} give equalities
between weighted $l^p$-norms of the Taylor coefficients and weighted $L^p$
norm of corresponding entire functions, when the involved weights are rotation
invariant and $p=2$. In general it is difficult to find such equalities between
coefficients and functions in other situations when the weights are not rotation
invariant, or $p$ is not equal to $2$. We refer to \cite{BiGa} for positive results
in this directions.
\end{rem}

\par

\begin{proof}[Proof of Theorem \ref{mainthm}]
By Proposition \ref{mainstep1prop} it follows that the map $f\mapsto
\mathfrak Vf$ is an isometric injective map from $M(\omega ,\mascB )$
to $A(\omega ,\mascB )$. We have to show that this map is
surjective.

\par

Therefore assume that $F\in A(\omega ,\mascB )$. Since $M(\omega _1,\mascB )\subseteq M(\omega _2,\mascB )$, as $\omega _2\le C\omega _1$, it follows from Theorem \ref{analident} and Proposition \ref{mainthmspecial} that
there is an element $f\in M^2_{(\omega _1)}(\rr d)\subseteq \mathcal S'_{1/2}(\rr d)$ such that $F=\mathfrak
Vf$, for some $\omega _1$. We have
$$
\nm f{M(\omega ,\mascB )} =\nm {\mathfrak Vf}{A(\omega ,\mascB )} =\nm F{A(\omega ,\mascB )}<\infty .
$$
Hence, $f\in M(\omega ,\mascB )$, and the result follows. The proof
is complete.
\end{proof}

\par

As a consequence of Theorems \ref{analident} (2) and \ref{mainthm},
we may identify the space $\Upsilon (\rr d)$ in Remark \ref{Upsilspace}
with unions of modulation spaces. More precisely we have the following.

\par

\begin{prop}\label{capcupmodsp}
Let $\mascB $ be a mixed quasi-norm space on $\rr {2d}$. Then
\begin{equation*}%\label{cuprelation}
\bigcup M(\omega ,\mascB )=\Upsilon (\rr d),
\end{equation*}
where the union is taken over all rotation invariant $\omega \in \mascP _{Q}^0(\rr {2d})$.

\par

Furthermore, if $0< \gamma <2$, $0<\ep < \min (\gamma ,2-\gamma )$ and
$\omega (x,\xi ) =Ce^{c(|x|^{\gamma}+|\xi |^{\gamma})}$, for some constants
$C>0$ and $c\in \mathbf R$ which are independent of $x,\xi \in \rr d$,
then the following is true:
\begin{enumerate}
\item if $c>0$, then $\maclS _{1/(\gamma +\ep )}(\rr d)\subseteq
M(\omega ,\mascB )\subseteq \maclS _{1/(\gamma -\ep )}(\rr d)$;

\vrum

\item if $c<0$, then $\maclS _{1/(\gamma -\ep )}'(\rr d)\subseteq
M(\omega ,\mascB )\subseteq \maclS _{1/(\gamma +\ep )}'(\rr d)$.
\end{enumerate}
\end{prop}

\par

\begin{proof}
By Theorems \ref{analident} (2) and \ref{mainthm}, we may assume that $\mascB =L^\infty$.
By the assumptions on $\omega$ we get $M^\infty _{(\omega )}\subseteq \Upsilon$.

\par

Let $f\in \Upsilon (\rr d)$. Then it follows from Proposition \ref{stftGelfand2}, Lemma
\ref{PQGweights} and the fact that $\mascP _{Q}^0$ is a group under multiplications that
$$
|V_\phi f| \le 1/\omega
$$
for some $\omega \in \mascP _{Q}^0(\rr {2d})$, where $\phi(x)=\pi ^{-d/4}e^{-|x|^2/2}$. 
This is the same as $f\in M^\infty _{(\omega )}$, and the first part follows.

\par

The assertions (1) and (2) follow by similar arguments in combination
of Proposition \ref{stftGelfand2} and are left for the reader. The proof is complete.
\end{proof}

\par

%%%%%%%%%%%%%%%%%%%%%%%%%%%%%
\section{Basic properties for spaces of analytic functions
and modulation spaces}\label{sec4}
%%%%%%%%%%%%%%%%%%%%%%%%%%%%%

\par

In this section we establish basic properties for the spaces
$A(\omega ,\mascB )$ and $M(\omega ,\mascB )$ when $\omega$
is an appropriate weight and $\mascB $ is a mixed quasi-norm
space. In view of Theorem \ref{mainthm} any property of
$A(\omega ,\mascB )$ carry over to $M(\omega ,\mascB )$,
and vice versa. We start by proving that these spaces are
quasi-Banach spaces. Then we prove that if $\nu _2(\mascB )
<\infty$, then $P(\cc d)$ is dense in $A(\omega ,\mascB )$,
and that the dual of $A(\omega ,\mascB )$ can be identified
with $A(1/\omega ,\mascB ')$ through a unique extension of
the $A^2$ form on $P(\cc d)$. A straight-forward consequence
of the latter results is that $P(\cc d)$ is dense in $A(\omega ,\mascB )$
with respect to the weak$^*$-topology, when $\nu _1(\mascB )>1$.
(Recall Subsection \ref{subsec1.4} for the definitions of $\nu _
(\mascB )$ and $\nu _2(\mascB _2)$.) Thereafter we introduce
the concept of narrow convergence to get convenient density
properties for certain $\mascB $ with $\nu _1(\mascB )=1$ and
$\nu _2(\mascB )=\infty$. Finally we formulate corresponding
results for modulation spaces.

\par

A cornerstone of these investigations concerns the projection operator
\begin{equation}\label{A2projection}
(\Pi _AF)(z) =\int F(w)e^{(z,w)}\, d\mu  (w),
\end{equation}
related to the reproducing formula \eqref{reproducing}. Here recall
that $d\mu (z)=\pi ^{-d}e^{-|z|^2}\, d\lambda (z)$, where $d\lambda (z)$
is the Lebesgue measure on $\cc d$. The minimal assumption on $F$
is that it should be locally integrable on $\cc d$ and satisfy
\begin{equation}\label{condanal}
\nm {F\cdot e^{N|\cdo |-|\cdo |^2}}{L^p}<\infty \quad \text{for every} \quad N \ge 0,
\end{equation}
where $p\in (0,\infty ]$ is fixed. We note that \eqref{condanal} is fulfilled for $p=1$ if $F\in L^1_{loc}(\cc d)$ and satisfies
\begin{equation}\label{condadjoint}
\int |F(z)|e^{-\gamma |z|^2}\, d\lambda (z) < \infty \quad \text{for some}  \quad \gamma <1.
\end{equation}

\par

\begin{lemma}\label{projlemma}
Let $\gamma ,\delta >0$ and let $F\in L^1_{loc}(\cc d)$. Then the following is true:
\begin{enumerate}
\item if $p\in [1,\infty ]$ and $F$ satisfies \eqref{condanal},
then $\Pi _AF$ is an entire function on $\cc d$;

\vrum

\item if $p\in (0,\infty ]$ and $F$ satisfies \eqref{condanal}
and in addition is entire, then $\Pi _AF=F$;

\vrum

\item if $4\delta (1-\gamma )\ge 1$, then for some constant $C>0$ it holds
$$
\nm {(\Pi _AF) \cdot e^{-\delta |\cdo |^2}}{L^1}\le C\nm {Fe^{-\gamma |\cdo |^2}}{L^1}
$$
when $F\in L^1_{loc}(\cc d)$ and satisfies \eqref{condadjoint};

\vrum

\item if $\gamma <3/4$ and \eqref{condadjoint} is fulfilled, then
\begin{equation}\label{adjformula}
(F,G)_{B^2} = (F, \Pi _A G)_{B^2} = (\Pi _AF,G)_{B^2},
\end{equation}
for every polynomial $G$ (which is analytic) on $\cc d$.
\end{enumerate}
\end{lemma}

\par

\begin{proof}
By H{\"o}der's inequality we may assume that $p=1$ when proving (1).
Let $E(z,w)= F(w)e^{(z,w)-|w|^2}$. The condition \eqref{condanal} implies
that if $z\in K$, where $K\subseteq \cc d$ is compact, then for each
multi-index $\alpha$ we have that $\partial ^\alpha _zE(z,w)$  is uniformly
bounded in $L^1$ with respect to $w$. The assertion (1) is now a
consequence of the fact that $z\mapsto E(z,w)$ is entire.

\par

Since $\sets {e^{t|z|-|z|^2}\in L^\infty _{loc}(\cc d)}{t\in \textbf R}$ is an
admissible family of weights, we may assume that $p=1$ in view of
Theorem \ref{analident}, when proving (2). The assertion then follows
from the same arguments as for the proof of Lemma A.2 in \cite{SiT2}.
In order to be self-contained we give here a proof.

\par

For every mutli-index $\alpha$ we have
\begin{align}
(\partial _z^\alpha E)(z,w) &= e^{(z,w)-|w|^2}\overline w^\alpha F(w)
\notag %\label{Gzder}
\intertext{and}
(\partial _z^\alpha \partial _{\overline z}E)(z,w) &= 0.\notag
\end{align}
Hence, by the assumptions it follows that the map $w\mapsto
(\partial _z^\alpha E)(z,w)$ belongs to $L^1(\cc d)$ for every $z\in \cc d$,
and that $F_0\equiv \Pi _AF$ in \eqref{A2projection} is analytic with 
derivatives
\begin{equation}\label{F1der}
(\partial ^\alpha F_0)(z) \equiv \int _{\cc d} e^{(z,w)}\overline w^\alpha F(w)\, d\mu (w).
\end{equation}
In particular we have
\begin{equation}\tag*{(\ref {F1der})$'$}
(\partial ^\alpha F_0)(0) \equiv \int _{\cc d} \overline w^\alpha F(w)\, d\mu (w).
\end{equation}

\par

We have to prove that $F_0=F$. Since both $F$ and $F_0$ are entire functions it suffices to prove
\begin{equation*}%\label{F1Feq}
\partial ^\alpha F_0(0)=\partial ^\alpha F(0),
\end{equation*}
for every multi-index $\alpha$.

\par

If  $w=(r_1e^{i\theta _1},\dots ,r_de^{i\theta _d})$, where
$r=(r_1,\dots ,r_d)\in [0,\infty )^d$ and $\theta =(\theta _1,
\dots ,\theta _d)\in \Delta _d=[0,2\pi ]^d$, then \eqref{fTaylor}
and \eqref{F1der}$'$ give
\begin{multline}\label{F1dercomp1}
\partial ^\alpha F_0(0) = \int _{\cc d} \overline w^\alpha F(w)\, d\mu  (w)
\\[1ex]
= \pi ^{-d}\int _{[0,\infty )^d}\Big (\int _{\Delta _d} r^\alpha e^{-i\scal
\theta \alpha}F(r_1e^{i\theta _1},\dots ,r_de^{i\theta _d})r_1\cdots
r_de^{-|r|^2}\, d\theta \Big )\, dr
\\[1ex]
=\pi ^{-d} \int _{[0,\infty )^d}  r^\alpha r_1\cdots r_de^{-|r|^2} J_\alpha (r)\, dr ,
\end{multline}
where
\begin{equation*}%\label{Jintdef}
J_\alpha (r) = \int _{\Delta _d} e^{-i\scal \theta \alpha}F(r_1e^{i\theta _1},\dots ,r_de^{i\theta _d})\,
d\theta = \int _{\Delta _d}\Big ( \sum _\beta a_\beta
r^\beta e^{i\scal \theta {\beta -\alpha}}\Big )\, d\theta 
\end{equation*}
By \eqref{perfunct2} it follows that we may interchange the order
of summation and integration. This gives
\begin{equation}\label{Jcomp}
J_\alpha (r) = \sum _\beta a_\beta r^\beta \int _{\Delta _d}
e^{i\scal \theta {\beta -\alpha}}\, d\theta =(2\pi )^da_\alpha r^\alpha ,
\end{equation}
in view of \eqref{I0def}

\par

By inserting \eqref{Jcomp} into \eqref{F1dercomp1} and
taking $u_j=r_j^2$ as new variables of integration, \eqref{fTaylor} gives
\begin{multline}\label{F0eqF}
\partial ^\alpha F_0(0) =2 ^{d}a_\alpha \int _{[0,\infty )^d}  r^{2\alpha}
r_1\cdots r_de^{-|r|^2}\, dr
\\[1ex]
= a_\alpha \int _{[0,\infty )^d}  u^{\alpha}e^{-(u_1+\cdots +u_d)}\, du
= a_\alpha \alpha ! =\partial ^\alpha F(0).
\end{multline}
This proves (2).

\par

The assertion (3) follows from the inequality
$$
|F(w)e^{(z,w)-|w|^2}e^{-\delta |z|^2}|\le |F(w)|e^{-\gamma |w|^2}
e^{-(1-\gamma )|w-(1-\gamma)^{-1}z/2|^2},
$$
and (4) is obtained by choosing $\delta >1$ in the latter estimate, giving that
$$
(z,w)\mapsto F(w)\overline{G(z)}e^{(z,w)-|w|^2}e^{-|z|^2}
$$
belongs to $L^1(\cc d\times \cc d)$ when $G$ is a polynomial. The relation
\eqref{adjformula} is now an immediate consequence of the reproducing
formula \eqref{reproducing} applied on $G$, and Fubbini's theorem.
\end{proof}

\par

\begin{rem}\label{derivrepr}
We note that if $F\in A(\cc d)$ and satisfies \eqref{condanal}, then
\begin{align*}
z^\alpha F(z) &= \int w^\alpha F(w)e^{(z,w)}\, d\mu  (w)
\intertext{and}
\partial ^\alpha F(z) &= \int \overline{w}^\alpha F(w)e^{(z,w)}\, d\mu (w),
\end{align*}
giving that
\begin{equation}\label{eq3.*}
\partial ^\alpha F(0)=(F,z^\alpha )_{A^2}
\end{equation}
(see also \cite{B1}).

\par

In fact, the first formula follows by replacing $F$ by $z^\alpha F$ in the reproducing formula and using (2) in Lemma \ref{projlemma}. For the second formula we note that the condition \eqref{condanal} and reproducing formula give
\begin{multline*}
\partial ^\alpha F(z) = \partial ^\alpha \Big (\int  F(w)e^{(z,w)}\, d\mu  (w) \Big )
\\[1ex]
=\int \partial _z^\alpha \big ( F(w)e^{(z,w)}\big )\, d\mu  (w) = \int \overline{w}^\alpha F(w)e^{(z,w)}\, d\mu  (w),
\end{multline*}
and the result follows.
\end{rem}

\par

\begin{rem}\label{F1eqF2rem}
It follows from the the proof, and especially \eqref{F0eqF}, of the previous lemma that if $F_1,F_2\in A(\cc d)$, \eqref{condanal} holds for $F=F_j$ and $p=1$, $j=1,2$, and that
$$
(F_1,G)_{A^2}=(F_2,G)_{A^2},\qquad G\in P(\cc d),
$$
then $F_1=F_2$.
\end{rem}

\par

Next we prove that $A(\omega ,\mascB )$ and $M(\omega ,\mascB )$
are Banach spaces when $\nu _1(\mascB )\ge 1$.

\par

\begin{thm}\label{Bspacethm}
Let $\omega _1\in \mascP _Q^0(\cc d)$, $\omega _2\in \mascP _Q(\cc d)$ and $\mascB$ be a mixed quasi-norm space on $\cc d$.
Then the following is true:
\begin{enumerate}
\item $M(\omega _1,\mascB )$ and $A(\omega _2,\mascB )$ are quasi-Banach spaces;

\vrum

\item if in addition $\nu _1(\mascB )\ge 1$, then $M(\omega _1,\mascB )$
and $A(\omega _2,\mascB )$ are Banach spaces.
\end{enumerate}
\end{thm}

\par

\begin{proof}
By Theorem \ref{mainthm} it suffices to prove the result for $A(\omega _2,\mascB )$. Since the 
statement is invariant under dilations, we may assume that \eqref{Gaussest} holds for $\omega =
\omega _2$ and $c=1/8$. Furthermore, since it is obvious that $\nm \cdo{B(\omega _2,\mascB )}$ is a norm when $\nu _1(\mascB )\ge 1$, it suffices to prove (1).

\par

By Theorem \ref{analident} it follows that $A(\omega _2,\mascB )$
is continuously embedded in $A(\omega _0,L^1)$, for some choice
of $\omega _0\in \mascP _Q$. Furthermore, it follows from the proof of
Theorem \ref{analident} that we may choose $\omega _0$ such that it
satisfies \eqref{Gaussest} with $c=1/6$. Consequently, any $F$
in $A(\omega _0,L^1)$ fulfills \eqref{condanal} with $p=1$.

\par

Now let $(F_j)_{j=1}^\infty$ be a Cauchy sequence in $A(\omega _2,\mascB )$. Since both
$B(\omega _2,\mascB )$ and $B(\omega _0,L^1)$ are quasi-Banach spaces,
it exists an element $F\in B(\omega _2,\mascB )\cap B(\omega _0,L^1)$ such
that $F_j\to F$ in $B(\omega _2,\mascB )$ and $B(\omega _0,L^1)$ as $j\to \infty$.

\par

We have to prove that $F\in A(\cc d)$. By the assumptions and Lemma \ref{projlemma} it
follows that $F_0=\Pi _AF$ in \eqref{A2projection} defines an analytic function, and that
$$
F_j(z)=\int F_j(w)e^{(z,w)}\, d\mu (w)
$$
for every $j$. Furthermore, for each compact set
$K\subseteq \cc d$ there is a constant $C>0$ such that
\begin{multline}\label{Cauchyest1}
\sup _K |F_j(z)-F_0(z)|
\le  \pi ^{-d}\int |F_j(w) -F(w)|e^{-|w|^2+C|w|}\, d\lambda (w)
\\[1ex]
= \pi ^{-d}\int \Big | (F_j(w) -F(w))e^{-|w|^2/2}\omega _0(2^{1/2}\overline w)\Big |
\cdot (e^{-|w|^2/2+C|w|}\omega  _0(2^{1/2}\overline w)^{-1})\, d\lambda (w)
\\[1ex]
\le C_{\omega _0} \int | F_j(w) -F(w)|e^{-|w|^2/2}\omega
_0(2^{1/2}\overline w)\, d\lambda (w) = C\nm {F_j-F}{B(\omega _0,L^1)},
\end{multline}
where
$$
C_{\omega _0}=\underset {w\in \cc d}\essup \big (e^{-|w |^2/2+C|w|}/\omega _0(2^{1/2}\overline w) \big )<\infty ,
$$
and $C$ is a constant. Since the right-hand side of \eqref{Cauchyest1}
turns to zero as $j\to \infty$, it follows that $F_j\to F_0$ locally uniformly
as $j\to \infty$. This proves that $F=F_0$, which is analytic, and the result follows.
\end{proof}

\par

\subsection{Density and duality properties} Next we prove that  if $\mascB $ is a mixed norm space 
with $\nu _2(\mascB )<\infty$, and that the weight $\omega \in \mascP _{Q}^0(\cc d)$ in 
addition should be dilated suitable,
then the set $P(\cc d)$ of polynomials on $\cc d$ is dense in
$A(\omega ,\mascB )$. Furthermore, in this situation we also prove that the
dual of $A(\omega ,\mascB )$ can be identified with $A(1/\omega ,\mascB ')$,
through a unique extention of the $A^2$ form on $P(\cc d)$.

\par

An important part of these considerations concerns possibilities to
approximate elements $F$ in $A(\omega ,\mascB )$ with their
dilations $F(\lambda \cdo )$ for $0<\lambda <1$. We note that
the latter functions belong to
\begin{equation}\label{APdef}
A_P(\cc d)\equiv \sets {F\in A(\cc d)}{F\cdot e^{-(1-\ep )|z|^2/2}\in
\mascB \ \text{for some}\ \ep >0},
\end{equation}
and that
\begin{equation*}%\label{APincl}
P(\cc d)\subseteq A_P(\cc d)\subseteq A(\omega ,\mascB ),\quad \text{when}
\quad \omega \in \mascP _Q^0(\cc d).
\end{equation*}
This is a straight-forward consequence of Theorem \ref{analident} and the definitions.

\par

\begin{prop}\label{APprop}
The set $A_P(\cc d)$ in \eqref{APdef} is independent of the mixed
quasi-norm space $\mascB$ on $\cc d$. Furthermore, if $\mascB$
is a mixed norm space on $\cc d$ and $\omega \in \mascP
_Q^0(\cc d)$, then $P(\cc d)$ is dense in $A_P(\cc d)$ with
respect to the topology in $A(\omega ,\mascB )$.
\end{prop}

\par

\begin{proof}
The first part follows immediately from Theorem
\ref{analident}, and the observation that 
$$
\Omega =\sets {e^{-(1-\ep )|z|^2/2}}{0 < \ep < 1,\ z\in \cc d}
$$
is an admissible family of weights on $\cc d$.

\par

The first part then shows that we may assume that $\mascB =L^1$ and $\omega =e^{\ep _0|\cdo |^2}$
for some small $\ep _0>0$ which depends on $\lambda >0$, when proving the second part.
The result is then an immediate consequence of \cite[Prop. 3.2]{SiT2}.
The proof is complete.
\end{proof}

\par

\begin{rem}
By similar arguments, using Proposition \ref{harmident} instead of
Theorem \ref{analident}, it follows that
$$
\sets {f\in \mathcal H(\rr d)}{f\cdot e^{-(1-\ep )|x|^2/2}\in \mascB \ \text{for some}\ \ep >0}
$$
is independent of the mixed norm space $\mascB $ on $\rr d$. Here recall that $\mathcal H(\rr d)$ is the set of harmonic functions on $\rr d$.
\end{rem}

\par

Our result on duality is the following.

\par

\begin{thm}\label{dualthm}
Let $\omega \in \mascP _Q^0(\cc d)$ be dilated suitable and $\mascB $ be a mixed norm space on $\cc d$ such that $\nu _2(\mascB )<\infty$. Then the following is true:
\begin{enumerate}
\item the $A^2$ form on $P(\cc d)$ extends uniquely to a continuous sesqui-linear form on 
$A(\omega ,\mascB )\times A(1/\omega ,\mascB ')$;

\vrum

\item the dual of $A(\omega ,\mascB )$ can be identified by $A(1/\omega ,\mascB ')$ 
through the extension of the $A^2$ form on $P(\cc d)$.
\end{enumerate}
\end{thm}

\par

Here we recall that $\mathscr B'=L^{p'}(V)$, when $\mathscr B=L^p(V)$,
and $p\in [1,\infty ]^n$ and $V$ are given by \eqref{Vparrays}.
%Also note that in contrast to most of the results in the present and previous sections, the 
%assumption on the space $\mascB$ in Theorem \ref{dualthm} is more restrictive in the sense that  it 
%is assumed that $\mascB$ should be a mixed norm space 

\par

We also have the following result on density, which is strongly
connected to the proof of Theorem \ref{dualthm}.

%We need some preparations for the proof of Theorem \ref{dualthm}.
%Especially  we need the following result on density
%properties for $A(\omega ,\mascB )$. 

\par

\begin{thm}\label{densthm}
Let $\omega \in \mascP _{Q}^0(\cc d)$ be dilated suitable, $\mascB $ be a 
mixed quasi-norm space on $\cc d$ such that $\nu _2(\mascB )<\infty$, and let $F\in
A(\omega ,\mascB )$ and $G\in A(\omega ,\mascB ')$. Then the following is 
true:
\begin{enumerate}
\item $P(\cc d)$ is dense in $A(\omega ,\mascB )$. If in addition $\nu _1(\mascB )\ge 1$, then
$P(\cc d)$ is dense in $A(\omega ,\mascB ')$ with respect to the weak$^*$-topology;

\vrum

\item if $0<\lambda <1$, then $F(\lambda \cdo )\to F$ in $A(\omega ,\mascB )$.
If in addition $\nu _1(\mascB )\ge 1$, then $G(\lambda \cdo )\to G$ with respect
to the weak$^*$-topology in $A(\omega ,\mascB ')$ .
\end{enumerate}
\end{thm}

\par

We start by proving the first parts of (1) and (2) in Theorem \ref{densthm}.
Thereafter we prove Theorem \ref{dualthm}, and finally we prove the last
parts of (1) and (2) in Theorem \ref{densthm}.

\par

Some preparations for the proofs are needed. We start by recalling the
following generalization of Lebesgue's theorem. 

\par

\begin{lemma}\label{genLeb}
Let $p\in (0,\infty )$, $d\mu$ be a positive measure, and let $f_j,f,g_j,g\in L^p(d\mu )$
be such that $f_j\to f$ and $g_j\to g$ pointwise a.{\,}e. as $j\to \infty$, $|f_j|\le g_j$,
$|f|\le g$ and $\nm {g_j}{L^p(d\mu )}\to \nm g{L^p(d\mu )}$ as $j\to \infty$. Then
$f_j\to f$ in $L^p(d\mu )$ as $j\to \infty$.
\end{lemma}

\par

\begin{proof}
Let $q=\max (1,p)$. Then the result follows by applying Fatou's
lemma on $2^q(g^p+g_j^p) - |f - f_j|^p$. The details are left for the reader.
\end{proof}

\par

The first part of Theorem \ref{densthm} (2) is an immediate consequence of the following lemma.

\par

\begin{lemma}\label{Flambdalim}
Let $0<\omega _0\in L^\infty _{loc}(\rr d)$ be such that
$$
\omega _0(x)\le C\omega _0(\lambda x),\quad x\in \rr d,\ 1-\theta <\lambda <1,
$$
for some constants $\theta \in (0,1)$ and $C>0$, and let and
$\mascB $ be a mixed quasi-norm space on $\rr d$ with $\nu _2(\mascB )<\infty$. 
Also let $f\in L^r_{loc}(\rr d)$ be such that $f
\cdot \omega _0\in \mascB $, where $r=\nu _1(\mascB )$. Then $f(\lambda \cdo )
\omega _0\in \mascB $ when $1-\theta <\lambda <1$, and
\begin{equation}\label{limitform}
\lim _{\lambda \to 1-}\nm {(f-f(\lambda \cdo ))\omega _0}{\mascB }=0.
\end{equation}
\end{lemma}

\par

\begin{proof}
%We shall use similar arguments as in the proof of Lebesgue's
%theorem in the literature.
We  only consider the case when $\mathscr B=L^p(\rr d)$,
$p\in (0,\infty )$. The straight-forward modifications to the
general case are  left for the reader.

\par

By the assumptions we have
%%%
%\begin{equation}\label{}
$$
|f(\lambda x)\omega _0(x)|\le C|f(\lambda x)\omega _0(\lambda x)|,
$$
%\end{equation}
%%%
and it follows that $\nm {f(\lambda \cdo )\omega _0}{L^p}\le
C\lambda ^{-d/p}\nm {f\omega _0}{L^p}<\infty$,
by a simple change of variables in the integral. Hence
$\{ f(\lambda \cdo )\omega _0 \} _{1-\theta <\lambda <1}$
is a bounded set in $L^p$.

\par

By straight-forward approximations, it follows that we may
assume that $f\in C_0(\rr d)$ and $\omega _0\in C(\rr d)$
when proving \eqref{limitform}. Then the result follows if we choose
$g_\lambda (x)=|f(\lambda x )|$, $g(x) = |f(x)|$ and $d\mu (x)
= \omega _0(x)\, dx$ in Lemma \ref{genLeb}. The proof is complete.
\end{proof}
%
%
%
%Since
%$$
%2^p(|f(x)|^p+|f(\lambda x)|^p)-|f(x)-f(\lambda x) |^p\ge 0,
%$$
%Fatou's lemma gives
%%%
%\begin{multline*}
%2^{p+1}\int |f(x)\omega _0(x)|^p\, dx
%\\[1ex]
%=\int \lim _{\lambda \to 1-} \big ( 2^p|f(x)|^p+2^p|f(\lambda x)|^p-|f(x)-f(\lambda x) |^p \big )
%\omega _0(x)^p\, dx
%\\[1ex]
%\le \liminf _{\lambda \to 1-} \int  \big ( 2^p|f(x)|^p+2^p|f(\lambda x)|^p-|f(x)-f(\lambda x)|^p \big )
%\omega _0(x)^p\, dx
%\\[1ex]
%= \liminf _{\lambda \to 1-} \int ( 2^p|f(x)|^p+2^p|f(\lambda x)|^p)\omega _0(x)\, dx -\limsup _
%{\lambda \to 1-}\int |f(x)-f(\lambda x) |^p \omega _0(x)^p\, dx
%\\[1ex]
%= 2^{p+1}\int |f(x)\omega _0(x)|^p\, dx - \limsup _{\lambda \to 1-}\int |f(x)-f(\lambda x)|^p \omega 
%_0(x)^p\, dx ,
%\end{multline*}
%%%
%where the last equality follows by Lemma \ref{genLeb}. This gives \eqref{limitform}, and the 
%proof is complete.
%\end{proof}

\par

\begin{proof}[Proof of the first parts of Theorem \ref{densthm}]
By Lemma \ref{Flambdalim} it suffices to prove that if $F\in A(\omega ,\mascB )$ and
$\lambda <1$ with $1-\lambda$ small enough, then
$\nm {F(\lambda \cdo ) -G}{A(\omega ,\mascB )}$ can
be made arbitrarily small as $G\in P(\cc d)$.

\par

Let $s\in \mathbf R$ be chosen such that $\lambda <s<1$.
Then it follows from the assumptions that
$$
C^{-1}\omega (z)e^{-|z|^2/2}\le e^{-s|z|^2/2}\le
C\omega (\lambda z)e^{-\lambda ^2|z|^2/2},
$$
for some constant $C$. In particular, we have
\begin{align}
\nm {F(\lambda \cdo ) -G}{A(\omega ,\mascB )} &\le
C\nm {(F(\lambda \cdo ) -G)e^{-s|\cdo |^2/2}}{\mascB}\label{differandnorms}
\intertext{and}
\nm {F(\lambda \cdo ) e^{-s|\cdo |^2/2}}{\mascB} &< \infty ,\label{Flambdaest}
\end{align}
for some constant $C>0$.

\par

By \eqref{Flambdaest} it follows that $F(\lambda \cdo )$ belongs to the set
$$
\mathcal A_{p,t}(\cc d)\equiv \sets {F\in A(\cc d)}{ \nm
{F\cdot  e^{-t|\cdo |^2/2}}{\mascB} < \infty},
$$
when $t=s$. Since the same is true for any choice of
$t\in (\lambda ,s)$, $\mathcal A_{p,t}$ increases with
$t$, and that $\sets {e^{-t|\cdo |^2/2}}{\lambda <t<s}$
is an admissible family of weights, it follows from Theorem
\ref{analident} that \eqref{differandnorms} and \eqref{Flambdaest}
hold after $s$ has been replaced by an appropriate
$t\in (\lambda ,s)$, and $p=(1,\dots ,1)$.
Since $P(\cc d)$ is dense in $\mathcal A_{p,t}(\cc d)$
with $p=(1,\dots ,1)$, in view of 
\cite[Proposition 3.2]{SiT2}, it follows that the right-hand side
of \eqref{differandnorms} can be made arbitrarily small, and
the assertion follows. This completes the proof
of the first parts of Theorem \ref{densthm} (1) and (2).
\end{proof}

\par

Next we turn to the proof of Theorem \ref{dualthm}. First we note that if $A(\omega ,\mascB )$ is the 
same as in Theorem \ref{dualthm} and $l\in (A(\omega ,\mascB ))'$, then
\begin{equation}\label{dualform}
l(F) = \int F(z)\overline {G(z)}\, d\mu  (z),\qquad F\in A(\omega ,\mascB ),
\end{equation}
for some $G\in B(1/\omega ,\mascB ')$.
(Note that the mixed norm space $\mascB '$ is the dual to $\mascB$ when
$\nu _2(\mascB )<\infty$ and the $L^2$ form is used.) In fact, it follows from
the definitions that there is a constant $C\ge 0$ such that
\begin{equation}\label{formest1}
| l(F)|\le C\nm F{B(\omega ,\mascB )}
\end{equation}
when $F\in A(\omega ,\mascB )$. By Hahn-Banach's theorem it follows that $l$ is 
extendable to a linear continuous form on $B(\omega ,\mascB )$ and that \eqref{formest1} 
still holds for $F\in B(\omega ,\mascB )$. The formula \eqref{dualform} now follows from 
well-known results in measure theory.

\par

From now on we let $l_G$ be the continuous linear form on
$A(\omega ,\mascB )$, defined by \eqref{dualform}
when $G\in B(1/\omega ,\mascB ')$. Then it follows from the previous investigations that the map
$G\mapsto l_G$ is surjective from $B(1/\omega ,\mascB ')$ to $(A(\omega ,\mascB ))'$.

\par

In what follows we link the kernel of the latter map with the kernel
%$N(\omega ,\mathscr B)$ to the projection operator
%$\Pi _A$. This means that
$$
N(\omega ,\mascB )\equiv \sets {F\in B(\omega ,\mascB )}{\Pi _A F=0}
$$
of the projection operator. Here we note that every
$F$ in $B(\omega ,\mascB )$ for $\omega \in \mascP _Q^0(\cc d)$ and mixed norm
space $\mascB $, fulfill the required properties in Lemma \ref{projlemma}.

\par

\begin{lemma}\label{kernellink}
Let $\mascB $ be a mixed norm space on $\cc d$ and
assume that $\omega \in \mascP _Q^0(\cc d)$ is dilated suitable.
Then the following is true:
\begin{enumerate}
\item $N(\omega ,\mascB )$ is a closed subspace of $B(\omega ,\mascB )$;

\vrum

\item if $\nu _2(\mascB )<\infty$, then the kernel of the map $G\mapsto l_G$, from $B
(1/\omega ,\mascB ')$ to $(A(\omega ,\mascB ))'$ is equal to $N(1/\omega ,\mascB  ')
$.
\end{enumerate}
\end{lemma}

\par

\begin{proof}
Assume that $F_j\in N(\omega ,\mascB )$, $j\ge 1$, converges to
$F\in B(\omega ,\mascB )$ in $B(\omega ,\mascB )$, as $j\to \infty$.
If $K\subseteq \cc d$ is compact, then for some constant $C_K$,
depending on $K$, H{\"o}lder's inequality gives
$$
\sup _{z\in K}|\Pi _A F| = \sup _{z\in K}|\Pi _A (F-F_j)|\le
C_K\nm {F-F_j}{B(\omega ,\mascB )}\to 0
$$
as $j\to \infty$. This proves (1).

\par

(2) By the first part of Theorem \ref{densthm} (1) it suffices to prove
that $l_G(F)=0$ for every polynomial $F$ on $\cc d$, if and only if
$G\in N(1/\omega ,\mascB ')$. By Lemma \ref{projlemma} (4) it follows that
\begin{equation}\label{formident1}
l_G(F) = l_{\Pi _AG}(F) = (F,G)_{A^2},\qquad F\in P(\cc d),
\end{equation}
and $G\in B(1/\omega ,\mascB ')$. This proves that $l_G =0$
when $G\in N(1/\omega ,\mascB ')$.

\par

On the other hand, if $l_G(F)=0$ for every polynomial $F$, then
\eqref{eq3.*} and \eqref{formident1} show that the entire function $\Pi _AG$ satisfies
$$
(\partial ^\alpha \Pi _AG)(0) =0,
$$
for every $\alpha$. This implies that $\Pi _AG =0$, i.{\,}e.
$G\in N(1/\omega ,\mascB ')$, and the proof is complete.
%
%
%However, this 
%equivalence is an immediate consequence of \eqref{eq3.*}, the fact that $\Pi _AG$ is analytic, and the equalities $l_G(F)=(F,G)_{A^2}=(F,\Pi 
%_AG)_{A^2}$, when $F$ is a polynomial (cf. Lemma \ref{projlemma}). The proof is 
%complete.
\end{proof}

\par

\begin{rem}\label{remGenLemKern}
Let $
\omega \in \mascP _Q^0(\cc d)$, $\mathscr B$ be a mixed norm space on $\cc d$, and 
let $A_*(\omega ,\mathscr B)$ be the completion of $P(\cc d)$ under the norm $\nm \cdo{A
(\omega ,\mathscr B)}$. Then the same arguments as in the proof of Lemma \ref{kernellink} 
give that the kernel of the map $G\mapsto l_G$, from $B
(1/\omega ,\mascB ')$ to $(A_*(\omega ,\mascB ))'$ is equal to $N(1/\omega ,\mascB  ')
$.

\par

We note that this generalizes Lemma \ref{kernellink} (2), since if in addition
$\nu _2(\mathscr B)<\infty$, then $A_*(\omega ,\mathscr B)=A(\omega ,\mathscr B)$
in view of the first part of Theorem \ref{densthm} (1).
\end{rem}

\par

As a consequence of Lemma \ref{kernellink} it follows that if $\nu _2(\mascB )<\infty$, then the surjective and continuous map $G\mapsto l_G$ from $B(1/\omega ,\mascB ')$ to $(A(\omega ,\mascB ))'$ induces a homeomorphism from the quotent space
$$
C(1/\omega ,\mascB ') \equiv B(1/\omega ,\mascB ')/N(1/\omega ,\mascB ')
$$
to $(A(\omega ,\mascB ))'$. Here note that Lemma \ref{kernellink} implies that $C(1/\omega ,\mascB ')$ is a Banach space under the usual quotent topology.

\par

\begin{proof}[Proof of Theorem \ref{dualthm}]
Let $C_w(1/\omega ,\mascB ')$ be equal to $C(1/\omega ,\mascB ')$ equipped by the induced weak$^*$-topology on $(A(\omega ,\mascB ))'$. For each $G\in B(1/\omega ,\mascB ')$ we write $G^*=G\mod N(1/\omega ,\mascB ')$ for its image in $C(1/\omega ,\mascB ')$ under the quotent map. Then it follows that $C_w(1/\omega ,\mascB ')$ is a local convex topological vector space, and that the separating vector space of linear functionals on $C_w(1/\omega ,\mascB ')$ is equal to
\begin{equation*}%\label{weakfunctionals}
\sets {\Lambda}{\Lambda (G^*)=\overline {l_G(F)}\  \text{for some}\  F\in A(\omega ,\mascB )}.
\end{equation*}
Note here that the equality $\Lambda (G^*)=\overline {l_G(F)}$ makes sense since Lemma \ref{kernellink} gives $l_{G_1}(F)=l_{G_2}(F)$ when $G_1,G_2\in B(1/\omega ,\mascB )$ are two different representatives of $G\mod N(1/\omega ,\mascB ')$ and $F\in A(\omega ,\mascB )$.

\par

Since $\Pi _A(F)=F$ when $F\in A(1/\omega ,\mascB ')$, it follows that the map $G\mapsto  G\mod N(1/\omega ,\mascB ')$ from $A(1/\omega ,\mascB ')$ to $C(1/\omega ,\mascB ')$ is continuous and injective. Let $C_0(1/\omega ,\mascB ')$ be the image of this map. The result follows if we prove that $C_0(1/\omega ,\mascB ')=C(1/\omega ,\mascB ')$.

\par

By Hahn-Banach's theorem it suffices to prove that if $\Lambda$ is a linear and continuous functional on $C_w(1/\omega ,\mascB ')$ which is zero on $C_0(1/\omega ,\mascB ')$, then $\Lambda $ is identically zero. Since the dual of $C_w(1/\omega ,\mascB ')$ is equal to $A(\omega ,\mascB )$ when using the $A^2$ form, we have
$$
\Lambda (G^*) =\Lambda _F(G^*) \equiv (G,F)_{A^2},
$$
for some $F\in A(\omega ,\mascB )$.

\par

Now $\Lambda _F(G^*)=0$ when $G\in A(1/\omega ,\mascB ')$. In particular $\Lambda _F(z^\alpha )=(z^\alpha ,F)_{A^2}=0$ for every multi-index $\alpha$. Since $\partial ^\alpha F(0)=(F,z^\alpha )_{A^2}$ in view of Remark \ref{derivrepr}, it follows that the entire function $F$ is zero together with all its derivatives at origin. This implies that $F$ is identically zero, and hence, $\Lambda$ is zero. This proves the result.
\end{proof}

\par

\begin{rem}
By Theorem \ref{dualthm} and its proof it follows that if $\omega \in \mascP _Q^0(\cc d)$ is dilated suitable and $\mascB $ is a mixed norm space such that $\nu _2(\mascB )<\infty$, then the mappings
$$
G\mapsto l_G\qquad \text{and}\qquad G\mapsto G^*
$$
from $A(1/\omega,\mascB ')$ to $(A(\omega,\mascB ))'$ and from $A(1/\omega,
\mascB ')$ to $C(1/\omega,\mascB ')$ respectively, are bijective and continuous. Hence 
these mappings are homeomorphisms by the open mapping theorem. In particular, the norms in respective space are equivalent, giving that for some 
$C>0$ it holds
$$
C^{-1}\nm {l_G}{(A(\omega ,\mascB ))'} \le \nm {G^*}{C(1/\omega,\mascB ')}\le \nm G{A
(1/\omega,\mascB ')}\le C\nm {l_G}{(A(\omega ,\mascB ))'},
$$
when $G\in A(1/\omega, \mascB ')$.
\end{rem}

\par

\begin{proof}[The end of the proof of Theorem \ref{densthm}]
We start to prove the second part of (2). Let $F\in A(1/\omega ,\mascB )$ be fixed, and choose the vector spaces $V_j\subseteq \cc d$ and $p\in [1,\infty )$ such that $\mascB =L^p(V)$ when $V=(V_1,\dots V_n)$. Then it follows by straight-forward computations that
\begin{multline*}
|(G-G(\lambda \cdo ),F)_{A^2}|
\\[1ex]
=\left |  \pi ^{-d}\int \big (  G(z)e^{-|z|^2/2}  \big )   \big (
\overline {F(z)}e^{-|z|^2/2} - \lambda ^{-2d} \overline {F(z/\lambda )}
e^{-(2-\lambda ^2)|z|^2/(2\lambda ^2)} \big ) d\lambda (z) \right |
\\[1ex]
\le C\nm {S(F \cdot e^{-|\cdo|^2/2}- \lambda ^{-2d}F(\cdo /\lambda )
e^{-(2-\lambda ^2)|\cdo |^2/(2\lambda ^2)})/\omega }{L^p(V)},
\end{multline*}
where $C=C_0\nm G{A(\omega ,\mascB ')}<\infty$, for some constant $C_0>0$,
and $S$ is the operator in \eqref{Sdef}.
By taking $2^{-1/2}\overline z/\lambda$ as new variables of integration we get
\begin{multline}\label{differest}
|(G-G(\lambda \cdo ),F)_{A^2}|
\\[1ex]
\le C_1\lambda ^{c}\nm {(F (\lambda \cdo )e^{-\lambda ^2|\cdo |^2/2}-
\lambda ^{-2d}F\cdot e^{-(1-\lambda ^2/2)|\cdo |^2})/\omega
(S^{-1}(\lambda \cdo ))}{L^p(V)},
\end{multline}
for some constants $c$ and $C_1$.

\par

Now we set
$$
\Phi _\lambda (z)= (F (\lambda z )e^{-\lambda ^2|z |^2/2}-
\lambda ^{-2d}F(z) e^{-(1-\lambda ^2/2)|z|^2})/\omega (2^{1/2}
\lambda \overline z ) 
$$
for the last integrand. By \eqref{addomegcond}, we get $|\Phi
_\lambda | \le \Psi _\lambda$, where
$$
\Psi _\lambda (z) = |F (\lambda z )|e^{-\lambda ^2|z |^2/2}/\omega
(2^{1/2}\lambda \overline z) + C\lambda ^{-2d}|F(z)| e^{-|z|^2/2})/\omega
(2^{1/2}\overline z),
$$
provided the constant $C$ is chosen sufficiently large. Since
$\Phi _\lambda \to 0$ pointwise, and 
$$
\Psi _\lambda (z)\to (C+1)F(z)e^{-|z|^2/2}/\omega (2^{1/2}\overline z)
$$
in $L^p(V)$ as $\lambda \to 1-$ in view of Lemma \ref{Flambdalim},
it follows from Lemma \ref{genLeb} that $\Phi _\lambda \to 0$ in
$L^p(V)$ as $\lambda \to 1-$. This implies that the right-hand side of
\eqref{differest} turns to zero as $\lambda \to 1-$, and the second part of (2) follows.

\par

It remains to prove the second part of (1). Let $G\in A(\omega ,\mascB ')$
and $F\in A(1/\omega ,\mascB )$.
By the second part of (2) it suffices to prove that for some
$0<\lambda <1$ and some polynomials $G_j$ we have
\begin{equation}\label{Glambdaweak}
|(G(\lambda ^2\cdo )-G_j,F)_{A^2}|\to 0\quad \text{as}\quad j\to \infty .
\end{equation}

\par

Therefore, let $G_j$ be a sequence of polynomials on $\cc d$.
By Proposition \ref{APprop} we get
$G(\lambda \cdo )e^{-|\cdo |^2/2}\in L^1(\cc d)$, and
\begin{multline}\label{Glambdaest}
|(G(\lambda ^2\cdo )-G_j,F)_{A^2}|\le \nm {G(\lambda ^2\cdo ) -
G_j}{A(\omega ,\mascB ')}\nm F{A(1/\omega ,\mascB )}
\\[1ex]
\le C\nm {(G(\lambda \cdo )-G_j(\cdo /\lambda ))e^{-|\cdo
|^2/2}}{L^1}\nm F{A(1/\omega ,\mascB )},
\end{multline}
for some constant $C$. In the last inequality we have applied
Theorem \ref{analident} with $\mascB _1=L^1$ and $\mascB _2
=\mascB$ on the estimates
\begin{multline*}
|G(\lambda w )-G_j(w /\lambda )|e^{-|w|^2/(2\lambda ^2)}\omega (2^{1/2}\overline w /\lambda )
\\[1ex]
\le C |G(\lambda w )-G_j(w /\lambda )|e^{-|w|^2/(2\lambda )}\le
C|G(\lambda w )-G_j(w /\lambda )|e^{-|w|^2/2},
\end{multline*}
for some constant $C$.

\par

Consequently, if $\omega _0=1$ and $G_j$ are chosen
such that $G_j(\cdo /\lambda )\to G(\lambda \cdo )$
in $A(\omega _0,L^1)$ as $j\to \infty$, then \eqref{Glambdaweak}
follows from \eqref{Glambdaest}. The proof is complete.
\end{proof}

\par

\subsection{Narrow convergence}

\par

Next we introduce the narow convergence for $A(\omega ,\mathscr B)$. We note that Theorem
\ref{densthm} give no explicit possibilities to approximate elements in $A(\omega ,\mascB )$
with polynomials when $\nu _1(\mascB )=1$ and $\nu _2(\mascB )=\infty$. In this context,
the narrow convergence makes such approximations possible in some of these situations. The
assumptions on the involved weight functions and $\mascB $ is that the pair $(\mascB ,\omega )$
should be narrowly feasible (cf. Definition \ref{feasiblpair}).

\par

In order to define the narrow convergence we introduce the functional 
$$
J_{F,\omega }(\zeta _2)\equiv \sup _{\zeta _1\in V_1}\Big ( |S(F(z)e^{-|z|^2/2})|\omega 
(z)  \Big ),\qquad
z=(\zeta _1,\zeta _2)\in V_1 \oplus V_1^\bot \simeq \cc d,
$$
when $F\in A(\omega ,\mascB )$. Here we recall that $S$ is the dilation operator,
given by \eqref{Sdef}.

\par

\begin{defn}\label{narrowconvdef}
Let $(\mascB ,\omega )$ be a narrowly feasible space weight pair on $\cc d$,
let $p$ and $V$ be the same as in Definition \ref{feasiblpair}, and let
$q=(p_2,\dots ,p_n)$ and $U=(V_2,\dots V_{n})$. Also let $F,F_j\in
A(\omega ,\mascB )$, $j\ge 1$. Then $F_j$ is said to converge to
$F$ \emph{narrowly} as $j\to \infty$, if the following conditions are fulfilled:
\begin{enumerate}
\item $S(F_je^{-|\cdo |^2/2})\omega \to S(Fe^{-|\cdo |^2/2})\omega$ in
$\mathcal S'_{1/2}(\cc d)$ as $j\to \infty$;

\vrum

\item $J_{F_j,\omega ,p}\to J_{F,\omega ,p}$ in $L^q(U)$ as $j\to \infty$.
%, for some $\omega _0\in \mascP _Q^0(\cc d)$ which is equivalent to $\omega$.
\end{enumerate}
\end{defn}

\par

The following result gives motivations for introducing the narrow convergence.

\par

\begin{thm}\label{narrowthm}
Let $(\mascB ,\omega )$ be a narrowly feasible pair on $\cc d$. Then the following is true:
\begin{enumerate}
\item $P(\cc d)$ is dense in $A(\omega ,\mascB )$ with respect to the narrow convergence;

\vrum

\item if $F\in A(\omega ,\mascB )$ and $0<\lambda <1$, then $F(\lambda \cdo )\to F$ narrowly as $\lambda \to 1-$.
\end{enumerate}
\end{thm}

\par

\begin{proof}
We start by proving (2). We may assume that $\omega =\omega _0$, where $\omega _0$ is the same as in Definition \ref{narrowconvdef}. Then
$$
S(F(\lambda \cdo ) e^{-|\cdo |^2/2})\omega  \, \to S(F\,  e^{-|\cdo |^2/2})\omega
$$
pointwise and in $\mathcal S_{1/2}'(\rr d)$ as $\lambda \to 1-$.

\par

Furthermore, if $z=(\zeta _1,\zeta _2)\in V_1\oplus V_1^\bot =\cc d$, then
$$
J_{F(\lambda \cdo ),\omega ,p}(\zeta _2)\le C_\lambda J_{F,\omega ,p}(\lambda \zeta _2),
$$
where
$$
C_\lambda =\sup _{z\in \cc d}\Big ( \frac {\omega (z)e^{-(1-\lambda ^2)|z|^2/4}}{\omega  (\lambda z)}\Big ),
$$
in view of the definition of $S$. Since $C_\lambda \to 1$ as $\lambda \to 1-$, and
$J_{F,\omega ,p}(\lambda \cdo )\to J_{F,\omega ,p}$
in $L^q(U)$ as $\lambda \to 1-$, in view of Lemma
\ref{Flambdalim}, it follows from Lemma \ref{genLeb}
that $J_{F(\lambda \cdo ),\omega ,p}\to J_{F,\omega ,p}$
in $L^q(U)$ as $\lambda \to 1-$. This proves (2).

\par

It remains to prove (1). Let $F\in A(\omega ,\mascB )$ and $0<\lambda <1$.
By Cantor's diagonal principle it suffices to prove that there is a sequence
$F_j$ of polynomials which converges to $F(\lambda \cdo )$ narrowly as
$j\to \infty$. Since
$$
\nm {J_{F(\lambda \cdo ),\omega ,p}-J_{F_j,\omega ,p}}{L^q(U)}\le
\nm {F(\lambda \cdo )-F_j}{A(\omega ,\mascB )},
$$
it suffices to prove that $F_j\to F(\lambda \cdo )$ in $A(\omega ,\mascB )$.
However, this fact is an immediate consequence of \eqref{addomegcond},
Proposition \ref{APprop} and (2). The proof is complete.
\end{proof}

\par

\begin{prop}\label{propnarrowseq}
Let $(\mathscr B,\omega )$ be a narrowly feasible pair on $\cc d$, $F_j,F\in A(\omega ,\mathscr B)$, $j=1,2,\dots$,
be such that $F_j\to F$ narrowly as $j\to \infty$, and let $G\in B(1/\omega ,\mathscr B')$. Then
$$
(F_j,G)_{B^2} \to (F,G)_{B^2}\quad \text{as}\quad j\to \infty .
$$
\end{prop}

\par

\begin{proof}
We may assume that $\omega =\omega _0$, where $\omega _0$ is the
same as in Definition \ref{narrowconvdef}, and we let $p$, $q$, $U$ and
$V$ be the same as in Definition \ref{narrowconvdef}. It follows from the
assumptions that
\begin{equation}\label{FjFpointw}
\lim _{j\to \infty}F_j(z)\overline {G(z)}e^{-|z|^2}= F(z) \overline{G(z)}e^{-|z|^2},
\end{equation}
and that
\begin{equation}\label{FjHjest}
\begin{aligned}
|S(F_j\, \overline {G}e^{-|\cdo |^2})(z)| &\le J_{F_j,\omega _0,p}(\zeta _2)
|S(G e^{-|\cdo |^2/2})(z)|/\omega (z),
\\[1ex]
|S(F\, \overline {G}e^{-|\cdo |^2})(z)| &\le J_{F,\omega _0,p}(\zeta _2)
|S(G e^{-|\cdo |^2/2})(z)|/\omega (z).
\end{aligned}
\end{equation}
Furthermore, by H{\"o}lder's inequality we get
$$
J_{F_j,\omega _0,p}(\zeta _2) |S(G e^{-|\cdo |^2/2})(z)|/\omega (z)
\to J_{F,\omega _0,p}(\zeta _2) |S(G e^{-|\cdo |^2/2})(z)|/\omega (z)
$$
in $L^1(\cc d)$ as $j\to \infty$. A combination of
\eqref{FjFpointw}, \eqref{FjHjest} and Lemma \ref{genLeb} now
implies that $S(F_j\overline {G}e^{-|\cdo |^2})\to S(F \overline{G}e^{-|\cdo |^2})$ in
$L^1(\cc d)$ as $j\to \infty$, which implies
$$
(F_j,G)_{B^2} = \int F_j(z)\overline {G(z)}\, d\mu (z) \to \int F(z)\overline {G(z)}\, d\mu (z) = (F,G)_{B^2}\quad \text{as}\quad j\to \infty .
$$
The proof is complete.
\end{proof}

\par

\subsection{General properties for modulation spaces}
We finish the section by using Theorem \ref{mainthm} in order to carry over
basic results on $A(\omega ,\mathscr B)$ spaces into corresponding
result for modulation spaces.

\par

The following result is an immediate consequences of Theorems
\ref{mainthm},  \ref{dualthm} and \ref{densthm}. Here and in what follows
we let $\mathscr S_0(\rr d)$ be the vector space which consists of all
finite linear combinations of Hermite functions.

\par

\begin{thm}\label{densdualmodthm}
Let $\omega \in \mascP _Q^0(\cc d)$ be dilated suitable
and $\mascB $ be a mixed norm space on $\rr {2d}$
such that $\nu _2(\mascB )<\infty$. Then the following is true:
\begin{enumerate}
\item the $L^2$ form on $\mathcal S_{1/2}(\rr d)$ extends uniquely
to a continuous sesqui-linear form on 
$M(\omega ,\mascB )\times M(1/\omega ,\mascB ')$;

\vrum

\item the dual of $M(\omega ,\mascB )$ can be identified by $M(1/\omega ,\mascB ')$ 
through the extension of the $L^2$ form on $\mathcal S_{1/2}(\rr d)$;

\vrum

\item $\mathscr S_0(\rr d)$ is dense in $M(\omega ,\mascB )$, and dense
in $M(\omega ,\mascB ')$ with respect to the weak$^*$-topology.
\end{enumerate}
\end{thm}

\par

The definition of narrow convergence for elements in certain modulation spaces is the following. Here the functional which corresponds to $J_{F,\omega }$ is
$$
H_{f,\omega }(\zeta _2)\equiv \sup _{\zeta _1\in V_1}\big (
|V_\phi f(x,\xi )|\omega (x,\xi )  \big ),\qquad
z=(\zeta _1,\zeta _2)\in V_1 \oplus V_1^\bot = \rr {2d},
$$

\par

\begin{defn}\label{narrowconvmoddef}
Let $(\mascB ,\omega )$ be a narrowly feasible pair on $\rr {2d}$, let $p$ and $V$ be the same as in Definition \ref{feasiblpair}, and let $q=(p_2,\dots ,p_n)$ and $U=(V_2,\dots V_{n})$. Also let $f,f_j\in M(\omega ,\mascB )$, $j\ge 1$. Then $f_j$ is said to converge to $f$ \emph{narrowly} as $j\to \infty$, if the following conditions are fulfilled:
\begin{enumerate}
\item $f_j\to f$ in $\mathcal S'_{1/2}(\rr d)$ as $j\to \infty$;

\vrum

\item $H_{f_j,\omega ,p}\to H_{f,\omega ,p}$ in $L^q(U)$ as $j\to \infty$.
%, for some $\omega _0\in \mascP _Q^0(\rr {2d})$ which is equivalent to $\omega$.
\end{enumerate}
\end{defn}

\par

By \eqref{bargstft2} it follows that $H_{f,\omega} = (2\pi )^{-d/2}J_{\mathfrak Vf,\omega}$.
Consequently, $f_j\to f$ narrowly in $M(\omega ,\mathscr B)$
as $j\to \infty$, if and only if $\mathfrak Vf_j\to \mathfrak V f$
narrowly in $A(\omega ,\mathscr B)$ as $j\to \infty$.
The following result is therefore an immediate consequence of
Theorems \ref{mainthm}, \ref{narrowthm}
and Proposition \ref{propnarrowseq}.

\par

\begin{thm}\label{narrowmodthm}
Let $(\mascB ,\omega )$ be a narrowly feasible pair on $\rr {2d}$. 
Also let $f_j,f\in M(\omega ,\mathscr B)$, $j=1,2,\dots$,
be such that $f_j\to f$ narrowly as $j\to \infty$, and let $g\in
M(1/\omega ,\mathscr B')$. Then the following is true:
\begin{enumerate}
\item $\mathscr S_0(\rr d)$ is dense in $M(\omega ,\mascB )$
with respect to the narrow convergence;

\vrum

\item $(f_j,g)_{L^2}\to (f,g)_{L^2}$ as $j\to \infty$.
\end{enumerate}
\end{thm}

\par

In Section \ref{sec6} we shall use these results to form a pseudo-differential
calculus involving symbols and target distributions in background of the
modulation space theory presented here.

\par

%%%%%%%%%%%%%%%%%%%%%%%%%%%%%
\section{Some consequences}\label{sec5}
%%%%%%%%%%%%%%%%%%%%%%%%%%%%%

\par

In this section we show some consequences of the results in the previous sections. We start 
by establishing continuity properties of $\Pi _A$ on appropriate $B(\omega ,\mathscr B)$ 
spaces. From these results it follows 
that $A(\omega ,\mathscr B)$ is a retract of $B(\omega ,\mathscr B)$ under $\Pi _A$. Thereafter
we use this property for establishing interpolation properties for $A(\omega ,\mathscr B)$, and
continuity properties of Toeplitz operators.

\par

\subsection{Continuity properties of $\Pi _A$ on $B(\omega ,\mathscr B)$} We start with the 
following result related to Lemma \ref{projlemma} (3).

\par

\begin{thm}\label{projthm}
Let $\mascB $ be a mixed norm space on $\rr {2d}$ and assume that $\omega \in
\mascP _Q^0(\rr {2d})$ is dilated suitable. Then the map $\Pi _A$ is continuous from $B(\omega ,\mathscr B)$ to $A(\omega ,\mathscr B)$. In particular, $A(\omega ,\mathscr B)$ is a retract of $B(\omega ,\mathscr B)$ under $\Pi _A$.
\end{thm}

\par

\begin{proof}
By Remark \ref{remGenLemKern} it follows that the mappings
$$
F\mapsto L_F\quad \text{and} \quad F\mod N(\omega ,\mathscr B)
$$
from $A(\omega ,\mathscr B)$ and $B(\omega ,\mathscr B)/N(\omega ,\mathscr B)$
respectively to
$$
\sets {l_F}{F\in B(\omega ,\mathscr B)}\subseteq (A_*(1/\omega ,\mathscr B'))'
$$
are well-defined, continuous and bijective. By the open mapping theorem it follows
that the inverse of the map $F\mapsto F\mod N(\omega ,\mathscr B)$ is continuous and bijective, and that
\begin{equation}\label{Fmodest}
\nm {F_0}{A(\omega ,\mathscr B)}\le C\inf _{G\in N(\omega ,\mathscr B)}\nm
{F+G}{B(\omega ,\mathscr B)},\qquad F-F_0\in N(\omega ,\mathscr B),\quad F_0\in A(\omega ,\mascB ),
\end{equation}
for some constant $C$.

\par

Now let $F\in B(\omega ,\mathscr B)$, set $F_1=\Pi _AF$, and choose
$F_0\in A(\omega ,\mascB )$
such that $l_{F}=l_{F_0}$. Then $l_{F_0}=l_{F_1}$ on $P(\cc d)$, and Lemma
\ref{projlemma} and Remark \ref{derivrepr} shows that $\partial ^\alpha
F_0(0) = \partial ^\alpha F_1(0)$ for every multi-index $\alpha$.
Consequently, $F_0(z)=F_1(z)$, since both $F_0$ and $F_1$
are entire. Furthermore, $F-F_0\in N(\omega ,\mathscr B)$,
since the restriction of $\Pi _A$ on $A(\omega ,\mathscr B)$
is the identity map.

\par

By \eqref{Fmodest} we now get
$$
\nm {\Pi _AF}{A(\omega ,\mathscr B)} = \nm {F_0}{A(\omega ,\mathscr B)}
\le C\inf _{G\in N(\omega ,\mathscr B)}\nm {F+G}{B(\omega ,\mathscr B)}
\le C\nm F{B(\omega ,\mathscr B)},
$$
and the assertion follows.
\end{proof}

\par

We recall that (real and complex) interpolation properties carry over
under retracts. These properties also include interpolation techniques
with more than two spaces involved. (Cf. \cite{AsKr1, AsKr2, BLo}.) In
particular, the following result is an immediate consequence of Theorem \ref{projthm}.
Here recall that  $(B_1,B_2)_{[\theta ]}$ denotes the complex interpolation space
with respect to $\theta \in [0,1]$, when $(B_1,B_2)$ is a compatible pair.
%Here and in what follows we 
%set
%$$
%1/p=(1/p_1,\dots ,1/p_n),\quad tp=(tp_1,\dots ,tp_n)\quad \text{and}
%\quad p+q=(p_1+q_1,\dots ,p_n+q_n),
%$$
%when $t\in \mathbf R$, $p=(p_1,\dots ,p_n)\in [1,\infty ]$
%and $q=(q_1,\dots ,q_n)\in [1,\infty ]$.

\par

\begin{prop}\label{interprop}
Any interpolation property valid for the $B(\omega ,\mascB)$ spaces also holds for the $A
(\omega ,\mascB)$ spaces, when the weights $\omega$ are dilated suitable and belong to 
$\mascP _Q^0(\cc d)$, and $\mascB$ are mixed norm spaces on $\cc d$. In particular, if $0\le \theta 
\le 1$, $(\omega ,\mascB _1)$ and $(\omega ,\mascB _2)$ are feasible pairs on $\cc d$, and 
that $\mascB=(\mascB _1, \mascB _2)_{[\theta ]}$, for $\theta \in [0,1]$, then
$$
(B(\omega ,\mascB _1),B(\omega ,\mascB _2))_{[\theta ]} = B(\omega ,\mascB )
\quad
\text{and}
\quad
(A(\omega ,\mascB _1),A(\omega ,\mascB _2))_{[\theta ]} = A(\omega ,\mascB ).
$$ 
\end{prop}

\par

Next we discuss further density properties of $A(\omega ,\mascB )$. We recall that $A_*(\omega ,\mascB )$ denotes the 
completion of $P(\cc d)$ under the norm $\nm \cdo {B(\omega ,\mascB )}$, when $\omega \in \mascP _Q(\cc d)$ and $\mascB$ 
is a mixed norm space on $\cc d$. We also let $B_*(\omega ,\mascB )$ be the completion of $C_0^\infty (\cc d)$ under the norm 
$\nm \cdo {B(\omega ,\mascB )}$. The following result links the $A_*(\omega ,\mascB )$ with $B_*(\omega ,\mascB )$.

\par

\begin{prop}\label{AstBstcap}
Let $\omega \in \mascP _Q^0(\cc d)$ be dilated suitable, and let $\mathscr B$ be a mixed norm space on $\cc d$. Then
$$
A(\cc d) \bigcap B_*(\omega ,\mascB ) = A(\omega ,\mascB ) \bigcap B_*(\omega ,\mascB ) = A_*(\omega ,\mascB ).
$$
\end{prop}

\par

\begin{proof}
Since $A(\omega ,\mascB )\subseteq A(\cc d)$, the result follows if we prove
\begin{equation}\label{AstBstincl}
A_*(\omega ,\mascB )\subseteq   A (\omega ,\mascB ) \bigcap B_*(\omega ,\mascB )
\quad \text{and}\quad
A(\cc d)\bigcap B_*(\omega ,\mascB )\subseteq A_*(\omega ,\mascB ).
\end{equation}
First we prove the first inclusion in \eqref{AstBstincl}. Since
$A_*(\omega ,\mascB ) \subseteq A(\omega ,\mascB )$, it
suffices to prove that $A_*(\omega ,\mascB ) 
\subseteq B_*(\omega ,\mascB )$.

\par

Let $F\in A_*(\omega ,\mascB )$. Then there is a sequence $F_j\in P(\cc d)$
such that $\nm {F-F_j}{A(\omega ,\mascB )}\to 0$ as 
$j\to \infty$. Hence \eqref{AstBstincl} follows if we prove that
for each $G\in P(\cc d)$, there are elements $G_j\in C_0^\infty (\cc 
d)$ such that $\nm {G-G_j}{B(\omega ,\mascB )}\to 0$ as $j\to \infty$.

\par

Let $\fy _j\in C_0^\infty (\cc d)$ be chosen such that $0\le \fy _j\le 1$ and $\fy _j(z)=1$ when
$|z|\le j$, and let $G_j=\fy _jG$. By H{\"o}lder's inequality, there is a constant $C$ such that
$$
\nm {G-G_j}{B(\omega ,\mascB )} \le C\sup _{z\in \cc d}\big (|G(z)-G_j(z)|e^{-|z|^2/4}\big )
\le C\big (\sup _{|z|\ge j}|G(z)|e^{-|z|^2/4}\big ) \to 0
$$
as $j\to \infty$. This gives the first inclusion in \eqref{AstBstincl}.

\par

In order to prove the second inclusion we instead assume that
$F\in A(\cc d )\cap B_* (\omega ,\mascB )$, and we
let $F_j\in C_0^\infty (\cc d)$ be a sequence such that
\begin{equation}\label{Fjlim2}
\nm {F-F_j}{B(\omega ,\mascB )}\to 0\quad \text{as}\quad  j\to \infty .
\end{equation}
By straight-forward computations it follows that $|\Pi _AF_j(z)|\le C_j
e^{C_j|z|}$, for some constants $C_j>0$, which implies that 
$\Pi _AF_j\in A_P(\cc d) \subseteq A_*(\omega ,\mascB )$.
Hence Theorem \ref{projthm} and \eqref{Fjlim2} give
$$
\nm {F-\Pi _AF_j}{B(\omega ,\mascB )} = \nm {\Pi _A(F-F_j)}{B(\omega ,\mascB )}
\le C\nm {F-F_j}{B(\omega ,\mascB )}\to 0\quad 
\text{as}\quad  j\to \infty ,
$$
for some constant $C$. Since any element in $A_P(\cc d)$ can
be approximated by elements in $P(\cc d)$
with respect to the norm $A(\omega ,\mascB )$, in view of Proposition \ref{APprop}, the result follows.
The proof is complete.
\end{proof}

\par

We may now extend \eqref{adjformula}. In fact, we have the following.

\par

\begin{prop}\label{adjprop}
Let $(\mascB ,\omega )$ be a feasible pair on $\cc d$, and let $F\in A(\omega ,\mascB )$ and
$G\in B(1/\omega ,\mascB ')$. Then \eqref{adjformula} holds.
%%%
%\begin{equation}\tag*{()$'$}
%(F,G)_{B^2} = (F, \Pi _A G)_{B^2} = (\Pi _AF,G)_{B^2}.
%\end{equation}
%%%
\end{prop}

\par

\begin{proof}
First assume that $\nu _2(\mathscr B)<\infty$. Then $A(\omega ,\mascB )=A_*(\omega ,\mascB )$. Let $F_j\in P(\cc d)$ be such that $F_j\to F$ in $A(\omega ,\mascB )$
as $j\to \infty$. Since $\Pi _AF=F$ and $\Pi _AF_j=F_j$, it follows from Lemma
\ref{projlemma} and Theorem \ref{projthm} that
\begin{equation}\label{scalarcomp1}
(\Pi _AF,G)_{B^2}=(F,G)_{B^2}=\lim _{j\to \infty}(F_j,G)_{B^2}
=\lim _{j\to \infty}(F_j,\Pi _AG)_{B^2} = (F,\Pi _AG)_{B^2},
\end{equation}
and the result follows in this case.

\par

Next we consider the case when $\nu _1(\mathscr B)>1$. Then $B(1/\omega ,\mascB ') = B_*(1/\omega ,\mascB ')$, and the result follows by similar arguments after
approximating $G$ with elements in $C^\infty _0$ and using the fact that
$$
|(\Pi _AG)(z)|\le Ce^{C|z|}\quad \text{for some}  \quad C >0,
$$
when $G\in C_0^\infty$ which is needed when applying Lemma
\ref{projlemma} (cf. \eqref{condadjoint}).

\par

Finally, if $(\mascB ,\omega )$ is narrowly feasible, then we choose
a sequence $F_j\in P(\cc d)$ which converges to $F$ narrowly as
$j\to \infty$. By Lemma \ref{projlemma}, Proposition \ref{propnarrowseq}
and Theorem \ref{projthm}, it follows that \eqref{scalarcomp1} holds
also in this case. The proof is complete.
\end{proof}

\par

The next result concerns convenient equivalent norms on $A(\omega ,\mascB)$.

\par

\begin{prop}\label{normequivprop}
Let $(\mascB , \omega )$ be a feasible pair on $\cc d$, and let
$$
\nmm F \equiv \sup |(F,G)_{A^2}|,\qquad F\in A(\omega ,\mathscr B),
$$
where the supremum is taken over all $G\in A(1/\omega ,\mascB ' )$ such that $\nm G{A(1/\omega ,\mascB ')}\le 1$.
Then $\nmm \cdo$ is a norm on $A(\omega ,\mascB )$ which is equivalent to $\nm \cdo{A(\omega ,\mascB )}$.
\end{prop}

\par

\begin{proof}
By H{\"o}lder's inequality we get $\nmm F \le \nm F{A(\omega ,\mascB )}$. We have to prove that
\begin{equation}\label{oppnormest}
\nm F{A(\omega ,\mascB )}\le C\nmm F,\qquad F\in A(\omega ,\mascB ),
\end{equation}
for some constant $C$ which is independent of $F\in A(\omega ,\mascB )$.

\par

%First we consider the case when $\nu _2(\mascB )<\infty$.
Let $\ep >0$ be fixed, and let $\Omega$ be the set of all
$G\in B(1/\omega ,\mascB ')$ such that $\nm G{B(1/\omega ,\mascB ')}\le 1$. Then the converse of H{\"o}lder's
inequality and Proposition \ref{adjprop} give
$$
\nm F{A(\omega ,\mascB )} =\sup _{G\in \Omega}|(F,G)_{B^2}|\le |(F,G_0)_{B^2}|+\ep = |(F,\Pi _AG_0)_{B^2}|+\ep ,
$$
for some choice of $G_0\in \Omega$. Since  $\Pi _A G_0\in A(1/\omega ,\mascB ')$ and
$\nm {\Pi _AG_0}{A(1/\omega ,\mascB ')}\le C\nm {G_0}{B(1/\omega ,\mascB ')}$ for some constant $C$, by Theorem 
\ref{projthm}, we obtain
$$
\nm F{A(\omega ,\mascB )} \le |(F,G_0)_{B^2}|+\ep = |(F,\Pi _AG_0)_{B^2}|+\ep \le C\sup |(F,G)_{B^2}|+\ep ,
$$
where the supremum is taken over all $G\in \Omega \cap A(1/\omega ,\mascB ')$. Since $\ep >0$
was arbitrary chosen, \eqref{oppnormest} follows. The proof is complete.
\end{proof}

\par

%We shall also show some consequences for Berezin-Toeplitz operators. Though there is already a djungle of different families of weight functions, we need to introduce two more such classes. 
%
%, and start by recalling their definition.
%
%\par
%
%\begin{defn}\label{defBerToeplop}
%Let $\omega ,\omega _0 \in \mascP _G^0(\cc d)$ be dilated suitable, $\mathscr B$ be mixed norm space on $\cc d$, and let $a\in 
%L^1_{loc}(\cc d)$ such that $(S^{-1}a)/\omega _0$, and let $S$ be as in \eqref{Sdef}. Then
%the Berezin-Toeplitz operator $\operatorname T_{\mathfrak V}(a )$ is
%the continuous operator on $A_{\mathscr S}'(\cc d)$, given by the
%formula
%$$
%\operatorname T_{\mathfrak V}(a )F = \Pi _A((S^{-1}a )F).
%$$
%\end{defn}
%
%\par

\subsection{Consequences for modulation spaces and Toeplitz operators}

\par

Next we use Theorem \ref{mainthm} to carry over the previous results for $A(\omega ,
\mascB )$ spaces into modulation spaces. First we need to investigate how $\Pi _A$ is 
linked to the composition $V_\phi \circ V_\phi ^*$, where $V_\phi ^*$ is the (Hilbert-)
adjoint of $V_\phi$. Here we let $V_\phi ^*F$ be the unique element 
in $\mathcal S_{1/2}'(\rr d)$ which satisfies
$$
(V_\phi ^*F,g)_{L^2(\rr d)} = (F,V_\phi g)_{L^2(\rr {2d})},\quad g\in \mathscr S_0(\rr d),
$$
when $F\in \mathcal S_{1/2}'(\rr {2d})$. We also let $\Pi _M$ be the continuous operator on $\mathcal S_{1/2}'(\rr {2d})$ given by $\Pi _M=V_\phi \circ V_\phi ^*$, and note that $\Pi _M$ is a projection from $\mathcal S_{1/2}'(\rr {2d})$ onto $V_\phi (\mathcal S_{1/2}'(\rr {d}))$.

\par

\begin{lemma}\label{PiAPiMlemma}
Let $\omega \in \mathscr P_Q^0(\cc d)$, and let $\mascB$ be a mixed norm space on $\rr {2d}$. Then
$$
\Pi _A = U_{\mathfrak V} \circ \Pi _M \circ U_{\mathfrak V}^{-1},
$$
on $B(\omega ,\mascB )$, where $U_{\mathfrak V}$ is given by \eqref{UVdef}.
\end{lemma}

\par

\begin{proof}
Let $F\in B(\omega ,\mascB )$, $F_0=U_{\mathfrak V}^{-1}F$, $g\in \mathscr S_0(\rr d)$, 
$G_0=V_\phi g$ and $G=U_{\mathfrak V}G_0$. Then $\Pi _MG_0=G_0$, and
\begin{multline*}
(F,G)_{B^2} = (U_{\mathfrak V}F_0,U_{\mathfrak V}G_0)_{B^2} = (F_0,G_0) _{L^2} = 
(F_0,V_\phi g) _{L^2}
\\[1ex]
= (V_\phi ^*F_0,g) _{L^2}
= ( \Pi _MF_0,V_\phi g) _{L^2} = ( \Pi _MF_0,G_0) _{L^2}
\\[1ex]
= (U_{\mathfrak V}(\Pi _M F_0),G)_{B^2} = ((U_{\mathfrak V}\circ
\Pi _M \circ U_{\mathfrak V}^{-1})F,G)_{B^2},
\end{multline*}
where
$U_{\mathfrak V}(\Pi _M F_0)=\mathfrak V(V_\phi ^*F_0)$ is analytic and satisfies \eqref
{condanal}. Furthermore, $(F,G)_{B^2}=(\Pi _AF,G)_{B^2}$ in view of Lemma \ref
{projlemma} (4). A combination of these equalities now gives
$$
((U_{\mathfrak V}\circ \Pi _M \circ U_{\mathfrak V}^{-1})F,G)_{B^2} = (\Pi _AF,G)_{B^2}.
$$
Since also $\Pi _AF$ is analytic it follows from Remark \ref{F1eqF2rem} that
$$
(U_{\mathfrak V}\circ \Pi _M \circ U_{\mathfrak V}^{-1})F = \Pi _AF,
$$
and the result follows.
\end{proof}

\par

The following result is now an immediate consequence of Theorems
\ref{mainthm} and  \ref{projthm}, and Lemma \ref{PiAPiMlemma}.
Here and in what follows we let $\mascB (\omega )$
be the set of all $f\in L^1(\rr d)$ such that $\nm f{\mascB (\omega )}\equiv \nm {f\omega}
{\mascB}$ is finite, when $\mascB$ is a mixed norm space on $\rr d$ and $\omega \in 
\mascP _Q(\rr d)$. In this situation we also set $\mascB '(\omega)=(\mascB ')(\omega )$.

\par

\begin{thm}\label{projthm2}
Let $\mascB $ be a mixed norm space on $\cc d$ and assume that $\omega \in \mascP _Q^0(\cc d)$ is dilated suitable. Then the following is true:
\begin{enumerate}
\item $\Pi _M$ is continuous from $\mathscr B(\omega )$ to
$V_\phi (M(\omega ,\mascB ))$;

\vrum

\item $V_\phi ^*$ is continuous from $\mathscr B(\omega )$ to
$M(\omega ,\mathscr B)$.
\end{enumerate}
In particular, $V_\phi (M(\omega ,\mascB ))$ is a retract of $B\mascB (\omega )$
under $\Pi _M$.
\end{thm}

\par

The next results are immediate consequences of Theorem \ref{mainthm},
Propositions \ref{interprop}, \ref{adjprop} and \ref{normequivprop},
and Lemma \ref{PiAPiMlemma}.

\par

\begin{prop}\label{interprop2}
Any interpolation property valid for the $B(\omega ,\mascB)$ spaces also hold for the $M
(\omega ,\mascB)$ spaces, when the weights $\omega$ are dilated suitable and belong to 
$\mascP _Q^0(\rr {2d})$, and $\mascB$ are mixed norm space on $\rr {2d}$. In particular, if $0\le 
\theta \le 1$, $(\omega ,\mascB _1)$ and $(\omega ,\mascB _2)$ are feasible pairs on $\rr 
{2d}$, and that $\mascB=(\mascB _1, \mascB _2)_{[\theta ]}$, for $\theta \in [0,1]$, then
$$
(M(\omega ,\mascB _1),M(\omega ,\mascB _2))_{[\theta ]} = M(\omega ,\mascB ).
$$ 
\end{prop}

\par

\begin{prop}\label{adjprop2}
Let $(\mascB ,\omega )$ be a feasible pair on $\rr {2d}$, and let
$f\in M(\omega ,\mascB )$ and $G\in \mascB '(1/\omega )$. Then
\begin{equation*}%\tag*{(\ref{adjformula})$'$}
(V_\phi f,G)_{L^2(\rr {2d})} = (f, V_\phi ^* G)_{L^2(\rr d)}.
\end{equation*}
\end{prop}

\par

\begin{prop}\label{normequivprop2}
Let $(\mascB , \omega )$ be a feasible pair on $\rr {2d}$, and let
$$
\nmm f \equiv \sup |(f,g)_{L^2}|,\qquad f\in M(\omega ,\mathscr B),
$$
where the supremum is taken over all $g\in M(1/\omega ,\mascB ' )$
such that $\nm g{M(1/\omega , \mascB ')}\le 1$.
Then $\nmm \cdo$ is a norm on $M(\omega ,\mascB )$ which is
equivalent to $\nm \cdo{M(\omega , \mascB )}$.
\end{prop}

\par

Next we consider Toeplitz operators and Berezin-Toeplitz operators. Let $a\in \mathcal S_{1/2}(\rr 
{2d})$ be  fixed. Then the \emph{Toeplitz operator} $\tp (a)$ is the linear and continuous operator on 
$\mathcal S_{1/2}(\rr d)$, defined by the formula
\begin{equation}\label{toeplitzdef}
(\tp  (a)f,g)_{L^2(\rr d)} = (a\, V_\phi f,V_\phi g)_{L^2(\rr {2d})}.
\end{equation}
There are several characterizations and several ways to extend the definition of 
Toeplitz and Berezin-Toeplitz operators (see e.{\,}g. \cite{Bau, Berezin71,
Cob01, CG03, Dau, Fo, GT1, GT2, SiT2, Toft2, To8B}
and the references therein). For example, the definition of $\tp  
(a)$ is uniquely extendable to every $a\in \mathcal S_{1/2}'(\rr {2d})$,
and then $\tp (a)$ is  continuous on $\mathcal S_{1/2}(\rr d)$. Furthermore,
it follows from \eqref{toeplitzdef} that $\tp (a)$ is uniquely extendable
to a continuous operator from $\Upsilon (\rr d)$ to $\mathcal S_{1/2}'(\rr d)$.

\par

Toeplitz operators  arise in pseudo-differential calculus in \cite{Fo,Le},
in the theory of quantization 
(Berezin quantization) in \cite{Berezin71}), and in signal processing in \cite{Dau}
(under the name of  time-frequency localization operators or STFT multipliers).
There are also strong connection between such operators and Berezin-Toeplitz
operators which we shall consider now.

\par

Let $\mascB _1,\mascB _2$ be mixed norm spaces on $\rr {2d}\simeq \cc d$,
$\omega _1,\omega _2\in \mascP _Q^0(\rr {2d})\simeq
\mascP _Q^0(\cc {d})$, $a\in \mascB _1(\omega _1)$, and $S$ be as in \eqref{Sdef}. Then
the \emph{Berezin-Toeplitz operator} $\operatorname T_{\mathfrak V}(a )$ is
the operator from $A(\omega _2,\mascB _2)$ to $A(\cc d)$, given by the formula
\begin{equation}\label{defBerToeplop}
\operatorname T_{\mathfrak V}(a )F = \Pi _A((S^{-1}a )F),\qquad F\in
A(\omega _2,\mascB _2).
\end{equation}

\par

It follows from \eqref{toeplitzdef} that if $a\in \mathcal S_{1/2}'(\rr {2d})$ and
$f\in \mathcal S_{1/2}'(\rr d)$, then
\begin{equation}\label{toeplitz2}
(V_\phi \circ \tp _\phi (a))f = \Pi _M(a\cdot F_0),\quad \text{where}\quad F_0=V_\phi f .
\end{equation}
Hence
\begin{equation}\label{Toeplconj}
\operatorname T_{\mathfrak V}(a ) \circ \mathfrak V = \mathfrak V \circ \tp (a),
\end{equation}
for appropriate $a\in \mathcal S_{1/2}'(\rr {2d})$, by  Lemma \ref{PiAPiMlemma}.

\par

We have now the following result. Here we recall that if
$p_j=(p_{j,1},\dots ,p_{j,n})\in [1,\infty ]^n$, $j=0,1,2$, then
$1/p_1+1/p_2=1/p_0$ means that $1/{p_{1,k}} +
1/{p_{2,k}}=1/{p_{0,k}}$ for every $k=1,\dots ,n$.

\par

\begin{prop}
Let $\mascB _j=L^{p_j}(V)$, $j=0,1,2$, be mixed norm space
such that $1/p_1+1/p_2=1/p_0$, and let $\omega _j\in
\mascP _D^0(\rr {2d})\simeq \mascP _D^0(\cc d)$ for $j=0,1,2$. Also let
$a\in \mascB _1(\omega _1)$. Then the following is true:
\begin{enumerate}
\item the Toeplitz operator $\tp (a)$ is continuous from $M(\omega _2,\mascB _2)$
to $M(\omega _0,\mascB _0)$;

\vrum

\item the Berezin-Toeplitz operator $\operatorname T_{\mathfrak V} (a)$
is continuous from $A(\omega _2,\mascB _2)$
to $A(\omega _0,\mascB _0)$.
\end{enumerate}
\end{prop}

\par

\begin{proof}
The assetion (2) follows immediately from the definitions and H{\"o}lder's
inequality, and (1) is then an immediate consequence of (2) and
\eqref{Toeplconj}. The proof is complete. 
\end{proof}

\par

%%%%%%%%%%%%%%%%%%%%%%%
\section{Pseudo-differential operators}\label{sec6}
%%%%%%%%%%%%%%%%%%%%%%%

\par

In this section we state results on pseudo-differential
operators in background of the modulation space theory
from the previous sections.
The proofs are in general omitted, since the results follows
by the same arguments as in \cite{To8A, To8B, To9}
in combination with the results in previous sections.

\par

\subsection{General properties of pseudo-differential operators}

\par

We start with the definition of pseudo-differential operators. Let
$t\in \mathbf R$ be fixed and let $a\in \mathcal S_{1/2}(\rr {2d})$. Then the
pseudo-differential operator $\op _t(a)$
with symbol $a$ is the linear and continuous operator on $\mathcal S_{1/2}(\rr d)$,
defined by the formula
\begin{equation}\label{e0.5}
(\op _t(a)f)(x)
=
(2\pi  ) ^{-d}\iint a((1-t)x+ty,\xi )f(y)e^{i\scal {x-y}\xi }\,
dyd\xi .
\end{equation}
The definition of $\op _t(a)$ extends to each $a\in \mathcal S_{1/2}'(\rr
{2d})$, and then $\op _t(a)$ is continuous from $\mathcal S_{1/2}(\rr d)$ to
$\mathcal S_{1/2}'(\rr d)$. (Cf. e.{\,}g. \cite {CPRT10}, and to some extent in \cite{Ho1}.) 
More precisely, for
any $a\in \mathcal S_{1/2}'(\rr {2d})$, the operator $\op _t(a)$
is defined as the linear and continuous
operator from $\mathcal S_{1/2}(\rr d)$ to $\mathcal S_{1/2}'(\rr d)$ with
distribution kernel given by
\begin{equation}\label{atkernel}
K_{a,t}(x,y)=(\mathscr F_2^{-1}a)((1-t)x+ty,x-y).
\end{equation}
Here $\mathscr F_2F$ is the partial Fourier transform of $F(x,y)\in
\mathcal S_{1/2}'(\rr {2d})$ with respect to the $y$ variable. This
definition makes
sense, since the mappings $\mathscr F_2$ and $F(x,y)\mapsto
F((1-t)x+ty,y-x)$ are homeomorphisms on $\mathcal S_{1/2}'(\rr
{2d})$.

\par

On the other hand, let $T$ be an arbitrary linear and continuous
operator from $\maclS (\rr d)$ to $\maclS '(\rr d)$. Then it follows from
Theorem 2.2 in \cite{LozPerTask} that for some
$K =K_T\in \maclS '_{1/2}(\rr {2d})$ we have
$$
(Tf,g)_{L^2(\rr d)} = (K,g\otimes \overline f )_{L^2(\rr {2d})},
$$
for every $f,g\in \maclS (\rr d)$. Now by letting $a$ be defined by
\eqref{atkernel} after replacing $K_{a,t}$ with $K$ it follows that
$T=\op _t(a)$. Consequently, the map $a\mapsto \op _t(a)$ is bijective from
$\maclS _{1/2}'(\rr {2d})$ to $\mathscr L(\maclS (\rr d),\maclS '(\rr d))$.

\par

If $t=1/2$, then $\op _t(a)$ is equal to the Weyl
quantization $\op ^w(a)$ of $a$. If instead $t=0$, then the standard
(Kohn-Nirenberg) representation $a(x,D)$ is obtained.

\par

In particular, if $a\in \maclS _{1/2}'(\rr {2d})$ and $s,t\in
\mathbf R$, then there is a unique $b\in \maclS _{1/2}'(\rr {2d})$ such that
$\op _s(a)=\op _t(b)$. By straight-forward applications of Fourier's
inversion  formula, it follows that
\begin{equation}\label{pseudorelation}
\op _s (a)=\op _t(b) \quad \Longleftrightarrow \quad b(x,\xi )=e^{i(t-s)\scal
{D_x}{D_\xi}}a(x,\xi ).
\end{equation}
(Cf. Section 18.5 in \cite{Ho1}.) Note here that the right-hand side makes sense,
because $e^{i(t-s)\scal
{D_x }{D_\xi }}$ on the Fourier transform side is a multiplication by
the function $e^{i(t-s)\scal x \xi }$, which is a continuous
operation on $\maclS _{1/2}'(\rr {2d})$, in view of the definitions.

\par

Let $t\in \mathbf R$ and $a\in \mathcal S_{1/2}'(\rr {2d})$ be
fixed. Then $a$ is called a rank-one element with respect to $t$, if
the corresponding pseudo-differential operator is of rank-one,
i.{\,}e.
\begin{equation}\label{trankone}
\op _t(a)f=(f,f_2)_{L^2} f_1, 
\end{equation}
for some $f_1,f_2\in \mathcal S_{1/2}'(\rr d)$. Here $f\in \mathcal S_{1/2}(\rr d)$. By
straight-forward computations it follows that \eqref{trankone}
is fulfilled, if and only if $a=(2\pi
)^{d/2}W_{f_1,f_2}^{t}$, where  the $W_{f_1,f_2}^{t}$ $t$-Wigner
distribution, defined by the formula
\begin{equation*}%\label{wignertdef}
W_{f_1,f_2}^{t}(x,\xi ) \equiv \mathscr F(f_1(x+t\cdo
)\overline{f_2(x-(1-t)\cdo )} )(\xi ),
\end{equation*}
which takes the form
$$
W_{f_1,f_2}^{t}(x,\xi ) =(2\pi )^{-d/2} \int
f_1(x+ty)\overline{f_2(x-(1-t)y) }e^{-i\scal y\xi}\, dy,
$$
when $f_1,f_2\in \mathcal S_{1/2}(\rr d)$. By combining these facts
with \eqref{pseudorelation}, it follows that
\begin{equation*}%\label{wignertransf}
W_{f_1,f_2}^{t} = e^{i(t-s)\scal {D_x }{D_{\xi}}} W_{f_1,f_2}^{s},
\end{equation*}
for each $f_1,f_2\in \mathcal S_{1/2}'(\rr d)$ and $s,t\in \mathbf R$. Since
the Weyl case is important to us, we set
$W_{f_1,f_2}^{t}=W_{f_1,f_2}$ when $t=1/2$. Then
$W_{f_1,f_2}$ is the usual (cross-) Wigner distribution of $f_1$ and
$f_2$.

\medspace

Next we discuss the Weyl
product, twisted convolution and related objects. Let $a,b\in
\mathcal S_{1/2}'(\rr {2d})$ be appropriate. Then the Weyl product $a\wpr
b$ between $a$
and $b$ is the function or distribution which fulfills $\op ^w(a\wpr
b) = \op ^w(a)\circ \op ^w(b)$, provided the right-hand side
makes sense. More generally, if $t\in \mathbf R$, then the product $\wpr
_t$ is defined by the formula
\begin{equation}\label{wprtdef}
\op _t(a\wpr _t b) = \op _t(a)\circ \op _t(b),
\end{equation}
provided the right-hand side makes sense as a continuous operator from
$\mathcal S_{1/2}(\rr d)$ to $\mathcal S_{1/2}'(\rr d)$.

\par

The Weyl product can also, in a convenient way, be expressed in terms
of the symplectic Fourier transform and twisted convolution. More
precisely, the \emph{symplectic Fourier transform} for $a \in
\mathcal{S} _{1/2}(\rr {2d})$ is defined by the formula
\begin{equation*}
(\mathscr{F}_\sigma a) (X)
= \pi^{-d}\int a(Y) e^{2 i \sigma(X,Y)}\,  dY.
\end{equation*}
Here $\sigma$ is the symplectic form, which is defined by
$$
\sigma(X,Y) = \scal y \xi -
\scal x \eta ,\qquad X=(x,\xi )\in \rr {2d},\ Y=(y,\eta )\in \rr {2d},
$$
where $\langle \cdot,\cdot \rangle$ denotes the usual scalar product
on $\rr d$, as before.

\par

It follows that ${\mathscr{F}_\sigma}$ is continuous on
$\mathcal S_{1/2}(\rr {2d})$, and extends as
usual to a homeomorphism on $\mathcal {S}_{1/2}'(\rr {2d})$, and to a
unitary map on $L^2(\rr {2d})$. Furthermore, $\mathscr{F}_\sigma^{2}$
is the identity operator.

\par

Let $a,b\in \mathcal S_{1/2}(\rr {2d})$. Then the \emph{twisted
convolution} of $a$ and $b$ is defined by the formula
\begin{equation*}%\label{twist1}
(a \ast _\sigma b) (X)
= (2/\pi)^{d/2} \int a(X-Y) b(Y) e^{2 i \sigma(X,Y)}\, dY.
\end{equation*}
The definition of $*_\sigma$ extends in different ways. For example,
it extends to a continuous multiplication on $L^p_{(v)}(\rr {2d})$ when $p\in
[1,2]$ when $v\in \mathscr P _E(\rr {2d})$ is submultiplicative (cf. \cite{To9}),
and to a continuous map from $\mathcal S_{1/2}'(\rr {2d})\times
\mathcal S_{1/2}(\rr {2d})$ to $\mathcal S_{1/2}'(\rr {2d})\cap C^\infty (\rr {2d})$. If $a,b \in
\mathcal {S}_{1/2}'(\rr {2d})$, then $a \wpr b$ makes sense if and only if $a
*_\sigma \widehat b$ makes sense, and then
\begin{equation}\label{tvist1}
a \wpr b = (2\pi)^{-d/2} a \ast_\sigma (\mathscr F_\sigma {b}).
\end{equation}
We also remark that for the twisted convolution we have
\begin{equation}\label{weylfourier1}
\mathscr F_\sigma (a *_\sigma b) = (\mathscr F_\sigma a) *_\sigma b =
\check{a} *_\sigma (\mathscr F_\sigma b),
\end{equation}
where $\check{a}(X)=a(-X)$ (cf. \cite{To1,To3}). A
combination of \eqref{tvist1} and \eqref{weylfourier1} gives
\begin{equation*}%\label{weyltwist2}
\mathscr F_\sigma (a\wpr b) = (2\pi )^{-d/2}(\mathscr F_\sigma
a)*_\sigma (\mathscr F_\sigma b).
\end{equation*}

\par

For admissible $a,b,c\in \mathcal S_{1/2}'(\rr {2d})$, it follows by straight-forward computations that
\begin{alignat*}{3}
(a_1*_\sigma a_2,b) &=  & (a_1,b*_\sigma \widetilde a_2) &= (a_2,\widetilde
a_1*_\sigma b),& \qquad (a_1*_\sigma a_2)*_\sigma b &= a_1*_\sigma
(a_2*_\sigma b)
\\[1ex]
(a_1\wpr a_2,b) &= & (a_1,b\wpr \overline{a_2}) &= (a_2,\overline
{a_1}\wpr b),& \qquad (a_1\wpr a_2)\wpr b &= a_1\wpr
(a_2\wpr b).
\end{alignat*}

\par

\subsection{Pseudo-differential operators and modulation spaces}

\par

Next we consider questions on Weyl quantizations
of pseudo-differential operators in the context of modulation space theory.
It is then convenient to use the symplectic
Fourier transform and the symplectic short-time Fourier transform, instead
of corresponding "ordinary" transformations. Here the symplectic short-time Fourier
transform of $a \in \mathcal {S}_{1/2}'(\rr {2d})$ with respect to the
window function $\Psi \in \mathcal{S'}_{1/2}(\rr {2d})$ is defined
by
\begin{equation}\nonumber
\mathcal V_{\Psi} a(X,Y) = \mathscr{F}_\sigma \big( a\, \Psi (\cdo -X)
\big) (Y),\quad X,Y \in \rr {2d}.
\end{equation}
Let $(\mascB,\omega)$ be an admissible pair on $\rr {4d}$. Then
$\mathcal M(\omega ,\mascB )$ denotes the modulation spaces of
Gelfand-Shilov distributions on $\rr {2d}$, where
the symplectic short-time Fourier
transform is used instead of the usual short-time Fourier transform
and the window function $\Psi (X)$ here
above is equal to
$$
\Phi (X) = (2/\pi )^{d/2}e^{-|X|^2},
$$
in the definitions of the norms. In a way similar as for the usual
modulation spaces, we set
$$
\mathcal M^{p,q}_{(\omega )}(\rr {2d})=\mathcal M(\omega ,L^{p,q}(\rr {4d})),
$$
when $p,q\in [1,\infty ]$. It follows that any property valid for the
modulation spaces in previous sections
carry over to spaces of the form $\mathcal M(\omega ,\mascB )$.

\par

The choice of window function is motivated by the simple
form that the following results attain.
For the proof we refer to the results in previous sections in combination
with the proof of Proposition 4.1 in \cite{To8A}.

\par

\begin{prop}\label{p5.3}
Let $p_j,q_j,p,q\in [1,\infty ]$ such that $p\le p_j,q_j$ $\le
q$, for $j=1,2$, and that
\begin{equation*}%\label{e5.6}
1/p_1+1/p_2=1/q_1+1/q_2=1/p+1/q.
\end{equation*}
Also let $\omega _1,\omega _2\in \mascP _Q^0(\rr {2d})$ and
$\omega \in \mathscr P_Q^0(\rr {4d})$  be such that
$(L^{p_j,q_j}(\rr {2d}),\omega _j)$, $j01,2$, and
$(L^{p,q}(\rr {4d}),\omega )$ are admissible pairs and satisfy
\begin{equation*}
\omega (X,Y) \le C\omega _1(X-Y)\omega _2(X+Y).
\end{equation*}
Then the map $(f_1,f_2)\mapsto W_{\! f_1,f_2}$ from $\mathcal S_{1/2}'(\rr
d)\times \mathcal S_{1/2}'(\rr d)$ to $\maclS _{1/2}'(\rr {2d})$ restricts to a
continuous mapping from $M^{p_1,q_1}_{(\omega _1)}(\rr d)\times
M^{p_2,q_2}_{(\omega _2)}(\rr d)$ to $\mathcal M^{p,q}_{(\omega
)}(\rr {2d})$, and
\begin{equation*}%\label{e5.7}
\nm {W_{\! f_1,f_2}}{\mathcal M^{p,q}_{(\omega )}}
\le
C\nm {f_1}{M^{p_1,q_1}_{(\omega _1)}} \nm
{f_2}{M^{p_2,q_2}_{(\omega _2)}},
\end{equation*}
where the constant $C$ is independent of $f_j\in M^{p_j,q_j}_{(\omega _j)}(\rr d)$, $j=1,2$.
\end{prop}

\par

We now arrive on the following continuity result for pseudo-differential operators.
Again we omit the proof, since the result is a straight-forward consequence of
Proposition 4.12 in \cite{To8A} and its proof in combinations with the results
on modulation spaces in previous sections.

\par

\begin{thm}\label{p5.4}
Let $p,q,p_j,q_j\in [1,\infty ]$ for
$j=1,2$, be such that
\begin{equation}\label{e5.4lebexp}
\frac 1{p_2}-\frac 1{p_1}=\frac 1{q_2}-\frac 1{q_1}=\frac
1{p}+\frac 1{q}-1,\quad q\le p_2,q_2\le p.
\end{equation}
Also let $\omega \in
\mathscr P_Q^0(\rr {4d})$ and $\omega _1,\omega _2\in
\mathscr P_Q^0(\rr {2d})$ be such that $(L^{p_j,q_j}(\rr {2d}),\omega _j)$
and $(L^{p,q}(\rr {4d}),\omega )$ are admissible pairs and satisfy
\begin{equation}\label{e5.9}
\frac {\omega _2(X-Y)}{\omega _1
(X+Y)} \le C \omega (X,Y),
\end{equation}
for some constant $C>0$. If $a\in \mathcal M^{p,q}_{(\omega )}(\rr
{2d})$, then $\op ^w(a)$ from $\mathcal S_{1/2}(\rr d)$ to $\mathcal S_{1/2}'(\rr
d)$ extends uniquely to a continuous map from
$M^{p_1,q_1}_{(\omega _1)}(\rr d)$ to $M^{p_2,q_2}_{(\omega _2)}(\rr
d)$.

\par

Moreover, if in addition $a$ belongs to the closure of $\mathscr S_0(\rr {2d})$
under the norm $\nm \cdo {\mathcal M ^{p,q}_{(\omega )}}$, then
$$
\op ^w(a)\, :\, M^{p_1,q_1}_{(\omega _1)}(\rr d)\to
M^{p_2,q_2}_{(\omega _2)}(\rr d)
$$
is compact.
\end{thm}

\par

Next we consider algebraic properties of modulation spaces under twisted convolution and Weyl product.
These investigations are based on the following lemma together with the observations
$$
\Phi \wpr \Phi =\pi ^{-d}\Phi\quad \text{and}\quad \Phi *_\sigma \Phi = 2^d\Phi
$$
(cf. \cite{To3}). We refer to \cite{To2,To9} for its proof.

\par

\begin{lemma}\label{cornerstone}
Assume that $a_1\in \Upsilon (\rr {2d})$, $a_2 \in \mathcal
S_{1/2}(\rr {2d})$, $\Phi (X)=(2/\pi )^{d/2}e^{-|X|^2}$ and $X,Y\in \rr
{2d}$. Then the following is true:
\begin{enumerate}
\item the map
$$
Z\mapsto e^{2 i \sigma(Z,Y)} (\mathcal V_{\Phi} a_1) (X-Y+Z,Z) \,
(\mathcal V_{\Phi} a_2)(X+Z,Y-Z)
$$
belongs to $L^1(\rr {2d})$, and
\begin{multline*}
\mathcal V_{\Phi} (a_1 \wpr a_2) (X,Y)
\\[1ex]
=
\int  e^{2 i \sigma(Z,Y)} (\mathcal V_{\Phi} a_1)(X-Y+Z,Z) \,
(\mathcal V_{\Phi} a_2)
(X+Z,Y-Z) \, dZ\text ;
\end{multline*}

\vrum

\item the map
$$
Z\mapsto e^{2 i \sigma(X,Z-Y)} (\mathcal V_{\Phi} a_1) (X-Y+Z,Z) \,
(\mathcal V_{\Phi} a_2)(Y-Z,X+Z)
$$
belongs to $L^1(\rr {2d})$, and
\begin{multline*}
\mathcal V_{\Phi } (a_1 *_\sigma a_2) (X,Y)
\\[1ex]
=
\int  e^{2 i \sigma(X,Z-Y)} (\mathcal V_{\Phi} a_1)(X-Y+Z,Z) \,
(\mathcal V_{\Phi}a_2)(Y-Z,X+Z) \, dZ.
\end{multline*}
\end{enumerate}
\end{lemma}

\par

The following two theorems now follows by combining the results in
previous sections and using similar arguments as in the
proofs of Theorem 0.3$'$ in \cite{HTW} and Theorem 2.2
in \cite{To9}. The details are left for the reader. Here and in what follows,   
the involved Lebesgue exponents should satisfy conditions of the form
\begin{align}
\frac {1}{p_1}+\frac {1}{p_2}-\frac {1}{p_0} &= 1-\Big (\frac
{1}{q_1}+\frac {1}{q_2}-\frac {1}{q_0}\Big ),\label{pqformulas1}
\\[1ex]
0\le \frac {1}{p_1}+\frac {1}{p_2}-\frac
{1}{p_0} &\le \frac {1}{p_j},\frac {1}{q_j}\le \frac {1}{q_1}+\frac
{1}{q_2}-\frac {1}{q_0} ,\quad j=0,1,2,\label{pqformulas2}
\intertext{and}
0\le \frac {1}{q_1}+\frac {1}{q_2}-\frac
{1}{q_0} &\le \frac {1}{p_j},\frac {1}{q_j}\le \frac {1}{p_1}+\frac
{1}{p_2}-\frac {1}{p_0},\quad j=0,1,2.\label{pqformulas3}
\end{align}
Furthermore, the involved weight functions should satisfy
\begin{equation}\label{vikt1}
\omega_0(X,Y) \leq C \omega_1(X-Y+Z,Z)
\omega_2(X+Z,Y-Z),\quad \! \! X,Y,Z\in \rr {2d},
\end{equation}
or
\begin{equation}\label{vikt2}
\omega_0(X,Y) \leq C \omega_1(X-Y+Z,Z)
\omega_2(Y-Z,X+Z),\quad \! \! X,Y,Z\in \rr {2d}.
\end{equation}

\par

\begin{thm}\label{algthm1}
Let $\omega _j\in \mathscr P_Q^0(\rr {4d})$
and $p_j, q_j \in
[1,\infty]$ be such that $(L^{p_j,q_j}(\rr {4d}),\omega _j)$
are admissible pairs for $j=0,1,2$, and
\eqref{pqformulas1}, \eqref{pqformulas2} and \eqref{vikt1} hold.
Then the map $(a_1,a_2)\mapsto a_1\wpr a_2$ on  $\mathscr  S_0(\rr{2d})$
extends uniquely to a continuous map from $\mathcal M_{(\omega
_1)}^{p_1,q_1}(\rr {2d}) \times \mathcal M_{(\omega _2)}^{p_2,q_2}(\rr
{2d})$ to $\mathcal M_{(\omega _0)} ^{p_0,q_0}(\rr {2d})$, and
$$
\| a_1 \wpr a_2 \|_{\mathcal M_{(\omega _0)}^{p_0,q_0} }
\leq C \| a_1 \|_{\mathcal M_{(\omega _1)}^{p_1,q_1} } \| a_2
\|_{\mathcal M_{(\omega _2)}^{p_2,q_2} },
$$
where the constant $C$ is independent of $a_j\in
\mathcal M_{(\omega _j)}^{p_j,q_j}(\rr {2d})$, $j=1,2$.
\end{thm}

\par

\begin{thm}\label{algthm2}
Let $\omega _j\in \mathscr P_Q^0(\rr {4d})$
and $p_j, q_j \in
[1,\infty]$ be such that $(L^{p_j,q_j}(\rr {4d}),\omega _j)$
are admissible pairs for $j=0,1,2$, and
\eqref{pqformulas1}, \eqref{pqformulas3} and \eqref{vikt2} hold.
Then the map $(a_1,a_2)\mapsto a_1*_\sigma a_2$
on $\mathscr S_0(\rr {2d})$ extends uniquely to a continuous map from
$\mathcal W_{(\omega _1)}^{p_1,q_1}(\rr {2d}) \times \mathcal
W_{(\omega _2)}^{p_2,q_2}(\rr {2d})$ to $\mathcal W_{(\omega
_0)}^{p_0,q_0}(\rr {2d})$, and
$$
\| a_1 *_\sigma a_2 \|_{\mathcal W_{(\omega _0)}^{p_0,q_0} }
\leq C \| a_1 \|_{\mathcal W_{(\omega _1)}^{p_1,q_1} } \| a_2
\|_{\mathcal W_{(\omega _2)}^{p_2,q_2} },
$$
where the constant $C$ is independent of $a_j\in \mathcal W_{(\omega
_j)}^{p_j,q_j}(\rr {2d})$, $j=1,2$.
\end{thm}

\par

\begin{rem}\label{weightconcondrem}
We note that $\omega _j$, $j=0,1,2$, fulfills all the required
properties in Theorem \ref{algthm1}, if
\begin{equation}\label{weightconcond}
\begin{aligned}
\omega _0(X,Y) &= \frac {\nu _3(X-Y)}{\nu _1(X+Y)},\qquad
\omega _1(X,Y)=\frac {\nu _2(X-Y)}{\nu _1(X+Y)},
\\[1ex]
\omega _0(X,Y) &= \frac {\nu _3(X-Y)}{\nu _2(X+Y)}.
\end{aligned}
\end{equation}
for some appropriate $\nu _1,\nu _2, \nu _3\in \mascP _Q^0(\rr {2d})$. Note here that such types of conditions appears for $\omega$ in Theorem \ref{p5.4}.
\end{rem}

\par

\subsection{Examples on calculi of pseudo-differential operators}

\par

Next we give some examples on symbol classes and continuity properties
for corresponding pseudo-differential operators. We are especially focused
on symbol classes of the form $\mathcal M^{\infty ,1}_{(\omega )}(\rr {2d})$, because
they are to some extend linked to certain classical symbol classes in the
pseudo-differential calculus. We have for example that
\begin{equation}\label{classicsymbmod}
S^{(\omega )}(\rr {2d}) =\bigcap _{N\ge 0}\mathcal M^{\infty ,1}_{(1/\omega _N)}(\rr {2d}),
\end{equation}
where $\omega \in \mascP (\rr {2d})$ and $\omega _N(X,Y)=\omega
(X)\eabs Y^{-N}$. Here $S^{(\omega )}(\rr {2d})$ is the symbol class
which consists of all $a\in C^\infty (\rr {2d})$ such that
$(\partial ^\alpha a)/\omega \in L^\infty (\rr {2d})$
for every multi-index $\alpha$. (Cf. \cite[Rem. 2.18]{HTW}.)

\par

We also remark that \eqref{classicsymbmod} can be used to carry over properties
valid for modulation spaces into $S^{(\omega )}$ spaces. For example, in
\cite[Rem. 2.18]{HTW} it is proved that Theorem \ref{algthm1} and
\eqref{classicsymbmod} imply $S^{(\omega _1)}
\wpr S^{(\omega _2)}\subseteq S^{(\omega _1\omega _1)}$ when
$\omega _1,\omega _2\in \mascP$. (See \cite[Sec. 18.5]{Ho1}
for an alternative proof of the latter fact.)

\medspace

As a consequence of Theorems  \ref{p5.4} and \ref{algthm1}, and
Remark \ref{weightconcondrem} we have the following. 

\par

\begin{prop}\label{contalgprop}
Let $p,q\in [1,\infty ]$, $\nu _1,\nu _2,\nu _3\in \mascP _D^0(\rr {2d})$ and
$\omega _0,\omega _1,\omega _2\in \mascP _D^0(\rr {4d})$ be such
that \eqref{weightconcond} is fulfilled. Then the following is true:
\begin{enumerate}
\item if $a_j\in \mathcal M^{\infty ,1}_{(\omega _j)}(\rr {2d})$, then the mappings
$$
\op ^w(a_1)\, :\, M^{p,q}_{(\nu _1)}(\rr d)\to M^{p,q}_{(\nu _2)}(\rr d),\quad
\op ^w(a_2)\, :\, M^{p,q}_{(\nu _2)}(\rr d)\to M^{p,q}_{(\nu _3)}(\rr d)
$$
are continuous;

\vrum

\item the map $(a_1,a_2)\mapsto a_1\wpr a_2$ is continuous from
$\mathcal M^{\infty ,1}_{(\omega _1)}(\rr {2d})\times  \mathcal M^{\infty ,1}_{(\omega _2)}(\rr {2d})$
to $\mathcal M^{\infty ,1}_{(\omega _0)}(\rr {2d})$.
\end{enumerate}
\end{prop}

\par

\begin{cor}\label{contalgcor}
Let $p,q\in [1,\infty ]$, $\nu \in \mascP _D^0(\rr {2d})$ and
\begin{equation}\label{weightconcond2}
\omega (X,Y)=\frac {\nu (X-Y)}{\nu (X+Y)}.
\end{equation}
Then the following is true:
\begin{enumerate}
\item if $a\in \mathcal M^{\infty ,1}_{(\omega )}(\rr {2d})$, then $\op ^w(a)$
is continuous on $M^{p,q}_{(\nu )}(\rr d)$;

\vrum

\item $(\mathcal M^{\infty ,1}_{(\omega )}(\rr {2d}),\wpr )$ is an algebra.
\end{enumerate}
\end{cor}

\par

\begin{proof}
The result follows by letting $\nu _1=\nu _2 = \nu _3 =\nu$ and
$\omega _0=\omega _1=\omega _2 = \omega$ in Proposition \ref{contalgprop}.
\end{proof}

\par

\begin{example}
Let $\nu (X)=e^{c|X|^\gamma}$, for $0\le \gamma <2$ and some constant
$c\in \mathbf R$. In this case, $\omega$ in
\eqref{weightconcond2} is given by
$$
\omega (X,Y) =e^{c(|X-Y|^\gamma -|X+Y|^\gamma )}.
$$

\par

In the case $\gamma \le 1$ one may use the inequality
$\omega (X,Y)\le e^{2|c|\cdot |Y|^\gamma}$ to conclude
that $\op ^w(a)$ is continuous on $M^{p,q}_{(\nu )}(\rr d)$, when
$a\in \mathcal M^{\infty ,1}_{(\omega _0)}$ and
$\omega _0(X,Y) = e^{2|c|\cdot |Y|^\gamma}$.

\par

More generally, let $\nu _j(X)=e^{c_j|X|^\gamma}$, for $0\le \gamma <2$ and some constants
$c_j\in \mathbf R$, $j=1,2$. In this case, $\omega _1$ in
\eqref{weightconcond} is given by
$$
\omega _1(X,Y) =e^{c_2|X-Y|^\gamma -c_1|X+Y|^\gamma }.
$$

\par

In the case $\gamma \ge 1$, $c_1=2^{\gamma -1}$ and $c_2=1$ we have
$$
|X-Y|^\gamma -2^{\gamma -1} |X+Y| \le (|X+Y|+2|Y|)^\gamma -2^{\gamma -1}|X+Y|\le 2^{2\gamma -1} |Y|^\gamma ,
$$
and $\omega (X,Y)\le e^{2^{2\gamma -1}|Y|^\gamma}$. Hence, if
$$
\omega _0(X,Y)=e^{2^{2\gamma -1}|Y|^\gamma},\quad \nu _1(X)=e^{2^{\gamma -1}|X|^\gamma},
\quad \nu _2(X)=e^{|X|^\gamma},
$$
and $a\in \mathcal M^{\infty ,1}_{(\omega _0)}$, then $\op ^w(a)$
is continuous from $M^{p,q}_{(\nu _1)}$ to $M^{p,q}_{(\nu _2)}$;
\end{example}

\par

\begin{example}
Let $\nu _j(X)=\eabs X^{c_j\eabs X}$ or $\nu _j(X)=\Gamma {\eabs X}^{c_j}$ for some constant $c_j\in \mathbf R$, $j=1,2$.
In this case, $\omega _1$ in \eqref{weightconcond} is given by
\begin{align*}
\omega _1(X,Y) &= \eabs {X-Y}^{c_2\eabs {X-Y}}\eabs {X+Y}^{-c_1\eabs {X+Y}}
\intertext{or}
\omega _1(X,Y) &= \big (\Gamma ( \eabs {X-Y})\big )^{c_2} \big (\Gamma ( \eabs {X+Y})\big )^{-c_1}.
\end{align*}
Hence, if $a\in \mathcal M^{\infty ,1}_{(\omega _1)}$, then $\op ^w(a)$
is continuous from $M^{p,q}_{(\nu _1)}$ to $M^{p,q}_{(\nu _2)}$.

\par

We note that different situations appear depending on the sign on $c_1$ and $c_2$:
\begin{enumerate}
\item if $c_1>0$ and $c_2>0$, then the weights
$\nu _j(X)$, $j=1,2$, turn rapidly to infinity at infinity. This implies
that the target space $M^{p,q}_{(\nu _1)}$ as well as the image
space $M^{p,q}_{(\nu _2)}$ are small, in the sense that their
elements turns rapidly to zero at infinity, fulfill hard restrictions
on oscillations at infinity, and are extendable to entire functions on $\cc d$.

\par

The corresponding weight $\omega _1$ turns rapidly to zero as
$X=Y$ and $|X|\to \infty$, while $\omega _1$ turns rapidly to
infinity as $X=-Y$ and $|X|\to \infty$.
%Roughly speaking, this
%means that if $|a|$ turns rapidly to $\infty$ at infinity, in some directions,
%then $a$ has to fulfill strong regularity conditions in those directions.
%If instead $a$ has heavy singularities in some direction, then the symbol $a$ has to
%turn rapidly to zero at infinity in those directions;

\vrum

\item if $c_1<0$ and $c_2<0$ (i.{\,}e. the adjoint situation
comparing to (1)), then the target space $M^{p,q}_{(\nu _1)}$
as well as the image
space $M^{p,q}_{(\nu _2)}$ are large, in the sense that their
elements are allowed to turns rapidly to infinity at infinity, with small restrictions
on oscillations and singularities at infinity.

\par

The corresponding weight $\omega _1$ turns rapidly to zero as
$X=-Y$ and $|X|\to \infty$, while $\omega _1$ turns rapidly to
infinity as $X=Y$ and $|X|\to \infty$;

\vrum

\item if $c_1<0$ and $c_2>0$, then the target space $M^{p,q}_{(\nu _1)}$
is large and the image
space $M^{p,q}_{(\nu _2)}$ is small.

\par

The corresponding weight $\omega _1$ turns rapidly to infinity at infinity. Hence,
the symbols in $\mathcal M^{\infty ,1}_{(\omega _1)}$ turn rapidly to zero at infinity,
fulfill hard restrictions
on oscillations at infinity, and are extendable to entire functions on $\cc {2d}$;

\vrum

\item if $c_1>0$ and $c_2<0$, then the target space $M^{p,q}_{(\nu _1)}$
is small and the image
space $M^{p,q}_{(\nu _2)}$ is large.

\par

The corresponding weight $\omega _1$ turns rapidly to zero at infinity. Hence,
the symbols in $\mathcal M^{\infty ,1}_{(\omega _1)}$ are allowed to turn rapidly
to infinity at infinity, with small restrictions
on oscillations and singularities at infinity.
\end{enumerate}
\end{example}

\par

\subsection{The case of moderate weights}

\par

It follows from the general results in previous sections that almost all results on
pseudo-differential operators in \cite{To5,To8A} can be extended to include
weights in the class $\mascP _E$. In what follows we state these extensions,
and leave most of the verifications for the reader.

\par

We start with the following general form of Feichtinger-Gr{\"o}chenig's kernel theorem.
The proof is the same as in \cite[Prop. 4.7]{To8A}, where Theorem 2.2 in \cite{LozPerTask}
should be applied instead of the classical Schwartz kernel theorem.

\par

\begin{prop}\label{schwartzgrochenig}
Let $d=d_1+d_2$, $\omega _j\in \mathscr P_E(\rr {2d_j})$ for
$j=1,2$ and let $\omega \in \mascP _E(\rr {d}\oplus \rr {d})$ be such that
\begin{equation*}%\label{omegakvot}
\omega (x,y,\xi,\eta )=\omega _2(x,\xi )/\omega _1(y,-\eta ).
\end{equation*}
Also let $T$ be a linear and continuous map from $\maclS _{1/2}(\rr {d_1})$
to $\maclS _{1/2}'(\rr {d_2})$. Then $T$ extends to a
continuous mapping from $M^1_{(\omega _1)}(\rr {d_1})$ to $M^\infty
_{(\omega _2)}(\rr {d_2})$, if and only if it exists an element $K\in
M^\infty _{(\omega )}(\rr d)$ such that
\begin{equation*}%\label{opkernel}
(Tf)(x)=\scal {K(x,\cdot )}{f}.
\end{equation*}
\end{prop}

\par

For the proof of the following result we refer to \cite[Prop. 4.8]{To8A} and its proof.

\par

\begin{prop}\label{kernelweyl}
Let $t\in \mathbf R$, $a\in \maclS _{1/2}'(\rr {2d})$, and let $K\in \maclS _{1/2}'
(\rr {2d})$ be the distribution kernel for the pseudo-differential
operator $\op _t(a)$. Also let $p\in [1,\infty]$, and $\omega
,\omega _0\in \mathscr P_E(\rr {2d}\oplus \rr {2d})$ be such that
$$
\omega (x,\xi ,\eta ,y) = \omega _0(x-ty, x + (1-t)y ,-\xi +(1-t)\eta , \xi +t\eta ).
$$
Then $a\in M^p_{(\omega )}(\rr {2d})$ if and only if $K\in
M^p_{(\omega _0)}(\rr {2d})$. Moreover, if $\phi \in \maclS _{1/2}(\rr
{2d})$ and
$$
\psi (x,y)=\int {\phi ((1-t)x+ty,\xi )}e^{i\scal {x-y}\xi }\,
d\xi ,
$$
then $\nm a{M^{p,\phi }_{(\omega )}}= \nm K{M^{p,\psi
}_{(\omega _0)}}$.
\end{prop}

%
%$$
%|(V_\psi K)(x-ty, x + (1-t)y ,-\xi +(1-t)\eta , \xi +t\eta )| = |(V_\fy a)(x,\xi ,\eta ,y)|
%$$
%

\par

The next result shows that pseudo-differential operators with symbols in modulation are to
some extent invariant under the choice of $t$ in \eqref{e0.5}.
We refer to \cite[Prop. 1.7]{To8A} for the proof. Here we let $S_\Phi$ be the
linear and continuous map on $\maclS _{1/2}(\rr d)$ and on $\maclS _{1/2}'(\rr d)$, defined
by the formula
\begin{equation}\label{sfiavb}
f\mapsto S_\Phi f\equiv (e^{i\Phi}\otimes \delta _{V_2})*f,
\end{equation}
where $\delta _{V_2}$ is the delta function on the vector space
$V_2\subseteq \rr d$ and $\Phi$ is a real-valued and non-degenerate
quadratic form on $V_1=V_2^\bot$.

\par

\begin{prop}\label{p1.45}
Let $\phi \in \maclS _{1/2}(\rr d)$, $\omega \in \mathscr P_E(\rr
{2d})$, $p,q\in [1,\infty ]$, $V_1,V_2\subseteq \rr d$
be vector spaces such that $V_2=V_1^\bot$. Also let $\Phi$ be
a real-valued and non-degenerate
quadratic form on $V_1$, and let $A_\Phi /2$ be the corresponding
matrix. If $\xi =(\xi _1,\xi _2)$ where $\xi _j \in V_j$ for $j=1,2$,
then
\begin{equation*}%\label{normekv2}
\begin{aligned}
\nm {S_\Phi f}{M^{p,q,\phi}_{(\omega _\Phi)}} &= \nm
{f}{M^{p,q,\psi}_{(\omega )}},\quad \text{where}\quad f\in \maclS _{1/2}'(\rr d),
\\[1ex]
\omega _\Phi (x,\xi ) &= \omega  (x-A_\Phi ^{-1}\xi _1,\xi) \quad
\text{and}\quad \psi =S_\Phi \phi .
\end{aligned}
\end{equation*}
In particular, the following are true:
\begin{enumerate}
\item[{\rm{(1)}}] the map \eqref{sfiavb} on $\maclS _{1/2}'(\rr d)$
restricts to a homeomorphism from $M^{p,q}_{(\omega )}(\rr d)$ to
$M^{p,q}_{(\omega _\Phi)}(\rr d)$;

\vrum

\item[{\rm{(2)}}] if $t\in \mathbf R$, $\omega _0\in \mathscr P_E(\rr
{2d}\oplus \rr {2d})$, and
$$
\omega _t(x,\xi,\eta ,y)= \omega _0(x-ty,\xi -t\eta ,y,\eta ),
$$
then the map $e^{it\scal {D_x}{D_\xi}}$ on $\maclS _{1/2}'(\rr {2d})$
restricts to a homeomorphism from $M^{p,q}_{(\omega
_0)}(\rr {2d})$ to $M^{p,q}_{(\omega _t)}(\rr {2d})$.
\end{enumerate}
\end{prop}

\par

By combining Propositions \ref{schwartzgrochenig}--\ref{kernelweyl}
we get the following result (cf. \cite[Thm. 4.6]{To8A}).

\par

\begin{thm}\label{p5.4B}
Let $t\in \mathbf R$ and $p,q,p_j,q_j\in [1,\infty ]$ for
$j=1,2$, be such that \eqref{e5.4lebexp} holds.
Also let $\omega \in
\mathscr P_E(\rr {4d})$ and $\omega _1,\omega _2\in
\mathscr P_E(\rr {2d})$ be such that
\begin{equation}\label{e5.9B}
\frac {\omega _2(x-ty,\xi +(1-t)\eta )}{\omega _1
(x+(1-t)y,\xi -t\eta )} \le C \omega (x,\xi ,\eta ,y),
\end{equation}
for some constant $C>0$. If $a\in M^{p,q}_{(\omega )}(\rr
{2d})$, then $\op _t(a)$ from $\mathcal S_{1/2}(\rr d)$ to $\mathcal S_{1/2}'(\rr
d)$ extends uniquely to a continuous map from
$M^{p_1,q_1}_{(\omega _1)}(\rr d)$ to $M^{p_2,q_2}_{(\omega _2)}(\rr
d)$.

\par

Moreover, if in addition $a$ belongs to the closure of $\mathscr S_0(\rr {2d})$
under the norm $\nm \cdo {M ^{p,q}_{(\omega )}}$, then
$$
\op _t (a)\, :\, M^{p_1,q_1}_{(\omega _1)}(\rr d)\to
M^{p_2,q_2}_{(\omega _2)}(\rr d)
$$
is compact.
\end{thm}

\par

\begin{thm}\label{omvthm5.2}
Let $t\in \mathbf R$, $a\in \maclS _{1/2}'(\rr{2d})$, $\omega \in
\mascP_E(\rr {2d}\oplus \rr {2d})$, and $\omega _1,\omega _2 \in
\mascP_E(\rr {2d})$ such that \eqref{e5.9} holds. Then the following is true:
\begin{enumerate}
\item the operator $\op _t(a)$ from $\maclS _{1/2}(\rr d)$ to
$\maclS _{1/2}'(\rr d)$ extends to a continuous mapping from
$M^1_{(\omega _1)}(\rr d)$ to $M^\infty
_{(\omega _2)}(\rr d)$, if and only if $a\in M^\infty
_{(\omega )}(\rr {2d})$;

\vrum

\item the map $a\mapsto \op _t(a)$ from $M^\infty
_{(\omega )}(\rr {2d})$ to the set of linear and continuous operators from
$M^1_{(\omega _1)}(\rr d)$ to $M^\infty
_{(\omega _2)}(\rr d)$.
\end{enumerate}
\end{thm}

\par

Finally we consider Schatten-von Neumann properties. Let
$\omega _1,\omega _2\in \mascP _E(\rr {2d})$. Then the set
$s_{t,p}(\omega _1,\omega _2)$ consists of all $a\in \maclS
_{1/2}'(\rr {2d})$ such that $\op _t(a)$ belongs to $\mathscr
I_p(\omega _1,\omega _2)$, the set of
Schatten-von Neumann operator of order $p\in [1,\infty ]$ from $M^2_{(\omega _1)}(\rr d)$
to $M^2_{(\omega _2)}(\rr d)$.
%We also let $s_{t,\sharp }(\omega _1,\omega _2)$
%be the set of all $a\in \maclS
%_{1/2}'(\rr {2d})$ such that $\op _t(a)$  from $M^2_{(\omega _1)}(\rr d)$
%to $M^2_{(\omega _2)}(\rr d)$ is compact.
Note that
$$
\mathscr I_1(\omega _1,\omega _2),\quad \mathscr
I_2(\omega _1,\omega _2)\quad \text{and}\quad  \mathscr
I_\infty (\omega _1,\omega _2),
$$
are the sets of trace-class, Hilbert-Schmidt and continuous operators respectively, from
$M^2_{(\omega _1)}(\rr d)$ to $M^2_{(\omega _2)}(\rr d)$.
The space $s_{t,p}(\omega _1,\omega _2)$
is equipped by the norm
$$
\nm a{s_{t,p}(\omega _1,\omega _2)}\equiv \nm {\op _t(a)}{ \mathscr
I_p(\omega _1,\omega _2)}.
$$
By Theorem \ref{omvthm5.2} it follows that the map $a\mapsto \op _t(a)$ from
$s_{t,p}(\omega _1,\omega _2)$ to $\mathscr
I_p(\omega _1,\omega _2)$ is continuous and bijective.

\par

It is easy to obtain a complete characterization of symbols to Hilbert-Schmidt operators.
In fact, we have the following result. We refer to \cite[Prop4.11]{To8A} for the proof.

\par

\begin{prop}\label{HSweyl}
Let $a\in \maclS_{1/2}'(\rr {2d})$, $\omega _1,\omega _2\in \mascP _E(\rr {2d})$
and that $\omega \in \mathscr P_E(\rr {2d}\oplus \rr {2d})$ be such that equality
is attained in \eqref{e5.9} for $t=1/2$ and some constant $C$. Then
$\op ^w(a)\in \mathscr I_2(\omega _1,\omega
_2)$, if and only if $a\in M^2_{(\omega )}(\rr
{2d})$. Moreover, for some constant $C>0$ it holds
$$
C^{-1}\nm a{M^2_{(\omega )}} \le \nm {a}{s^w_2(\omega _1,\omega _2)}\le C\nm
{a}{M^2_{(\omega )}},
$$
for every $a\in \maclS _{1/2}'(\rr {2d})$.
\end{prop}

\par

We have now the following result.

\par

\begin{thm}\label{schattenthm}
Let $t\in \mathbf R$ and $p,q,p_j,q_j\in [1,\infty ]$ for
$j=1,2$, satisfy
\begin{equation*}%\label{pqolikhet}
p_1\le p \le p_2,\quad q_1\le \min (p,p')\quad \text{and}\quad q_2\ge
\max (p,p').
\end{equation*}
Also let $\omega \in \mathscr
P_E(\rr {2d}\oplus \rr {2d})$ and $\omega _1,\omega _2\in \mathscr P_E(\rr
{2d})$ be such that equality is attained in \eqref{e5.9B}, for some
constant $C$. Then
\begin{equation}\label{embedding1}
M^{p_1,q_1}_{(\omega )}(\rr{2d})\subseteq s_{t,p}(\omega
_1,\omega _2)\subseteq  M^{p_2,q_2}_{(\omega )}(\rr{2d})
\end{equation}
Moreover, for some constant $C>0$ it holds
$$
C^{-1}\nm a{M^{p_2,q_2}_{(\omega )}} \le \nm {a}{\mathscr
s_{t,p}(\omega _1,\omega _2)}\le C\nm
{a}{M^{p_1,q_1}_{(\omega )}}
$$
for every $a\in \maclS _{1/2}'(\rr {2d})$.
\end{thm}

\par

\begin{proof}
By Proposition \ref{p1.4} and Theorem \ref{omvthm5.2} it follows that $s_{t,\infty}\subseteq
M^{\infty}_{(\omega )}$, and by Theorem \ref{p5.4B} we get $M^{\infty ,1}_{(\omega )}
\subseteq s_{t,\infty}$. By duality we obtain $M^{1}_{(\omega )}
\subseteq s_{t,1}$ and $s_{t,1}\subseteq
M^{1,\infty}_{(\omega )}$ Furthermore, If $p_1=p_2=q_1=q_2=2$, then \eqref{embedding1} follows
from and Propositions \ref{p1.45} and \ref{HSweyl}. The result now follows for general $p$
by interpolating these cases. The proof is complete.
\end{proof}

\par

\subsection{A pseudo-differential calculus in the Bargmann-Fock setting}

\par

In this section we show some possibilities to establish a pseudo-differential
calculus on Banach spaces of analytic functions, in the frame-work of the
theory of the Bargmann transform. The definition of the calculus is in some
sense similar to the usual pseudo-differential calculus, defined in Section
\ref{sec1} (cf. \eqref{e0.5}). We show that usual partial differential operators
have convenient forms, and remark that the usual calculus in Section \ref{sec1}, to
some extent, can be considered as a part of this pseudo-differential
calculus on analytic functions.

\par

Before the definition of the calculus on analytic functions, we consider
properties of compositions of the Bargmann transform with the Fourier
transform, translations or modulations. It is then convenient to
introduce some notations.

\par

The \emph{Fourier-Bargmann} transform $\mathscr F_{\mathfrak V,t}$ of any 
function or distribution $F$ on $\cc d$ of order $t\in \mathbf R$ is given by
$$
(\mathscr F_{\mathfrak V,t}F)(z) =F(e^{-it\pi /2}z).
$$
We also set $\mathscr F_{\mathfrak V}=\mathscr F_{\mathfrak V,1}$, and call this
map the Fourier-Bargmann transform. We note that $\mathscr F_{\mathfrak V,t}^{-1}=\mathscr F_{\mathfrak V,-t}$, and that $(\mathscr F_{\mathfrak V}^{-1}F)(z)=F(iz)$. The following lemma shows that the latter formula is strongly related to Fourier's inversion formula.

\par

\begin{lemma}\label{lemmaFourBarg}
Let $\omega \in \mascP _Q^0(\cc d)$, $\mascB$ be a mixed norm space on $\cc d$, and set $\omega _t(z)=\omega (e^{-it\pi /2}z) = (\mathscr F_{\mathfrak V,t}\omega )(z)$. Then the following is true:
\begin{enumerate}
\item $\mathscr F_{\mathfrak V,t}$ restricts to continuous bijective mappings from $B(\omega ,\mascB )$ to $B(\omega _t,\mascB )$, and $A(\omega ,\mascB )$ to $A(\omega _t,\mascB )$;

\vrum

\item $\mathfrak V\circ \mathscr F$ is equal to $\mathscr F_{\mathfrak V}\circ \mathfrak V$ as mappings from $M(\omega ,\mascB )$ to $A(\omega _1,\mascB )$;

\vrum

\item if $f\in M(\omega ,\mascB )$, then
\begin{align*}
(\mathfrak V(f(\cdo -x/\sqrt 2 )))(z) &= e^{\scal zx -|x|^2/4}(\mathfrak Vf)(z-x),
\\[1ex]
(\mathfrak V(e^{i\sqrt 2\scal \cdo \xi}f))(z) &= (\mathfrak Vf)(z+i\xi ).
\end{align*}
\end{enumerate}
\end{lemma}

\par

We note that (2) and (3) in Lemma \ref{lemmaFourBarg} in some special cases were
proved already in \cite{B1,FGW,Gc1,GL}.

\par

\begin{proof}
The assertions (1) and (3) follows immediately from the definitions, and (2) follows
by a straight-forward application of Fourier's inversion formula. The details are left
for the reader.
\end{proof}

\par

By Lemma \ref{lemmaFourBarg} and the investigations in Section
\ref{sec1}, it follows that $e^{i\scal x\xi}$, $\mathscr F$ and $dx$
in \eqref{e0.5} concerning the usual pseudo-differential calculus
correspond to $e^{i(z,w)}$, $\mathscr F_{\mathfrak V}$ and
$d\mu (z)$ respectively. The following definition of our complex
version of pseudo-differential operators, is based on these
observations.

\par

\begin{defn}\label{defcomplpseudo}
Let $t\in \mathbf R$ and let $a\in (\maclS _{1/2})'(\cc d\oplus \cc d)$ be such that
\begin{enumerate}
\item $a(z,w)e^{-|z|^2-|w|^2+N(\eabs z +\eabs w)}\in \mathscr S '(\cc d\oplus \cc d)$
for every $N\ge 0$;

\vrum

\item if $p\in P(\cc d)$, then $z\mapsto \scal {a(z,i\cdo )}{e^{-|\cdo |^2+(z,\cdo )}p}$ is entire.
\end{enumerate}
Then the (complex) pseudo-differential operator $\op _{\mathfrak V,t}(a)$
with respect to the symbol $a$ is the linear operator from $P(\cc d)$ to $A(\cc d)$, given by
\begin{multline}\label{complpseudodef}
(\op _{\mathfrak V,t}(a)F)(z)
\\[1ex]
= \iint a((1-t)z+tw_1,w_2)F(w_1)
e^{i((z,w_2)-(w_2,w_1))}\, d\mu (w_1)d\mu (w_2)
\\[1ex]
= \iint a((1-t)z+tw_1,iw_2)F(w_1)
e^{((z,w_2)+(w_2,w_1))}\, d\mu (w_1)d\mu (w_2),
\end{multline}
when $F\in P(\cc d)$.
\end{defn}

\par

We note that the reproducing kernel in combination with the fact that
$w_1\mapsto a((1-t)z+tw_1,w_2)F(w_1)$ in \eqref{complpseudodef}
is analytic and satisfying appropriate conditions give
$$
(\op _{\mathfrak V,t}(a)F)(z)
= \int a((1-t)z+tw,w)F(w)
e^{i(z,w)}\, d\mu (w).
$$
If
$$
(w_1,w_2)\mapsto a((1-t)z+tw_1,w_2)F(w_1) e^{i((z,w_2)-(w_2,w_1))}
$$
in \eqref{complpseudodef} is not an integrable function, then $\op _{\mathfrak V,t}(a)$
is defined as the operator with kernel
$$
(z,w)\mapsto \pi ^{-d}\Pi _A(a((1-t)z+tw,i\cdo )
e^{((z,\cdo )-(\cdo ,w))})e^{-|w|^2}.
$$

\par

For conveniency we also set $\op _{\mathfrak V} = \op _{\mathfrak V,0}$.
%In
%this case we may relax the condition (1) in Definition \ref{defcomplpseudo} into
%$$
%\fy (z)a(z,w)e^{-|w|^2+N\eabs w}\in \mathscr S'(\cc d\oplus \cc d),
%$$
%for every $N\ge 0$ and $\fy \in C_0^\infty (\cc d)$.

\par

The following proposition gives motivations for considering operators
of the form $\op _{\mathfrak V,t}(a)$.

\par

\begin{prop}\label{oppolsymb}
Let $N\ge 0$ be an integer, $a_\beta \in A(\cc d)$ for every
$\beta \in \mathbf N^d$ such that $|\beta |\le N$, and let
$$
a(z,w) = \sum _{|\beta |\le N}a_\beta (z)\overline w^\beta .
$$
Then
$$
(\op _{\mathfrak V}(a)F)(z) = \sum _{|\beta |\le N}a_\beta (z) (D^\beta F)(z),
\qquad F\in P(\cc d).
$$
\end{prop}

\par

\begin{proof}
The result follows by straight-forward computations, using Remark \ref{derivrepr}
\end{proof}

\par

\begin{rem}
We may use Lemma \ref{lemmaFourBarg} and the mapping properties for the Bargmann transform to reformulate certain pseudo-differential operators of the form $\op _t(a)$ into pseudo-differential operators given by Definition \ref{defcomplpseudo}. The details are left for the reader.
\end{rem}

\par

\begin{rem}
If $a(z,w) =(S^{-1}b)(w/i)$, then it follows by the definitions that $\op _{\mathfrak V,t}(a) = \operatorname T_{\mathfrak V}(b)$. Hence the set of Berezin-Toeplitz operators can be considered as a subclass of the Bargmann pseudo-differential operators.
\end{rem}

\par

\begin{rem}
Let $a$ fulfills the conditions in Definition \ref{defcomplpseudo}, and assume in addition that $w\mapsto a(z,w)$ is analytic. Then it follows by the reproducing formula that
$$
(\op _{\mathfrak V,t}(a)F)(z) = a(z,z)F(z),
$$
when $F\in P(\cc d)$.
\end{rem}

\medspace

\end{document}